\documentclass[11pt,oneside]{amsart}
\usepackage{fouriernc} 

\usepackage{gjotex/Environments}   
\usepackage{gjotex/Definitions}    
\usepackage{gjotex/PageSetup}      
\usepackage{gjotex/TikzSetup}      

\newcommand{\inob}{{\textsc{i}}}
\newcommand{\tmob}{\textsc{t}}

\usepackage{calc}
\newlength{\argwidth}
\newlength{\argheight}
\newlength{\accentshiftL}
\newlength{\accentshiftR}
\newlength{\accentwidth}
\newcommand\nh[1]{%
  \settowidth{\argwidth}{$#1$}%
  \settoheight{\argheight}{$#1$}%
  \setlength{\accentshiftL}{-\argwidth/2-\accentwidth/2}%
  \setlength{\accentshiftR}{\argwidth/2-\accentwidth/2}%
  \ifdim\argwidth<\accentwidth
  \setlength{\accentshiftR}{0pt}\fi%
  #1\hspace{\accentshiftL}\raisebox{\argheight+4.4pt}{\rotatebox{-90}{$($}}\hspace{\accentshiftR}
}

\usepackage{xparse}

\newcommand{\sint}{{\textstyle\int}} 
\newcommand{\coef}{\,;\,} 

\newcommand{\inv}{{-1}}
\newcommand{\<}{\langle}
\renewcommand{\>}{\rangle}  
\newcommand{\Anglearrow}[1]{\rotatebox{#1}{$\Rightarrow$}}
\renewcommand{\id}{1}
\renewcommand{\Id}{1}

\DeclareMathOperator{\pb}{\mathrm{pb}}
\newcommand{\TO}{\Rightarrow}

\newcommand{\dar}{\downarrow}

\newcommand{\internalLaxComma}[2]{#1 \dar #2}
\NewDocumentCommand \laco
      { > { \SplitArgument { 1 } { , } } m }
      { \internalLaxComma #1 }

\newcommand{\opdar}{\dar^\op\hspace{.5pt}}
\newcommand{\internalOpLaxComma}[2]{#1 \opdar #2}
\NewDocumentCommand \oplaco
      { > { \SplitArgument { 1 } { , } } m }
      { \internalOpLaxComma #1 }

\newcommand{\internalLDAR}[2]{\De #1 \dar \nh{#2}}
\NewDocumentCommand \ldar
      { > { \SplitArgument { 1 } { , } } m }
      { \internalLDAR #1 }

\newcommand{\internalOPRDAR}[2]{\nh{#1} \opdar \De \Id_{#2}}
\NewDocumentCommand \oprdar
      { > { \SplitArgument { 1 } { , } } m }
      { \internalOPRDAR #1 }

\newcommand{\bsimp}{B}

\hypersetup{
  pdfkeywords={
    opfibration,
    Quillen's Theorem B,
    group completion,
    symmetric monoidal 2-category
  },
}

\makeatletter
\@namedef{subjclassname@2020}{\textup{2020} Mathematics Subject Classification}
\makeatother

\usepackage{gjotex/AuthorInfo}     

\title{2-categorical opfibrations, Quillen's Theorem B, and $S^{-1}S$}
\authorGurski
\authorJohnson
\authorOsorno
\date{May 2021}

\hypersetup{pdfauthor=\authors}

\begin{document}

\begin{abstract}
  In this paper we show that the strict and lax pullbacks of a
2-categorical opfibration along an arbitrary 2-functor are homotopy equivalent.  We give two
applications.  First, we show that the strict fibers of an opfibration
model the homotopy fibers.  This is a version of Quillen's Theorem B
amenable to applications.  Second, we compute the $E^2$ page of a
homology spectral sequence associated to an opfibration and apply this
machinery to a 2-categorical construction of $S^{-1}S$.  We show that
if $S$ is a symmetric monoidal 2-groupoid with faithful translations
then $S^{-1}S$ models the group completion of $S$.

\end{abstract}

\maketitle

\tableofcontents

\section{Introduction}

Fibrations and opfibrations of 1-categories, also known as
Grothendieck fibrations and opfibrations (originally fibered and cofibered categories), were introduced in
\cite[Exp. VI]{sga1} for the application of categorical algebra to
descent problems.  If $P\cn C \to D$ is a fibration of 1-categories,
then there is a natural adjunction between $P^\inv(x)$ and the comma
category $x \dar P$.  This makes $P^\inv$ a contravariant
pseudofunctor from $C$ to $\Cat$ and moreover gives a homotopy
equivalence between the classifying spaces of~$P^\inv(x)$ and~$x \dar
P$.  Dually, if $P$ is an opfibration, then $P^\inv$ is a covariant
pseudofunctor to $\Cat$ and there is a homotopy equivalence on
classifying spaces between $P^\inv(x)$ and $P \dar x$.

In this paper we study both homotopical and categorical features of 
opfibrations for 2-categories and lax comma objects for 2-functors.  The former are dual to the fibrations of Buckley \cite{Buc14Fibred}, and the latter are due to Gray \cite{Gray80Closed,Gray80Existence}.
In \cref{sec:theorem_B} we
recall the relevant definitions and prove the following result as
\cref{prop:fiber_vs_htpyfiber}.
\begin{thm}\label{thm:intro-fiber_vs_htpyfiber}
  Suppose
  \[
  C \fto{P} D \xleftarrow{F} B
  \]
  is a cospan of 2-categories and all 2-cells of $D$ are invertible.
  If $P$ is an opfibration then the inclusion
  \[
  i\cn \pb(P,F) \to  \laco{P,F},
  \]
  from the pullback of $P$ and $F$ to the lax comma object $\laco{P,F}$,
  induces a homotopy equivalence on classifying spaces.
\end{thm}
\noindent We go on to give two applications of this result, outlined in
\cref{sec:intro-quillen-b,sec:intro-s-inv-s} below.

\subsection{Quillen's Theorem B for opfibrations of 2-categories}\label{sec:intro-quillen-b}

Our first application is related to Quillen's Theorems A and B
\cite{Qui1973Higher}. Theorem A states that under suitable conditions,
if the comma objects of a functor between categories are assumed to be
contractible, then the functor induces a homotopy equivalence on
classifying spaces.  It can be seen as a special case of Quillen's
Theorem B, which says that the comma objects of a functor model its
homotopy fibers.

Versions of Theorems A and/or B for higher categories have
appeared in many papers. For Theorem A, this includes Bullejos-Cegarra
\cite{BC03Geometry} (Theorem A for 2-categories) and Ara-Maltsiniotis \cite{AM18QuillenAI,AM20QuillenAII}
(Theorem A for strict $\infty$-categories). For Theorem B, this includes 
Cegarra \cite{Ceg11Homotopy} (Theorem B for 2-categories),
Calvo-Cegarra-Heredia \cite{CCH14Bicategorical} (Theorem B for
bicategories), 
 and Ara
\cite{Ara19QuillenB} (Theorem B for strict $\infty$-categories).

Combining the 2-categorical version of Quillen's Theorem B (see \cref{thm:cegarra} below) 
with \cref{thm:intro-fiber_vs_htpyfiber} gives the following form of
Quillen's Theorem B, which appears as
\cref{cor:quillenB-opfib-2}. This is a form that more closely
resembles the version used in $K$-theoretic applications.

\begin{cor}\label{cor:quillenB-opfib}
  Assume that all 2-cells of $D$ are invertible, and suppose $P\cn C \to D$
   is an opfibration such that, for every 1-cell $x
  \to y$ in $D$ the 2-functors $P^\inv(x) \to P^\inv(y)$ induced by
  base change are homotopy equivalences on classifying spaces.  Then
  for each $x \in C$ the sequence
  \[
  P^\inv(x) \to C \to D
  \]
  induces a homotopy fiber sequence on classifying spaces.
\end{cor}

Chiche \cite{chiche} defines a notion of pre-(coop)fibration for which the conclusion of \cref{thm:intro-fiber_vs_htpyfiber} holds in the special case in which $P$ is the inclusion of a single object.  However the comparison between our opfibrations and the definitions of Chiche is nontrivial and beyond the scope of this work.

Two-dimensional (op)fibrations in the sense we study here have appeared in work of
Hermida \cite{Her99Some} (for 2-categories), Bakovic \cite{bakovicFib}
(for bicategories), and Buckley \cite{Buc14Fibred} (with additional
results for both 2-categories and bicategories).  However, a result
like \cref{thm:intro-fiber_vs_htpyfiber} has not previously appeared, and
therefore nor has \cref{cor:quillenB-opfib}.  Identifying the
conditions that
\begin{enumerate}
  \renewcommand{\labelenumi}{(\textit{\alph{enumi}})}
  \item allow one to model homotopy fibers via strict
    fibers and
  \item occur in applications of interest
\end{enumerate}
was a major motivation for all of the results in this paper.

Striking this balance allows us to analyze the homology spectral
sequence of an opfibration in \cref{sec:spectral-sequence}.  We apply
the theory of lax comma objects and opfibrations to give an explicit
description of the $E^2$ page in the following result, which appears as \cref{thm:ss}.
\begin{thm}\label{thm:intro-ss}
  If $P \cn C \to D$ is an opfibration and all 2-cells of $D$ are
  invertible, then there is a spectral sequence
 \[
  E^2_{p,q} = H_p\big( D \coef \sH_q P^\inv \big) \Rightarrow H_{p+q} (C)
 \]
 where $\sH P^\inv$ denotes the local coefficient system given by
 homology of the fibers of $P$.
\end{thm}
\noindent The development of the homology spectral sequence for a
2-functor in \cref{sec:spectral-sequence} generalizes the
corresponding theory for 1-categorical opfibrations, and might be
conceptually familiar.


\subsection{Group completion and the $S^\inv S$ construction}\label{sec:intro-s-inv-s}

Our second application relates to Quillen's categorical model for topological group
completion. The primary difference between an arbitrary $E_{\infty}$ space and the
zeroth space of an $\Om$-spectrum occurs at the level of path components.
The action of an $E_{\infty}$ operad on a space $X$ induces a multiplication on 
$\pi_0 X$, and the fact that the operad is $E_{\infty}$ equips $\pi_0 X$
 with the structure of a commutative
monoid. If $Y = Y_0$ is the zeroth space of an $\Om$-spectrum $\{Y_n\}$, then
the isomorphism of monoids $\pi_0 Y \cong \pi_n Y_n$ for any $n > 1$ shows that $\pi_0 Y$ is actually an
abelian group.  We say that an $E_{\infty}$
space $X$ is \emph{grouplike} if $\pi_0X$ is a group.
\emph{Topological group completion} \cite{Bar1961Note,BP1972Homology,May1974Einfty,Seg74Categories,McDS75Homology} is a process that 
\begin{itemize}
\item universally completes the commutative monoid
$\pi_0 X$ to an abelian group,
\item similarly completes the homology of $X$,
\item but in the process radically changes the higher homotopy groups of $X$.
\end{itemize}


For a symmetric monoidal category $S$ with invertible morphisms and
faithful translations, Quillen's $S^\inv S$ construction
\cite{Gra1976Higher} provides a categorical model for the
topological group completion of the classifying space $BS = |NS|$.  In the case
that $S$ is a commutative monoid treated as a discrete symmetric
monoidal category, $S^\inv S$ is equivalent (as a 
symmetric monoidal category) to the algebraic group completion 
$S^{\textrm{gp}}$---the target of a universal monoid homomorphism
from $S$ into an abelian group. In this sense, topological group 
completion generalizes algebraic group completion.

The 1-categorical $S^\inv S$ construction is crucial for the
definition and study of higher algebraic $K$-theory.  In the case that
$R$ is a commutative ring and $S = \coprod_n GL_n(R)$, then the main
result of \cite{Gra1976Higher} provides a homotopy equivalence
\begin{equation}\label{eq:SinvS-eq-plus3}
B(S^\inv S) \fto{\hty} K_0(R) \times BGL(R)^+.
\end{equation}
In the low-dimensional cases $n=1,2$, the categorical structure of
$S^\inv S$ provides direct algebraic access to the $K$-theory groups
$K_n(R) = \pi_n BGL(R)^+$ for $n>0$.

As our second application, we obtain \cref{thm:main} below,
generalizing the construction and results of \cite{Gra1976Higher} to
the case in which $S$ is a symmetric monoidal bicategory.  We assume
without loss of generality that $S$ is in a strict form known as a
permutative Gray monoid.  In \cref{sec:pgms} we recall this notion and
the result that each symmetric monoidal bicategory is equivalent, via
a symmetric monoidal pseudofunctor, to a permutative Gray monoid
(\cref{thm:sm2cat-equiv-to-pgm}).  For simplicity we state the main
result in these terms.
\begin{thm}\label{thm:main}
  Let $S$ be a permutative Gray monoid.  There is a symmetric
  monoidal 2-category $S^\inv S$ and a symmetric monoidal
  2-functor
  \[
  i\cn S \to S^\inv S
  \]
  with the following properties.
  \begin{enumerate}
  \item\label{main:i}  The classifying space of $S^\inv S$ is
    grouplike.
  \item\label{main:ii} If $S$ is, moreover, a 2-groupoid with faithful
    translations, then $i$ is a group completion on classifying
    spaces.
  \end{enumerate}
\end{thm}
\noindent Claim \eqref{main:i} is proved in \cref{prop:S-inv-S-pgm}
by direct analysis of the construction $S^\inv S$.  Claim
\eqref{main:ii} follows by applying \cref{thm:intro-ss} to the
opfibration
\[
S^\inv S \to S^\inv *.
\]

\begin{warn}
 As it is an easy source of confusion, we emphasize that our construction concerns symmetric monoidal bicategories and not symmetric bimonoidal categories. The latter are 1-dimensional categories with two symmetric monoidal structures, one of which distributes over the other. As pointed out by Thomason \cite{Tho80Phony}, the classical $S^\inv S$ construction on a symmetric bimonoidal category fails to have a bimonoidal structure.
\end{warn}

\subsection{Relation to other forms of group-completion}
We relate \cref{thm:main} back to the program in
\cite{GJO2019dimensional}.  Recall that a stable $P_2$-equivalence is
a symmetric monoidal functor inducing an isomorphism on stable
homotopy groups in dimensions 0, 1, and 2.  A Picard 2-category
(\cref{defn:picard}) is a symmetric monoidal 2-category in which all
objects and morphisms are invertible in the weak sense.  In
\cite{GJO2019dimensional} we proved that every symmetric monoidal
bicategory is stably $P_2$-equivalent to a strict Picard
2-category. This strict Picard 2-category is a categorical model for
the Postnikov 2-truncation of the spectrum associated to the original
symmetric monoidal bicategory. The construction in
\cite{GJO2019dimensional} involves a long zigzag of stable
$P_2$-equivalences, and employs a topological group completion of
$E_{\infty}$-spaces. One outcome of the present paper is an alternate
construction of the 2-truncation in simpler and purely categorical
terms.  The caveat is that such a construction is only available under
the additional assumption of faithful translations. We do not know a
general method for imposing the condition of faithful translations on
a general symmetric monoidal category or 2-category while preserving
its stable homotopy type, other than the full machinery developed in
\cite{GJO2019dimensional}.

\begin{oprob}
  Given a symmetric monoidal (2-)category $M$, find a purely algebraic
  functorial construction of a symmetric monoidal (2-)category
  $M^{ft}$ with faithful translations together with a (perhaps only
  pseudo or lax) natural stable equivalence $M \to M^{ft}.$
\end{oprob}

In the presence of faithful translations, we can use the results in
this paper and \cite{GJO2019dimensional} to construct an explicit
2-truncation as follows. If $S$ is a permutative Gray monoid which is
a 2-groupoid with faithful translations, then $S \to S^\inv S$ is a
homotopy group completion by \cref{thm:main}. Section 5.1 of
\cite{GJO2019dimensional} defines a 2-groupoidification functor $WN$
such that $WN(S^\inv S)$ is a symmetric monoidal 2-groupoid which
receives a natural stable $P_2$-equivalence $ S^\inv S \to WN(S^\inv
S)$.  We note that $WN(S^\inv S)$ is a symmetric monoidal 2-category
but generally not a permutative Gray monoid even when $S$ is assumed
to be so.  Nevertheless, the composite $S \to S^\inv S \to WN(S^\inv
S)$ is a stable $P_2$-equivalence, and $WN(S^\inv S)$ is a Picard
2-category by construction.  Therefore the classifying space of
$WN(S^\inv S)$ is the stable Postnikov 2-truncation of $BS$.

We point out that there is another, purely algebraic, notion of group
completion one might study for symmetric monoidal categories or
2-categories.  We want to distinguish between the fundamentally
homotopic notion and the purely algebraic one. The algebraic group
completion can be constructed formally via factorization system
theory. The category of groups sits as a reflective subcategory of the
category of monoids, and the group completion of a monoid is given by
applying the left adjoint to this inclusion. This left adjoint can be
seen as an application of the small object argument for the weak
factorization system cofibrantly generated by the single map
$\mathbb{N} \to \mathbb{Z}$, and so the group completion of a monoid
is then the fibrant replacement for this factorization system.

The algebraic group completion for symmetric monoidal 1-categories,
respectively 2-categories, is given by applying an analogous left
adjoint.  It is more complicated to write down explicitly and even
less computable. Such a completion takes values in Picard categories,
respectively Picard 2-categories.  From the perspective of algebraic
weak factorization systems \cite{GT2006Natural, BG2016AlgebraicI,
  BG2016AlgebraicII} this construction must exist and is given, in
dimension two, by finding generating maps $i_0, i_1, i_2$ such that
lifting against the unique map to the terminal 2-category with respect
to $i_{\al}$ exhibits all $\al$-dimensional cells as invertible. The
codomain of $i_0$ is then the free Picard 2-category generated by a
single object, and the authors will give an explicit model for this
object in a future paper. In contrast, the group completion $S^\inv S$
we study below is defined in a way that is unexpected from a purely
categorical perspective, but nevertheless models the topological group
completion and requires less intense 2-categorical machinery to
construct.

\addtocontents{toc}{\SkipTocEntry}
\subsection*{Outline}
The rest of this paper is organized as follows.  In
\cref{sec:theorem_B} we provide 2-categorical background, define what
it means for a 2-functor to be an opfibration, and establish \cref{thm:intro-fiber_vs_htpyfiber} and
\cref{cor:quillenB-opfib}.
\cref{sec:spectral-sequence} contains the construction of a spectral
sequence for an opfibration, and we analyze the local coefficients
arising from the lax comma object and strict pullback.  In
\cref{sec:SinvS} we give the necessary background on symmetric
monoidal 2-categories and construct the 2-categorical version of
$S^\inv S$. We then prove part \eqref{main:ii} of \cref{thm:main} via
a collapsing spectral sequence argument generalizing that of
\cite{Gra1976Higher}.

\section{Opfibrations and Quillen's Theorem B}\label{sec:theorem_B}

This section has two primary goals: to introduce the 2-categorical
notion of opfibration, and to study its homotopy-theoretic
properties. We begin with 2-categorical background in
\cref{sec:conventions} and then in \cref{sec:opfibrations} give the
definition of opfibration.  We define the lax and oplax comma objects
for a cospan of 2-functors in \cref{sec:comma-objs}.
\Cref{sec:comparison} is the technical heart of this paper; its main
result is \cref{prop:fiber_vs_htpyfiber}, which compares the strict
pullback and lax comma object for an opfibration.  We apply
\cref{prop:fiber_vs_htpyfiber} at the end of \cref{sec:comparison} to
prove \cref{cor:quillenB-opfib-2}: the strict fibers of an opfibration
model the homotopy fibers on classifying spaces.  This is Quillen's
Theorem B for opfibrations.  Looking ahead, we will also apply
\cref{prop:fiber_vs_htpyfiber} in \cref{sec:spectral-sequence} to
identify the $E^2$ page of the homology spectral sequence for an
opfibration.  That, in turn, is essential for our study of $S^\inv S$
in \cref{sec:SinvS}.

\subsection{2-categorical definitions and conventions}\label{sec:conventions}
To fix terminology and notation, we give the following overview of our
conventions.  We refer the reader to
\cite{benabou,Lac10Companion,johnson-yau} for further background on
2-categories and bicategories.
\begin{itemize}
\item We write $\iicat$ for the category of 2-categories and 2-functors.
\item We write $[n]$ for the poset $0 \leq 1 \leq \cdots \leq n$
  treated as a category or locally discrete 2-category.
\item We write $e$ for the unit object of a monoidal
  2-category, use 1 for identity 1-cells and 2-cells, and
  omit subscripts on identities unless they add clarity.
\item We usually denote the composition of 1-cells and vertical
  composition of 2-cells by juxtaposition, but use $\circ$ when
  necessary for either readability or clarity.
    
\item We write the horizontal composition of appropriately composable
  2-cells as $\al'*\al$.  If either $\al$ or $\al'$ is an identity
  2-cell we write
  \[
  \al' * f = \al' * 1_f \mathrm{\ and\ } g * \al = 1_g * \al
  \]
  for the whiskerings.
\end{itemize}

\begin{defn}\label{defn:functor}
  Let $X, Y$ be 2-categories. A \emph{lax functor} $F \cn X \to Y$ consists of 
  \begin{itemize}
  \item a function on objects $F \cn \ob X \to \ob Y$,
  \item functors $X(x, x') \to Y(Fx, Fx')$ for all objects $x,x'\in X$,  
  \item 2-cells $F_2 \cn Fg \circ Ff \Rightarrow F(g \circ f)$ for all
    composable pairs of 1-cells $g$ and $f$, natural in both
    arguments, and
  \item 2-cells $F_0 \cn 1_{Fx} \Rightarrow F(1_x)$ for all objects $x \in X$.
  \end{itemize}
  These are subject to three axioms, one for composable triples of
  1-cells and two for composing a 1-cell with an identity on either
  side.

  An \emph{oplax functor} $F \cn X \to Y$ consists of 
  \begin{itemize}
  \item a function on objects $F \cn \ob X \to \ob Y$,
  \item functors $X(x, x') \to Y(Fx, Fx')$ for all objects $x,x'\in X$,  
  \item 2-cells $F_2 \cn F(g \circ f) \Rightarrow Fg \circ Ff$ for all
    composable pairs of 1-cells $g$ and $f$, natural in both
    arguments, and
  \item 2-cells $F_0 \cn F(1_x) \Rightarrow 1_{Fx}$ for all objects $x \in X$.
  \end{itemize}
  These are subject to three analogous axioms.

  We say that a lax or oplax functor $F$ is a \emph{pseudofunctor} if
  the structure 2-cells $F_2, F_0$ are all isomorphisms.  We say that
  a lax or oplax functor $F$ is \emph{normal} if the 2-cells $F_0$ are
  identities for all $x$. We note that for lax functors between 2-categories, the condition that $F_0$
  is the identity implies that $F_2$ is also the identity when either $f$ or $g$ is 
  an identity 1-cell, but should be required separately for a lax functor between bicategories.
\end{defn}

\begin{defn}\label{defn:lax_trans}
  Given 2-categories $X$ and $Y$ together with a pair of 2-functors $F, G \cn X
  \to Y$, a \emph{lax transformation} $\al \cn F \Rightarrow G$,
  consists of
  \begin{itemize}
  \item 1-cells $\al_x \cn Fx \to Gx$ for all objects $x \in X$ and
  \item 2-cells $\al_f \cn Gf \circ \al_x \Rightarrow \al_y \circ Ff$
    for all 1-cells $f \cn x \to y$ in $X$.
  \end{itemize}
  These are subject to axioms stating that $\al_f$ is natural in
  2-cells $\de \cn f \Rightarrow g$, that $\al_1$ is the identity
  2-cell, and that $\al_{gf}$ is the appropriate pasting of $\al_f$
  and $\al_g$.
  
  A \emph{pseudonatural transformation} $\al$ is a lax transformation
  such that $\al_f$ is an invertible 2-cell for all $f$.  We say that
  $\al$ is \emph{2-natural} if $\al_f$ is an identity 2-cell for all
  $f$.  An \emph{oplax transformation} $\al \cn F \Rightarrow G$,
  where $F, G \cn X \to Y$ are 2-functors, consists of
  \begin{itemize}
  \item 1-cells $\al_x \cn Fx \to Gx$ for all objects $x \in X$ and
  \item 2-cells $\al_f \cn \al_y \circ Ff \Rightarrow Gf \circ \al_x$
  \end{itemize}
  subject to analogous axioms.
  \end{defn}
  
\begin{defn}\label{defn:modif}
  A \emph{modification} $\Ga \cn \al \Rrightarrow \al'$, where $\al,
  \al'$ are either both lax or oplax transformations with common
  source and target, consists of 2-cells $\Ga_x \cn \al_x \Rightarrow
  \al'_x$ for all objects $x \in X$ subject to one axiom stating that
  the cells $\Ga_x, \Ga_y$ are compatible with $\al_f, \al'_f$ for all
  1-cells $f \cn x \to y$.
\end{defn}

\begin{notn}\label{lax_oplax_homs}
  Let $X$ and $Y$ be 2-categories.
  \begin{enumerate}
  \item We write $[X, Y]$ to denote the 2-category of 2-functors,
    2-natural transformations, and modifications from $X$ to $Y$.
  \item We write $\lax(X, Y)$
    to denote the 2-category of 2-functors, lax
    transformations and modifications from $X$ to $Y$.
  \item We write $\oplax(X, Y)$
    to denote the 2-category of 2-functors, oplax
    transformations and modifications from $X$ to $Y$. 
  \end{enumerate}
\end{notn}

\begin{defn}\label{defn:nopl-nerve}
  For a 2-category $X$ we let $NX$ denote the normal oplax nerve: its
  $p$-simplices are the normal oplax functors $[p] \to X$, with $[p]$
  regarded as a locally discrete 2-category.  We call $NX$ the
  \emph{nerve} of $X$, and its geometric realization $|NX|$ the
  \emph{classifying space} of $X$.
\end{defn}
This is one of several homotopy equivalent nerves for 2-categories described in
\cite{CCG10Nerves}.  In \cref{sec:spectral-sequence} we will give an
alternate equivalent description of $NX$ using the oriented simplices
of Street \cite{Str87Algebra}.  We make repeated use of the following
result
\begin{thm}[\cite{CCG10Nerves}]\label{thm:trans-htpy}
  The geometric realization of a lax or oplax natural transformation
  is a homotopy between the realizations of its source and target
  2-functors.
\end{thm}

\subsection{Opfibrations}\label{sec:opfibrations}
In this section we give the definition of an opfibration between 2-categories. Our
definition is essentially the same as the notion of fibration between
bicategories developed by Buckley \cite{Buc14Fibred}, but in the
opfibrational form and in the special case that the bicategories and
pseudofunctors involved are in fact 2-categories and
2-functors. Making those changes yields the following definitions.

\begin{defn}\label{defn:cart-2}
 Let $P\cn K \to L$ be a 2-functor. For 1-cells $g,h\cn x \to y$ in
 $K$, we say that a 2-cell $\ga\cn g \TO h$ is \emph{cartesian} with
 respect to $P$ if it is a cartesian morphism in the category
 $K(x,y)$, i.e., for all 1-cells $k\cn x\to y$ in $K$, the following
 square is a pullback.
 \[
 \begin{tikzpicture}[x=40mm,y=20mm]
   \draw[tikzob,mm] 
   (0,0) node (a) { K(x,y)(k,g)}
   (1,0) node (b) {L(Px, Py)(Pk,Pg)}
   (0,-1) node (d)  { K(x,y)(k,h)}
   (1,-1) node (c) {L(Px,Py)(Pk,Ph)};
   \path[tikzar,mm] 
   (a) edge[swap] node {\ga_{*}} (d)
   (d) edge node[swap] {P} (c)
   (a) edge node {P} (b)
   (b) edge node {P\ga_{*}} (c);
 \end{tikzpicture}
 \]
 Explicitly, this means that for each $\psi\cn
 k\TO h$ and $\phi\cn Pk \TO Pg$ such that $P\psi = (P\ga) \circ
 \phi$, there is a unique 2-cell $\phi^c\cn k \TO g$ such that $P\phi^c =
 \phi$ and $\ga \circ \phi^c = \psi$.  
\end{defn}
\begin{defn}\label{defn:loc-fib}
  We say that a 2-functor $P\cn K \to L$ is a \emph{local fibration}
  if the induced functor on hom-categories
  \[
  P\cn K(x,y) \to L(Px,Py)
  \]
  is a fibration for each $x,y \in K$.  That is, each 2-cell $\al\cn f
  \Rightarrow P(g)$ in $L(Px,Py)$ has a cartesian lift $\wh\al\cn
  \wh{g} \Rightarrow g$ in $K(x,y)$, with $P \wh\al=\al$.
\end{defn}

\begin{defn}\label{defn:opcart}
  We say that a 1-cell $h\cn a \to d$ in $K$ is \emph{opcartesian}
  with respect to $P$ if the following two conditions hold.
  \begin{enumerate}
  \item\label{it:opcart-1} Suppose given a 1-cell $u \cn a \to b$ in $
    K$, a 1-cell $t \cn Pd \to Pb$, and an 2-cell isomorphism $\al$ as
    shown below.
    \[
    \begin{tikzpicture}[x=20mm,y=20mm]
      \draw[tikzob,mm] 
      (0,0) node (a) {Pa}
      (0,-1) node (d)  {Pd}
      (1,-1) node (c) {Pb};
      \path[tikzar,mm] 
      (a) edge[swap] node {Ph} (d)
      (a) edge node {Pu} (c)
      (d) edge[swap] node {t} (c);
      \draw[tikzob,mm]
      (.35,-.7) node[rotate=45,2label={above,\al}] {\Rightarrow}
      ;
    \end{tikzpicture}
    \]
    Then there exists a 1-cell $\tilde{t} \cn d \to b$ in $ K$ and 2-cell isomorphisms 
    \[
    \al_1 \cn t \cong P\tilde{t}, \quad \al_2 \cn \tilde{t} h \cong u\]
    such that the following equality holds.
    \[
    \begin{tikzpicture}[x=30mm,y=20mm]
      \draw[tikzob,mm] 
      (0,0) node (a) {Pa}
      (0,-1) node (d)  {Pd}
      (1,-1) node (c) {Pb}
      (.35,-.7) node[rotate=45,2label={above,\al}] {\Rightarrow}
      ;
      \path[tikzar,mm] 
      (a) edge[swap] node {Ph} (d)
      (a) edge[bend left=30] node {Pu} (c)
      (d) edge[swap, bend right=30] node {t} (c);
      \draw (1.4,-.5) node {=};
      \begin{scope}[shift={(2,0)}]
        \draw[tikzob,mm] 
        (0,0) node (a) {Pa}
        (0,-1) node (d)  {Pd}
        (1,-1) node (c) {Pb}
        (.5,-1) node[rotate=90,2label={below,\al_1}] {\Rightarrow}
        (.38,-.4) node[rotate=45,2label={above,P \al_2}] {\Rightarrow}
        ;
        \path[tikzar,mm] 
        (a) edge[swap] node {Ph} (d)
        (a) edge[bend left=30] node {Pu} (c)
        (d) edge[bend left=30] node {P\tilde{t}} (c)
        (d) edge[swap, bend right=30] node {t} (c);
      \end{scope}
    \end{tikzpicture}
    \]
    We say that $(\tilde{t}, \al_1, \al_2)$ is a \emph{lift} of $(u, t, \al)$.

  \item\label{it:opcart-2} Suppose given the following:
    \begin{itemize}
    \item a 2-cell $\beta \cn t \Rightarrow t'$;
    \item a pair of 1-cells $u, u' \cn a \to b$ in $ K$ with a 2-cell
      $\rho \cn u \Rightarrow u'$ between them; and
    \item 2-cell isomorphisms $\al \cn t \circ Ph \cong Pu$ and $\al'
      \cn t' \circ Ph \cong Pu'$, and lifts $(\tilde{t}, \al_1,
      \al_2)$, $(\tilde{t'}, \al_1', \al_2')$, of $(u, t,
      \al)$ and $(u', t', \al')$ respectively,
      such that the following
      equality holds.
      \begin{equation}\label{eq:opcart-diag0}
        \begin{tikzpicture}[x=30mm,y=20mm,vcenter]
          \draw[tikzob,mm] 
          (0,0) node (a) {Pa}
          (0,-1) node (d)  {Pd}
          (1,-1) node (c) {Pb}
          (.2,-.82) node[rotate=45,2label={above,\al}] {\Rightarrow} 
          (.55,-.4) node[rotate=45,2label={above,P\rho}] {\Rightarrow};
          \path[tikzar,mm] 
          (a) edge[swap] node {Ph} (d)
          (a) edge[bend left=40] node {Pu'} (c)
          (a) edge[bend right=30] node {Pu} (c)
          (d) edge[swap, bend right=30] node {t} (c);
          \draw (1.4,-.5) node {=};
          \begin{scope}[shift={(2,0)}]
            \draw[tikzob,mm] 
            (0,0) node (a) {Pa}
            (0,-1) node (d)  {Pd}
            (1,-1) node (c) {Pb}
            (.5,-1) node[rotate=90,2label={below,\be}] {\Rightarrow}
            (.38,-.4) node[rotate=45,2label={above,\al'}] {\Rightarrow};
            \path[tikzar,mm] 
            (a) edge[swap] node {Ph} (d)
            (a) edge[bend left=40] node {Pu'} (c)
            (d) edge[bend left=30] node {t'} (c)
            (d) edge[swap, bend right=30] node {t} (c);
          \end{scope}
        \end{tikzpicture}
      \end{equation}
    \end{itemize}
    \noindent Then there exists a unique 2-cell $\tilde{\be} \cn
    \tilde{t} \Rightarrow \tilde{t}'$ such that the following two
    equalities hold
    \begin{equation}\label{eq:opcart-diag1}
      \begin{tikzpicture}[x=30mm,y=20mm,vcenter]
        \draw[tikzob,mm] 
        (0,0) node (a) {Pd}
        (1,0) node (b)  {Pb}
        (2,0) node (c) {Pd}
        (3,0) node (d) {Pb}
        (.5,.3) node[rotate=-90,2label={above,\al_1}] {\Rightarrow} 
        (.5,-.38) node[rotate=-90,2label={above,P\wt{\be}}] {\Rightarrow} 
        (2.5,.3) node[rotate=-90,2label={above,\be}] {\Rightarrow} 
        (2.5,-.38) node[rotate=-90,2label={above,\al_1'}] {\Rightarrow};
        \path[tikzar,mm] 
        (a) edge[bend left=90] node {t} (b)
        (a) edge[swap] node {P\wt{t}} (b)
        (a) edge[swap, bend right=90] node {P\wt{t}'} (b)
        (c) edge[bend left=90] node {t} (d)
        (c) edge[swap] node {t'} (d)
        (c) edge[swap, bend right=90] node {P\wt{t}'} (d);
        \draw (1.5,0) node {=};
      \end{tikzpicture}
    \end{equation}
    \begin{equation}\label{eq:opcart-diag2}
      \begin{tikzpicture}[x=30mm,y=20mm,vcenter]
        \draw[tikzob,mm] 
        (0,0) node (a) {a}
        (0,-1) node (d)  {d}
        (1,-1) node (c) {b}
        (.2,-.82) node[rotate=45,2label={above,\al_2}] {\Rightarrow} 
        (.55,-.4) node[rotate=45,2label={above,\rho}] {\Rightarrow};
        \path[tikzar,mm] 
        (a) edge[swap] node {h} (d)
        (a) edge[bend left=40] node {u'} (c)
        (a) edge[bend right=30] node {u} (c)
        (d) edge[swap, bend right=30] node {\wt{t}} (c);
        \draw (1.4,-.5) node {=};
        \begin{scope}[shift={(2,0)}]
          \draw[tikzob,mm] 
          (0,0) node (a) {a}
          (0,-1) node (d)  {d}
          (1,-1) node (c) {b}
          (.5,-1) node[rotate=90,2label={below,\wt{\be}}] {\Rightarrow}
          (.38,-.4) node[rotate=45,2label={above,\al_2'}] {\Rightarrow};
          \path[tikzar,mm] 
          (a) edge[swap] node {h} (d)
          (a) edge[bend left=40] node {u'} (c)
          (d) edge[bend left=30] node {\wt{t}'} (c)
          (d) edge[swap, bend right=30] node {\wt{t}} (c);
        \end{scope}
      \end{tikzpicture}
    \end{equation}
  \end{enumerate}
\end{defn}

We give a characterization of opcartesian 1-cells via bipullbacks.

\begin{lem}[{\cite[Proposition  3.1.2]{Buc14Fibred}}]\label{lem:opfib_iff_bp}
  A 1-cell $h \cn a \to d$ in $ K$ is opcartesian with respect to $P$
  if and only if the commutative square
  \[
  \begin{tikzpicture}[x=30mm,y=20mm]
    \draw[tikzob,mm] 
    (0,0) node (a) { K(d,b)}
    (1,0) node (b) {L(Pd, Pb)}
    (0,-1) node (d)  { K(a,b)}
    (1,-1) node (c) {L(Pa,Pb)};
    \path[tikzar,mm] 
    (a) edge[swap] node {h^{*}} (d)
    (d) edge node[swap] {P} (c)
    (a) edge node {P} (b)
    (b) edge node {Ph^{*}} (c);
  \end{tikzpicture}
  \]
  is a bipullback in $\cat$ for all objects $b \in  K$.
\end{lem}

\cref{defn:cart-2,defn:loc-fib,defn:opcart} combine to give the following
definition of opfibration for 2-functors.

\begin{defn}\label{defn:opfib}
  A 2-functor $P\cn K \to L$ is an
  \emph{opfibration} if the following conditions hold:
  \begin{enumerate}
  \item 1-cells $f\cn P(x) \to z$ have opcartesian lifts $\wh{f}\cn x
    \to \wh{x}$, with $P\wh{f}=f$,
  \item $P$ is a local fibration, and
  \item horizontal composites of cartesian 2-cells are cartesian.
  \end{enumerate}
\end{defn}

\begin{rmk}\label{rmk:variances}
The previous definition mixes the variances of the lifting properties: 1-cells have opcartesian lifts, while 2-cells have cartesian ones. Any other mix
of variances can be obtained by applying either $\op$ (reversing the direction of 1-cells) or $\mathrm{co}$ (reversing the direction of 2-cells) to both the source and target. We have chosen this set of variances with the particular application of \cref{sec:application} in mind.
\end{rmk}

\subsection{Lax and oplax comma objects}\label{sec:comma-objs}

In this section we define and study two comma object constructions
that we call, respectively, the lax comma object $\laco{F,G}$ (see
\cref{defn:laco}) and the oplax comma object $\oplaco{F,G}$ (see
\cref{defn:oplaco}). In each case, the input is a pair of 2-functors
$F, G$ and the output is a 2-category together with a universal lax,
respectively oplax, transformation.  These are special cases of a more
general theory of lax limits developed by Gray
\cite{Gray80Closed,Gray80Existence}.

\begin{defn}\label{defn:laco}
  Given a cospan of 2-functors 
  \[
  \begin{tikzpicture}[x=12mm,y=12mm]
    \draw[tikzob,mm] 
    (0,0) node (e) {Z}
    (1,0) node (c) {Y}
    (1,1) node (d) {X}
    ;
    \path[tikzar,mm] 
    (e) edge[swap] node {G} (c)
    (d) edge node {F} (c)
    ;
  \end{tikzpicture}
  \]
  we define the \emph{lax comma object} $\laco{F,G}$ as the following
  2-category.
  \begin{itemize}
  \item An object of $\laco{F,G}$ consists of a triple $[x,f,z]$ where
    $x$ is an object of $X$, $z$ is an object of $Z$, and $f\cn F(x)
    \to G(z)$ is a 1-cell in $Y$.
    
  \item A 1-cell from $[x,f,z]$ to $[x',f',z']$ is given by a triple
    $[s,\al,t]$ where $s\cn x \to x'$ is a 1-cell in $X$, $t\cn z \to
    z'$ is a 1-cell in $Z$, and $\al$ shown below is a 2-cell in $Y$.
    \[
    \begin{tikzpicture}[x=18mm,y=18mm]
      \draw[tikzob,mm] 
      (0,0) node (Fx) {Fx}
      (1,0) node (Gz) {Gz}
      (0,-1) node (Fx') {Fx'}
      (1,-1) node (Gz') {Gz'}
      (.5,-.45) node[rotate=225,2label={above,}] {\Rightarrow}
      node[below right] {\al}
      ;
      \path[tikzar,mm] 
      (Fx) edge node {f} (Gz)
      (Fx') edge[swap] node {f'} (Gz')
      (Fx) edge[swap] node {Fs} (Fx')
      (Gz) edge node {Gt} (Gz')
      ;
    \end{tikzpicture}
    \]

  \item A 2-cell from $[s,\al,t]$ to $[s',\al',t']$ is 
  a pair
    $[\phi,\ga]$ where $\phi\cn s \Rightarrow s'$ is a 2-cell in $X$
    and $\ga\cn t \Rightarrow t'$ is a 2-cell in $Z$, subject to the
    following equality of pastings below.
    \begin{equation}\label{eq:laco-2}
    \begin{tikzpicture}[x=25mm,y=20mm,vcenter]
      \newcommand{\boundary}{
      \draw[tikzob,mm] 
      (0,0) node (Fx) {Fx}
      (1,0) node (Gz) {Gz}
      (0,-1) node (Fx') {Fx'}
      (1,-1) node (Gz') {Gz'}
      ;
      \path[tikzar,mm] 
      (Fx) edge node {f} (Gz)
      (Fx') edge[swap] node {f'} (Gz')
      (Fx) edge[bend right=40, swap] node {Fs'} (Fx')
      (Gz) edge[bend left=40] node {Gt} (Gz')
      ;
      }
      \boundary
      \draw[tikzob,mm] 
      (.70,-.45) node[rotate=225,2label={above,\al}] {\Rightarrow}
      (Fx) ++(0,-.55) node[rotate=180, 2label={below,F\phi}] {\Rightarrow}
      ;
      \path[tikzar,mm] 
      (Fx) edge[bend left=40] node {Fs} (Fx')
      ;

      \draw (1.75,-.5) node {=};

      \begin{scope}[shift={(2.5,0)}]
        \boundary
        \draw[tikzob,mm] 
        (.2,-.45) node[rotate=225,2label={above,\al'}] {\Rightarrow}
        (Gz) ++(0,-.55) node[rotate=180, 2label={below,G\ga}] {\Rightarrow}
        ;
        \path[tikzar,mm] 
        (Gz) edge[bend right=40,swap] node {Gt'} (Gz')
        ;
      \end{scope}
    \end{tikzpicture}
    \end{equation}
  \end{itemize}
  Composition of 1-cells is given by composing the 1-cell components
  and vertically pasting the 2-cell components.  Composition of
  2-cells is given componentwise.
\end{defn}

\begin{rmk}
When $X$, $Y$ and $Z$ are considered as strict $\infty$-categories, the lax comma $\infty$-category of \cite[\S 6]{AM20QuillenAII} is in fact a 2-category, and coincides with the lax comma object just defined. Several of the results presented below can be seen as special cases of results in \cite{AM20QuillenAII}.
\end{rmk}

\begin{prop}\label{prop:laco-univ}
  The lax comma object $\laco{F,G}$ is equipped with projection 2-functors
  $p_X$ and $p_Z$ together with a lax natural transformation $\pi$ as
  shown below.
  \[
  \begin{tikzpicture}[x=20mm,y=16mm]
    \draw[tikzob,mm] 
    (0,0) node (z) {Z}
    (1,1) node (x) {X}
    (1,0) node (y) {Y}
    (0,1) node (laco) {\laco{F,G}}
    (.5,.55) node[rotate=225,2label={above,}] {\Rightarrow}
    node[below right] {\pi}
    ;
    \path[tikzar,mm] 
    (z) edge[swap] node {G} (y)
    (x) edge node {F} (y)
    (laco) edge node {p_X} (x)
    (laco) edge[swap] node {p_Z} (z)
    ;
  \end{tikzpicture}
  \]
\noindent These data are universal in the following sense.  Let $K$ be
a 2-category, $R$ and $Q$ be 2-functors, and $\la$ be a lax natural
transformation as shown below.
  \[
  \begin{tikzpicture}[x=20mm,y=16mm]
    \draw[tikzob,mm] 
    (0,0) node (z) {Z}
    (1,1) node (x) {X}
    (1,0) node (y) {Y}
    (0,1) node (laco) {K}
    (.5,.55) node[rotate=225,2label={above,}] {\Rightarrow}
    node[below right] {\la}
    ;
    \path[tikzar,mm] 
    (z) edge[swap] node {G} (y)
    (x) edge node {F} (y)
    (laco) edge node {R} (x)
    (laco) edge[swap] node {Q} (z)
    ;
  \end{tikzpicture}
  \]
  Then there exists a unique 2-functor $h\cn K \to \laco{F,G}$ such
  that $R=p_X\circ h$, $Q=p_Z\circ h$, and the following equality of
  pasting diagrams holds.
  \[
  \begin{tikzpicture}[x=20mm,y=20mm,baseline={(0,.5).base}]
    \newcommand{\boundary}{
    \draw[tikzob,mm] 
    (-.7,.7) node (k) {K}
    (0,-1) node (z) {Z}
    (1,0) node (x) {X}
    (1,-1) node (y) {Y}
    ;
    \draw[tikzar,mm] 
    (z) edge[swap] node (g) {G} (y)
    (x) edge node (f) {F} (y)
    (k) edge[swap, bend right] node (q) {Q} (z)
    (k) edge[bend left] node (r) {R} (x)
    ;
    }
    \begin{scope}
    \boundary
    \draw[tikzob,mm,2label={above,}] (y) ++(45:1.5) node {=};
    \draw[2cell] 
    node[between=k and y at .5, rotate=225,2label={above,\la}] {\Rightarrow}
    ;
    \end{scope}
    \begin{scope}[shift={(3.8,0)}]
      \boundary
    \draw[tikzob,mm] 
    (0,0) node (laco) {\laco{F, G}}
    ;
    \draw[tikzar,mm] 
    (laco) edge node {p_X} (x)
    (laco) edge[swap] node {p_Z} (z)
    (k) edge[dashed] node[pos=.6] {\exists !\,h} (laco)
    ;
    \draw[2cell]
    node[between=laco and y at .5, rotate=225, 2label={above,\pi}]
    {\Rightarrow}
    ;
    \end{scope}
  \end{tikzpicture}
   \]
\end{prop}

\begin{proof}
  The 2-functor $p_X\cn \laco{F,G} \to X$ is given by projection onto the
  first coordinate; the 2-functor $p_Z\cn \laco{F,G} \to Z$ is given by
  projection to the final coordinate.  The component of $\pi$ at an
  object $[x,f,z]$ is $f$.  The laxity of $\pi$ at a 1-cell
  $[s,\al,t]$ is $\al$.

  Now we turn to the universal property.  Given $K$, $R$, $Q$, and $\la$ as in the statement, we define a
  2-functor $h \cn K \to \laco{F,G}$ as follows:
  \begin{itemize}
  \item $k \mapsto [Rk,\la_k,Qk]$;
  \item $(m\cn k \to k') \mapsto [Rm, \la_m, Qm]$; and
  \item $(\mu\cn m \Rightarrow m') \mapsto [R\mu, Q\mu]$.
  \end{itemize}
  Note that the first and last coordinate are determined by the
  compatibility with $p_X$ and $p_Z$, respectively, while
  compatibility with $\pi$ determines the middle coordinate for
  objects and 2-cells, proving uniqueness.
\end{proof}

The 2-categorical dualities $\op$ (reversing 1-cells) and $\co$
(reversing 2-cells) provide the following isomorphism of lax comma
objects.
\begin{lem}\label{lem:lax-coop}
  For any 2-functors $F, G$, there is an isomorphism
  \[
  (\laco{G, F})^\coop \iso \laco{F^\coop, G^\coop}
  \]
  induced by the universal properties.
\end{lem}
\begin{proof}
  We use the identification $\lax(D,E) =
  \lax(D^\coop,E^\coop)^\coop$. The result follows by applying $\coop$
  to the square defining $\laco{G, F}$.
\end{proof}
  
\begin{defn}\label{defn:oplaco}
  Given a cospan of 2-functors
  \[
  \begin{tikzpicture}[x=12mm,y=12mm]
    \draw[tikzob,mm] 
    (0,0) node (e) {Z}
    (1,0) node (c) {Y}
    (1,1) node (d) {X}
    ;
    \path[tikzar,mm] 
    (e) edge[swap] node {G} (c)
    (d) edge node {F} (c)
    ;
  \end{tikzpicture}
  \]
  we define the \emph{oplax comma object} $\oplaco{F,G}$ as follows.

  \begin{itemize}
  \item An object of $\oplaco{F,G}$ consists of a triple $[x,f,z]$
    where $x$ is an object of $X$, $z$ is an object of $Z$, and $f\cn
    F(x) \to G(z)$ is a 1-cell in $Y$. In other words, the objects are
    the same as those of $\laco{F,G}$.
  \item A 1-cell from $[x,f,z]$ to $[x',f',z']$ is given by a triple
    $[s,\al,t]$ where $s\cn x \to x'$ is a 1-cell in $X$, $t\cn z \to
    z'$ in $Z$, and $\al \cn f' \circ Fs \Rightarrow Gt \circ f$ is a
    2-cell in $Y$. Note that the direction of the 2-cell is opposite
    to the one in $\laco{F,G}$.
  \item A 2-cell from $[s,\al,t]$ to $[s',\al',t']$ is 
  a pair
    $[\phi,\ga]$ where $\phi\cn s \Rightarrow s'$ is a 2-cell in $X$
    and $\ga\cn t \Rightarrow t'$ is a 2-cell in $Z$ subject to the an
    equality of pasting diagrams, analogous to \cref{eq:laco-2}.
  \end{itemize}
\end{defn}

We need the explicit descriptions of both the lax and oplax comma
objects, but we note that they are formally dual.
\begin{lem}\label{lem:lp-coop}
  For any 2-functors $F, G$, there is an isomorphism
  \[
  (\laco{G, F})^\op \iso \oplaco{F^\op, G^\op}
  \]
  induced by the universal properties.
\end{lem}
\begin{proof}
  We use the identification $\oplax(D,E) = \lax(D^\op,E^\op)^\op$. The
  result follows by applying $\op$ to the square defining $\laco{G,
    F}$.
\end{proof}

The following result is dual to \cref{prop:laco-univ} and left to the
reader.
\begin{prop}\label{prop:oplaco-univ} 
  The oplax comma object $\oplaco{F,G}$ is equipped with projection
  2-functors $p_X$ and $p_Z$ together with an oplax natural
  transformation $\pi$ as shown below.
  \[
  \begin{tikzpicture}[x=20mm,y=16mm]
    \draw[tikzob,mm] 
    (0,0) node (z) {Z}
    (1,1) node (x) {X}
    (1,0) node (y) {Y}
    (0,1) node (laco) {\oplaco{F,G}}
    (.5,.55) node[rotate=225,2label={above,}] {\Rightarrow}
    node[below right] {\pi}
    ;
    \path[tikzar,mm] 
    (z) edge[swap] node {G} (y)
    (x) edge node {F} (y)
    (laco) edge node {p_X} (x)
    (laco) edge[swap] node {p_Z} (z)
    ;
  \end{tikzpicture}
  \]
  \noindent These data are universal in the following sense.  Let $K$
  be a 2-category, $R$ and $Q$ be 2-functors, and $\la$ be an oplax
  natural transformation as shown below.
  \[
  \begin{tikzpicture}[x=20mm,y=16mm]
    \draw[tikzob,mm] 
    (0,0) node (z) {Z}
    (1,1) node (x) {X}
    (1,0) node (y) {Y}
    (0,1) node (laco) {K}
    (.5,.55) node[rotate=225,2label={above,}] {\Rightarrow}
    node[below right] {\la}
    ;
    \path[tikzar,mm] 
    (z) edge[swap] node {G} (y)
    (x) edge node {F} (y)
    (laco) edge node {R} (x)
    (laco) edge[swap] node {Q} (z)
    ;
  \end{tikzpicture}
  \]
  Then there exists a unique 2-functor $h\cn K \to \oplaco{F,G}$ such
  that $R=p_X\circ h$, $Q=p_Z\circ h$, and the following equality of
  pasting diagrams holds.
  \[
  \begin{tikzpicture}[x=20mm,y=20mm,baseline={(0,.5).base}]
    \newcommand{\boundary}{
    \draw[tikzob,mm] 
    (-.7,.7) node (k) {K}
    (0,-1) node (z) {Z}
    (1,0) node (x) {X}
    (1,-1) node (y) {Y}
    ;
    \draw[tikzar,mm] 
    (z) edge[swap] node (g) {G} (y)
    (x) edge node (f) {F} (y)
    (k) edge[swap, bend right] node (q) {Q} (z)
    (k) edge[bend left] node (r) {R} (x)
    ;
    }
    \begin{scope}
    \boundary
    \draw[tikzob,mm,2label={above,}] (y) ++(45:1.5) node {=};
    \draw[2cell] 
    node[between=k and y at .5, rotate=225,2label={above,\la}] {\Rightarrow}
    ;
    \end{scope}
    \begin{scope}[shift={(3.8,0)}]
      \boundary
    \draw[tikzob,mm] 
    (0,0) node (laco) {\oplaco{F, G}}
    ;
    \draw[tikzar,mm] 
    (laco) edge node {p_X} (x)
    (laco) edge[swap] node {p_Z} (z)
    (k) edge[dashed] node[pos=.6] {\exists !\,h} (laco)
    ;
    \draw[2cell]
    node[between=laco and y at .5, rotate=225, 2label={above,\pi}]
    {\Rightarrow}
    ;
    \end{scope}
  \end{tikzpicture}
  \]
\end{prop}

We need a few key results about the (op)lax comma objects.

\begin{lem}\label{lem:lp-id} 
For any 2-functor $G\cn E \to D$, the projection
\[
p_E\cn \laco{1_D, G} \to E
\]
is a homotopy equivalence on classifying spaces.
\end{lem}
\begin{proof}
  The universal property of the lax comma object, applied to the commutative square
  \[
  \begin{tikzpicture}[x=20mm,y=12mm,baseline={(0,.5).base}]
    \draw[tikzob,mm] 
    (0,0) node (e) {E}
    (1,1) node (c) {D}
    (1,0) node (d) {D}
    (0,1) node (laco) {E}
    ;
    \path[tikzar,mm] 
    (e) edge[swap] node {G} (d)
    (c) edge node {1} (d)
    (laco) edge[swap] node {1} (e)
    (laco) edge node {G} (c)
    ;
  \end{tikzpicture}
  \]
  gives an inclusion $J\cn E \to \laco{1_D, G}$ such that $p_E\circ J = 1_E$. It sends an object
  $z \in E$ to $[Gz, 1, z]$ and a 1-cell $g \cn x \to x'$ to $[Gg, 1,
    g]$.  It
  is straightforward to verify that there is a lax transformation $\mu
  \cn 1 \Rightarrow J\circ p_E$ with components
  \[
  \begin{array}{rcl}
    \mu_{[x,f,z]} & = & [f,1,1],\\
    \mu_{[s,\al,t]} & = & [\al,1].
  \end{array}
  \]
  Upon taking classifying spaces, these 2-functors and transformations
  give the desired homotopy equivalence (see \cref{thm:trans-htpy}).
\end{proof}

\begin{notn}\label{hat}
  Given an object $k$ of a 2-category $K$, we let $\nh{k}$ denote the
  2-functor $* \to K$ which sends the unique object to $k$.
\end{notn}

\begin{notn}\label{Delta}
 Given a 2-functor $F\cn C \to D$, we let $\De F$ denote any of the
 diagonal 2-functors $C \to [E, D]$, resp. $C \to \lax(E, D)$,
 $C \to \oplax(E, D)$, given by $\De F(c)(e) = F(c)$ for all
 cells $c \in C$ and $e \in E$.
\end{notn}

\begin{defn}\label{oplax-init}
  Let $E$ be a 2-category, and let $\inob\in E$ be an object. Consider
  the 2-functors
  \[
  \begin{tikzpicture}[x=30mm,y=20mm]
    \draw[tikzob,mm] 
    (1,0) node (E) {E}
    (0,0) node (pt) {*}
    ;
    \path[tikzar,mm] 
    (E) edge[bend left=20] node {u} (pt)
    (pt) edge[bend left=20] node {\nh{\inob}} (E);
  \end{tikzpicture}
  \]
  where $u$ is the unique map.  We say that $\inob$ is an \emph{oplax
    initial object} if there exists an oplax natural transformation
  \[
  h \cn \nh{\inob} \circ u \Rightarrow \Id_{E}
  \]
  such that the whiskering $h*\nh{\inob}$ is the identity at
  $\nh{\inob}$.  Note that the 2-functor $u \circ \nh{\inob}$ and the
  whiskering $u*h$ are necessarily identities.  We say that $\tmob \in
  E$ is an \emph{oplax terminal object} if it is oplax initial in
  $E^\op$.
\end{defn}
  
To be explicit, the data of an oplax initial object consists of
components $h_j\cn \inob \to j$ for each object $j \in E$ together
with 2-cells $h_{f}$ for each 1-cell $f\colon i \to j$ in $E$ as in
the following square.
\[
\begin{tikzpicture}[x=15mm,y=15mm]
  \draw[tikzob,mm] 
  (0,0) node (a) {\inob}
  (1,0) node (b) {\inob}
  (0,-1) node (c) {i}
  (1,-1) node (d) {j}
  (.45,-.45) node[rotate=225, 2label={above,}] {\Rightarrow}
  node[below right] {h_{f}}
  ;
  \path[tikzar,mm] 
  (a) edge[equal] node {} (b)
  (c) edge[swap] node {f} (d)
  (a) edge[swap] node {h_i} (c)
  (b) edge node {h_j} (d)
  ;
\end{tikzpicture}
\]
These must satisfy that $h_\inob=1_\inob$ and $h_{1_\inob}=1_{1_\inob}$.

\begin{lem}\label{lem:lp-initial}
  Consider the following cospan.
  \[
  \begin{tikzpicture}[x=12mm,y=12mm]
    \draw[tikzob,mm] 
    (0,0) node (e) {E}
    (1,1) node (c) {C}
    (1,0) node (d) {D}
    ;
    \path[tikzar,mm] 
    (e) edge[swap] node {G} (d)
    (c) edge node {F} (d)
    ;
  \end{tikzpicture}
  \]
  If $E$ has an oplax initial object $\inob$, then the lax comma objects
  \[
  \begin{tikzpicture}[x=24mm,y=12mm,baseline={(0,.5).base}]
    \draw[tikzob,mm] 
    (0,0) node (e) {*}
    (1,1) node (c) {C}
    (1,0) node (d) {\oplax(E, D)}
    (0,1) node (laco) {\laco{\De F, \nh{G}}}
    (.5,.5) node[rotate=225,2label={above,}] {\Rightarrow}
    node[below right] {}
    ;
    \path[tikzar,mm] 
    (e) edge[swap] node {\nh{G}} (d)
    (c) edge node {\De F} (d)
    (laco) edge[swap] node {} (e)
    (laco) edge node {} (c)
    ;
  \end{tikzpicture}
  \qquad \mathrm{and} \qquad 
  \begin{tikzpicture}[x=24mm,y=12mm,baseline={(0,.5).base}]
    \draw[tikzob,mm] 
    (0,0) node (e) {*}
    (1,1) node (c) {C}
    (1,0) node (d) {D}
    (0,1) node (laco) {\laco{F, \nh{G(\inob)}}}
    (.5,.5) node[rotate=225,2label={above,}] {\Rightarrow}
    node[below right] {}
    ;
    \path[tikzar,mm] 
    (e) edge[swap] node {\nh{G(\inob)}} (d)
    (c) edge node {F} (d)
    (laco) edge[swap] node {} (e)
    (laco) edge node {} (c)
    ;
  \end{tikzpicture}
  \]
  have homotopy equivalent classifying spaces. 
\end{lem} 
\begin{proof}
  By the universal property of \cref{prop:laco-univ} we have unique
  2-functors $d$ and $e$ indicated below.
  \[
  \begin{tikzpicture}[x=18mm,y=18mm,rotate=-45,baseline={(0,.5).base}]
    \newcommand{\boundary}{
    \draw[tikzob,mm] 
    (0,0) node (laco) {\laco{F, \nh{G(\inob)}}}
    (1,-1) node (pt) {*}
    (1,1) node (c) {C}
    (3,-0) node (de) {\oplax(E,D)}
    ;
    \draw[tikzar,mm] 
    (laco) edge[swap] node {} (pt)
    (laco) edge node {} (c)
    (pt) edge[swap, bend right] node (Ghat) {\nh{G}} (de)
    (c) edge[bend left] node (DF) {\De F} (de)
    ;
    }
    \begin{scope}
    \boundary
    \draw[tikzob,mm]
    (2,0) node (d) {D}
    ;
    \draw[tikzob,mm,2label={above,}] (d) ++(45:1.5) node {=};
    \draw[tikzar,mm] 
    (d) edge['] node[pos=.3] {\De 1_D} (de)
    (pt) edge node {\nh{G(\inob)}} (d)
    (c) edge['] node {F} (d)
    ;
    \draw[2cell] 
    node[between=d and Ghat at .55, rotate=-90, 2label={below,(\star)}]
    {\Rightarrow}
    node[between=laco and d at .5, rotate=225,2label={above,}] {\Rightarrow}
    ;
    \end{scope}
    \begin{scope}[shift={(45:3.8)}]
      \boundary
    \draw[tikzob,mm] 
    (.85,0) node (laco2) {\laco{\De F, \nh{G}}}
    ;
    \draw[tikzar,mm] 
    (laco2) edge node {} (c)
    (laco2) edge node {} (pt)
    (laco) edge[dashed] node[pos=.6] {\exists !\,d} (laco2)
    ;
    \draw[2cell]
    node[between=laco2 and de at .5, rotate=-90, 2label={above,}]
    {\Rightarrow}
    ;
    \end{scope}
  \end{tikzpicture}
  \]
  \[
  \begin{tikzpicture}[x=18mm,y=18mm,rotate=-45,baseline={(0,.5).base}]
    \newcommand{\boundary}{
    \draw[tikzob,mm] 
    (0,0) node (laco) {\laco{\De F, \nh{G}}}
    (1,-1) node (pt) {*}
    (1,1) node (c) {C}
    (3,-0) node (de) {D}
    ;
    \draw[tikzar,mm] 
    (laco) edge[swap] node {} (pt)
    (laco) edge node {} (c)
    (pt) edge[swap, bend right] node (Ghat) {\nh{G(\inob)}} (de)
    (c) edge[bend left] node (DF) {F} (de)
    ;
    }
    \begin{scope}
    \boundary
    \draw[tikzob,mm]
    (2,0) node (dpt) {\oplax(E,D)}
    ;
    \draw[tikzob,mm,2label={above,}] (dpt) ++(45:1.5) node {=};
    \draw[tikzar,mm] 
    (dpt) edge['] node[pos=.3] {\ev_\inob} (de)
    (pt) edge node {\nh{G}} (dpt)
    (c) edge['] node {\De F} (dpt)
    ;
    \draw[2cell] 
    node[between=laco and dpt at .5, rotate=225,2label={above,}] {\Rightarrow}
    ;
    \end{scope}
    \begin{scope}[shift={(45:3.8)}]
      \boundary
    \draw[tikzob,mm] 
    (.85,0) node (laco2) {\laco{F, \nh{G(\inob)}}}
    ;
    \draw[tikzar,mm] 
    (laco2) edge node {} (c)
    (laco2) edge node {} (pt)
    (laco) edge[dashed] node {\exists !\, e} (laco2)
    ;
    \draw[2cell]
    node[between=laco2 and de at .5, rotate=-90, 2label={above,}]
    {\Rightarrow}
    ;
    \end{scope}
  \end{tikzpicture}
  \]
  In these diagrams, the unlabeled regions commute and we write
  $\ev_\inob$ for evaluation at the object $\inob$ of $E$.  The
  unlabeled 2-cells are the structure 2-cells of the lax comma objects
  $\laco{\De F, \nh{G}}$ and $\laco{F, \nh{G(\inob)}}$.  The 2-cell
  labeled $(\star)$ is the 2-natural transformation whose component at
  the unique object of $*$ is the oplax transformation $Gh$.

  Explicit formulas for $d$ and $e$ are given as follows.  An object
  of $\laco{F, \nh{G(\inob)}}$ is a pair $[c,f_\inob]$ where $c \in C$
  and $f_\inob$ is a 1-cell $F(c) \to G(\inob)$.  Then applying $G$ to
  $h$ we have an oplax transformation whose component at $j \in E$ is
  given by
  \[
  ((Gh)_* f_\inob)_j = G(h_j) \circ f_\inob \cn F(c) \to G(\inob) \to G(j).
  \]
  Thus we obtain
  \[
  d[c,f_\inob] = [c, (Gh)_* f_\inob]
  \]
  on objects, with similar formulas on higher cells.  On the other
  hand, an object of $\laco{\De F, \nh{G}}$ is a pair $[c,f]$ where $c
  \in C$ and $f$ is an oplax transformation $\De F(c) \to G$.  We
  obtain
  \[
  e[c,f] = [c,f_\inob].
  \]

  The reader can easily verify that $e \circ d$ is the identity
  because both $u \circ \nh{\inob}$ and $h *\nh{\inob}$ are
  identities.  For the other composite, we have
  \[
  de[c,f] = d[c,f_\inob] = [c,G(h)_* f_\inob].
  \]
  The oplaxity of $f$ with respect to the 1-cells $h_j \cn \inob \to
  j$ provides a modification
  \[\ga_{c,f} \cn G(h)_* f_\inob \Rrightarrow f,\]  
  so that the 1-cell $[1_c,\ga_{c,f}] \cn [c,G(h)_* f_\inob] \to
  [c,f]$ is the component of a 2-natural transformation $\Id
  \Rightarrow de$.  The equality $ed = \Id$ and the 2-natural
  transformation $\Id \Rightarrow de$ provide a homotopy equivalence
  on classifying spaces as desired (see \cref{thm:trans-htpy}).
\end{proof}

\begin{rmk}
  The above argument can be rephrased in terms of the universal
  property of the lax comma object, with the relatively short proof
  above replaced by a sequence of large pasting diagrams.
\end{rmk}

Combining \cref{lem:lp-id,lem:lp-initial} proves the following
corollary.
\begin{cor}\label{cor:lp-id-contract}
  If $E$ has an oplax initial object $\inob$ then $\laco{\De \Id_{D},
    \nh{G}}$ is contractible for any $G\cn E \to D$.
\end{cor}

Applying \cref{lem:lp-coop} to \cref{lem:lp-initial} and
\cref{cor:lp-id-contract} yields the following dual result.
\begin{cor}\label{cor:lp-terminal}
  If $E$ has an oplax terminal object $\tmob$, then
  \[
  \oplaco{\nh{G}, \De F} \hty \oplaco{\nh{G(\tmob)}, F}.
  \]
  In particular, if $F = \Id_{D}$ then $\oplaco{\nh{G}, \De
    \Id_{D}}$ is contractible for any $G\cn E \to D$.
\end{cor}

\subsection{Comparison of pullbacks and lax comma objects}\label{sec:comparison}

This section is devoted to a comparison of the lax comma object with
the pullback of 2-categories.

\begin{notn}\label{defn:pb_notation}
  Given a cospan of 2-categories
  \[
  A \fto{F} C \fot{G} B
  \]
  we let $\pb(F,G)$ denote the pullback. For $k=0,1,2$, a $k$-cell $w$
  of the pullback consists of a pair $(w_1,w_2)$ of $k$-cells with
  $w_1 \in A, w_2 \in B$ such that $F(w_1) = G(w_2)$.
\end{notn}

\begin{rmk}
  We often refer to the pullback of 2-categories as constructed above
  as the \emph{strict 2-pullback} to emphasize that this is the
  strictest possible notion amongst all of the various types of
  pullbacks of 2-functors one could construct.
\end{rmk}

\begin{thm}\label{prop:fiber_vs_htpyfiber} 
  Let $P\cn  K \to L$ be an opfibration and $F\cn J \to L$ be
  any 2-functor.  Assume that all of the 2-cells of $L$ are invertible.  Then the inclusion
  \[
  i\cn \pb(P,F) \to \laco{P,F}
  \]
  induces a homotopy equivalence on classifying spaces.
\end{thm}
\begin{proof}
  We will define a 2-functor
  \[
  H\cn \laco{P,F} \to \pb(P,F),
  \]
  such that $Hi$ is the identity on $\pb(P,F)$, together with a
  pseudonatural transformation $\eta\cn 1 \Rightarrow iH$.  This implies
  that $i$ and $H$ induce inverse homotopy equivalences on classifying
  spaces (see \cref{thm:trans-htpy}).
  
  In order to define $H$, we note that $P$ being an opfibration
  between 2-categories means that
  \begin{enumerate}
  \renewcommand{\labelenumi}{(\theenumi)}
  \renewcommand{\theenumi}{\arabic{enumi}}
  \item\label{hfib:1} given a 1-cell $f\cn Px \to y$, we have an
    opcartesian lift $\wh{f} \cn x \to \wh{x}$; and
  \item\label{hfib:2} given a 2-cell $\si \cn u \Rightarrow Pr$, we
    have a cartesian lift $\wh{\si} \cn \wh{r} \Rightarrow r$.
  \end{enumerate}
  
  We now define $H$ as an assignment on 0-, 1-, and 2-cells of
  $\laco{P,F}$, using the notation of \cref{defn:laco}.
  \begin{itemize}
  \item For an object $[x, f, z]$, apply \eqref{hfib:1} to $f\cn Px \to
    Fz$.  This gives an opcartesian lift $\wh{f}:x \to \wh{x}$ such that
    $P\wh{f} = f$ and hence $P\wh{x} = Fz$.  Therefore $(\wh{x}, z)$ is
    an object of the pullback $\pb(P, F)$.  Note that if $f$ is an
    identity 1-cell, we take $\wh{x} = x$ and $\wh{f} = 1$.  Define
    \[
    H[x,f,z] = (\wh{x},z).
    \]
  
  \item For a 1-cell $[s, \al, t]\cn [x, f, z] \to [x', f', z']$, recall
    that $\al$ is an invertible 2-cell $Ft \circ f \cong f' \circ Ps$. Since $f = P\wh{f}$ and
    $f' = P\wh{f}'$, we have that $\al$ is a 2-cell isomorphism as
    depicted below.
    \[
    \begin{tikzpicture}[x=20mm,y=20mm]
      \draw[tikzob,mm] 
      (0,0) node (a) {Px}
      (0,-1) node (d)  {P\wh{x}}
      (1,-1) node (c) {P\wh{x}'};
      \path[tikzar,mm] 
      (a) edge[swap] node {P\wh{f}} (d)
      (a) edge node {P(\wh{f'}s)} (c)
      (d) edge[swap] node {Ft} (c);
      \draw[tikzob,mm]
      (.35,-.7) node[rotate=45,2label={above,\al}] {\Rightarrow}
      ;
    \end{tikzpicture}
    \]
    Thus, there exists a lift $(\wt{s},\ol{\al},\wt{\al})$ of
    $(\wh{f'}s,Ft,\al)$; see \cref{defn:opcart}
    \eqref{it:opcart-1}. Explicitly, we have
    \begin{itemize}
      \renewcommand{\labelitemii}{$\circ\:$}
    \item $\wt{s}\cn \wh{x} \to \wh{x}'$,
    \item $\ol{\al}\cn Ft \Rightarrow P\wt{s}$ invertible, and
    \item $\wt{\al}\cn \wt{s} \wh{f} \Rightarrow \wh{f}' s$ invertible
    \end{itemize}
    such that $\al$ is equal to the pasting of $\ol{\al}$ with
    $P\wt{\al}$.  Next we apply \eqref{hfib:2} with $\si=\ol{\al}$ to
    get a cartesian lift $\wh{\al}\cn\wh{s} \Rightarrow \wt{s}$ of
    $\ol{\al}$ such that $P\wh{\al} = \ol{\al}$ and hence $P\wh{s} =
    Ft$.  Note that $\wh{\al}$ is invertible since $\ol{\al}$ is.
    Therefore $(\wh{s},t)$ is a 1-cell $(\wh{x},z) \to (\wh{x}',z')$ in
    the pullback $\pb(F,P)$.
    
    We take specific lifts in two special cases. For an identity 1-cell
    \[
    [1, 1, 1] \cn [x, f, z] \to [x, f, z],
    \]
    we take $\wh{s}$ to be the identity 1-cell on $\wh{x}$. This can be
    done since $\wt{s}=1_{\wh{x}}$, $\ol{\al}=1_{1_{P\wh{x}}}$, and
    $\wt{\al}=1_{\wh{f}}$ is a valid lift for
    $(\wh{f},1_{P\wh{x}},1_{P\wh{f}})$, and identity 2-cells are always
    cartesian.
    
    For a 1-cell
    \[
    [s,1,t] \cn [x,1,z] \to [x',1,z'],
    \]
    we take $\wh{s}=\wt{s}=s$, with $\ol{\al}$, $\wt{\al}$ and
    $\wh{\al}$ being the appropriate identity 2-cells.  Define
    \[
    H[s, \al, t] = (\wh{s},t).
    \]
    
  \item For a 2-cell $[\phi, \ga] \cn [s, \al, t] \Rightarrow [s', \al',
    t']$, the condition given in \cref{eq:laco-2} for $[\phi,\ga]$ gives
    a version of \cref{eq:opcart-diag0}. The data required for
    \cref{defn:opcart}~(\ref{it:opcart-2}) consists of
    \begin{enumerate}
    \item the opcartesian 1-cell $h$ in the definition is $\wh{f}$;
    \item the 2-cell $\be\cn t \Rightarrow t'$ in the definition is
       $F\ga\cn Ft \Rightarrow Ft'$;
    \item the 2-cell $\rho\cn u \Rightarrow u'$ is $\wh{f}' * \phi\cn
      \wh{f}'s \Rightarrow \wh{f}'s'$;
    \item the isomorphism 2-cells $\al,\al'$ are the 2-cells of the same
      names; and
    \item the required lifts are $(\wt{s}, \ol{\al}, \wt{\al})$ and
      $(\wt{s}', \ol{\al}', \wt{\al}')$, respectively.
    \end{enumerate}
    Thus we can apply \cref{defn:opcart} to obtain a unique 2-cell of the form
    $\wt{\phi} \cn \wt{s} \Rightarrow \wt{s}'$ satisfying the
    appropriate versions of \cref{eq:opcart-diag1,eq:opcart-diag2}.
    Using that $P\wh{\al} = \ol{\al}$ and $P\wh{\al}'=\ol{\al}'$,
    together with \eqref{eq:opcart-diag1}, we have the following
    equality:
    \begin{align}\label{eq:the-equality}
      P(\wt{\phi} \circ \wh{\al})
      = P\wt{\phi} \circ P\wh{\al}
      = P\wt{\phi} \circ \ol{\al}
      = \ol{\al}' \circ F\ga
      = P\wh{\al}' \circ F\ga.
    \end{align}
    Since the 2-cell $\wh{\al}'$ is cartesian, the pullback in
    \cref{defn:cart-2} applied to the pair $(\ga,\wt{\phi})$ produces a
    unique 2-cell $\wh{\phi} \cn \wh{s} \Rightarrow \wh{s}'$ such that both
    $P\wh{\phi} = F\ga$ and $\wh{\al}' \wh{\phi} = \wt{\phi}
    \wh{\al}$. Therefore $(\wh{\phi}, \ga)$ is a 2-cell $(\wh{s}, t)
    \Rightarrow (\wh{s}', t')$ in $\pb(F,P)$.  Define
    \[
    H[\phi,\ga] = (\wh{\phi},\ga).
    \]
  \end{itemize}
  \noindent This completes the definition of $H$ on cells.  Note that
  our definition of $H$ on 2-cells produces unique $\wt{\phi}$ and
  $\wh{\phi}$.  This uniqueness implies that $H$ strictly preserves
  identity 2-cells and vertical composition of 2-cells.
  
  We now construct the pseudofunctoriality constraints for $H$. We have
  specifically chosen the definition of $H$ on 1-cells to strictly
  preserve identities. Given composable 1-cells
  \[
  [s_1, \al_1, t_1] \cn [x, f, z] \to [x', f', z'] \mathrm{\quad and\quad}
  [s_2, \al_2, t_2] \cn [x', f', z'] \to [x'', f'', z''],
  \]
  we must construct a 2-cell isomorphism 
  \[
  [\wh{s}_2, t_2] \circ [\wh{s}_1, t_1] \cong [\wh{s_2 \circ s_1}, t_2 \circ t_1].
  \]
  The second coordinate will be the identity 2-cell, and we now
  produce a 2-cell isomorphism $\wh{s_2}\circ \wh{s_1} \cong \wh{s_2
    \circ s_1}$.
  
  We have $\wt{s}_2$, $\wt{s}_1$, and $\wt{s_2\circ s_1}$ constructed
  as above.  Let $\al_{21}$ denote the pasting of $\al_2$ with
  $\al_1$.  Then
  \[
  (\wt{s_2\circ s_1},\ol{\al}_{21}, \wt{\al}_{21}) \mathrm{\quad and \quad}
  (\wt{s}_2\circ\wt{s}_1, \ol{\al}_2 * \ol{\al}_1, (\wt{\al}_2*s_1) \circ (\wt{s}_2*\wt{\al}_1))
  \]
  are two lifts of $(\wh{f}''s_2s_1,F(t_2t_1),\al_{21})$.  Thus, by
  \cref{defn:opcart} \eqref{it:opcart-2} with $\rho$ and $\be$
  identity 2-cells, there exists a unique isomorphism $\de \cn
  \wt{s}_2 \circ \wt{s}_1 \cong \wt{s_2\circ s_1} $ compatible with
  the rest of the data. We define the pseudofunctoriality constraint
  to have first coordinate
  \[
  \wh{s}_2 \circ \wh{s}_1
  \xRightarrow{\wh{\al}_2 * \wh{\al}_1}
  \wt{s}_2 \circ \wt{s}_1
  \xRightarrow{\delta}
  \wt{s_2 \circ s_1}
  \xRightarrow{(\wh{\al}_{21})^\inv}
  \wh{s_2 \circ s_1}.
  \]
  Applying $P$ to this composite gives the identity, so we have
  constructed a 2-cell in $\pb(P,F)$ as desired. The pseudofunctor
  axioms follow from the uniqueness of the 2-dimensional lifts. This
  concludes the construction of $H$ as a pseudofunctor $P \downarrow F
  \to \pb(P,F)$.

  We next consider the composite
  \[
  \pb(P, F) \stackrel{i}{\hookrightarrow} P \downarrow F \stackrel{H}{\to} \pb(P, F).
  \]
  On objects and 1-cells, this composite is the identity by
  construction. For 2-cells, one checks that if $P\phi = F\ga$ then
  the definition of $H$ on $[\phi,\ga]$ gives the equalities $\phi = \wt{\phi} =
  \wh{\phi}$ by uniqueness. The pseudofunctoriality constraints are
  constructed using the unique 2-dimensional lifts, which in this case
  are all identities, so the composite above is the identity
  2-functor.

  Now consider 
  \[
  P \downarrow F \stackrel{H}{\to} \pb(P, F) \stackrel{i}{\hookrightarrow} P \downarrow F.
  \]
  On objects, this maps $[x, f, z]$ to $[\wh{x}, 1, z]$, so we define
  the component of a pseudonatural transformation $\eta: 1 \Rightarrow
  iH$ to be the 1-cell $[\wh{f}, 1_f, 1_z]$. For a 1-cell $[s, \al,
    t]$, the pseudonaturality constraint is given by
  \[
  [\wt{\al} \circ (\wh{\al} * \wh{f}), 1_t] \cn [\wh{s}\wh{f},1_{Ft \circ f},t] \cong [\wh{f}\,'s,\al,t].
  \]
  One can check that this is indeed a 2-cell in $P \downarrow F$. The
  pseudonaturality axiom reduces to the equality
  \[
  \wt{\al}_{21} \circ (\delta * \wh{f}\,) = (\wt{\al}_2 * s_1) \circ (\wt{s}_2 * \wt{\al}_1),
  \]
  which is precisely the compatibility \eqref{eq:opcart-diag2} for
  $\de$.
\end{proof}

\begin{rmk}
  Our proof of \cref{prop:fiber_vs_htpyfiber} only makes use of
  invertibility for $\al$ and $\al'$.  Therefore the same argument
  also proves, without the assumption of invertible 2-cells in $L$,
  that there is a homotopy equivalence between the classifying spaces
  of the pseudo-comma object (defined via a pseudonatural
  transformation in the square, but with a universal property still up
  to isomorphism) and the pullback.
\end{rmk}

\begin{rmk}
  In the case of 1-categories, one has an adjunction between the
  pullback and comma object.  The construction we have given in the
  proof of \cref{prop:fiber_vs_htpyfiber} can be extended to a
  biadjunction, but we leave that to the interested reader.
\end{rmk}

Now we apply \cref{prop:fiber_vs_htpyfiber} to give a version of
Quillen's Theorem B for opfibrations of 2-categories.  We use the
following terms.
\begin{defn}
  Suppose $F\cn C \to D$ is a 2-functor and $\phi\cn x \to y$ is a
  1-cell in $D$.
  \begin{enumerate}
  \item Let $\laco{F,\nh{x}}$, respectively $\laco{F,\nh{y}}$, denote
    the lax comma object for the cospan given by $F$ and $x$,
    respectively $y$, regarded as a 2-functor $* \to D$.
  \item Regard $\phi$ as a 2-natural transformation of 2-functors $*
    \to D$. Let $\phi_*\cn \laco{F,\nh{x}} \to \laco{F,\nh{y}}$ denote
    the 2-functor given by the universal property described in
    \cref{prop:laco-univ} with $\la$ given by the following pasting.
    \[
    \begin{tikzpicture}[x=20mm,y=16mm]
      \draw[tikzob,mm] 
      (0,0) node (z) {*}
      (1,1) node (x) {X}
      (1,0) node (y) {Y}
      (0,1) node (laco) {\laco{F,\nh{x}}}
      (.5,.65) node[rotate=225,2label={below,\pi}] {\Rightarrow}
      (.5,0) node[rotate=270,2label={below,\phi}] {\Rightarrow}
      ;
      \path[tikzar,mm] 
      (z) edge[swap,bend right] node {\nh{y\,}} (y)
      (z) edge[bend left] node {\nh{x\,}} (y)
      (x) edge node {F} (y)
      (laco) edge node {} (x)
      (laco) edge[swap] node {} (z)
      ;
    \end{tikzpicture}
    \]
    We call $\phi_*\cn \laco{F,\nh{x}} \to \laco{F,\nh{y}}$ the
    \emph{base change} 2-functor induced by $\phi$.
  \end{enumerate}
\end{defn}

\begin{thm}[\cite{Ceg11Homotopy}]\label{thm:cegarra}
  Suppose $F\cn C \to D$ is a 2-functor such that for every 1-cell
  $\phi\cn x \to y$ in $D$ the base change 2-functor $\phi_*\colon
  \laco{F,\nh{x}} \to \laco{F,\nh{y}}$ is a homotopy equivalence on
  classifying spaces.  Then for each $x \in D$ and each $\dot{x}\in
  F^\inv(x)$, the sequence
  \[
  \laco{F,\nh{x}} \to C \to D
  \] 
  induces a homotopy fiber sequence on classifying spaces with respect
  to the basepoints $[\dot{x},1] \in B\laco{F,\nh{x}}$, $\dot{x} \in
  BC$, and $x \in BD$.
\end{thm}

\begin{rmk}
  The result in \cite{Ceg11Homotopy} is given and proved for
  $\laco{\nh{x},F}$, with the base change maps going in the opposite
  direction. The version given here follows from the original,
  together with \cref{lem:lax-coop} and the fact that given a
  2-category $C$, there is a natural homotopy equivalence $BC \simeq
  B(C^\coop)$ (see \cite{CCG10Nerves}).
\end{rmk}

If $F$ is an opfibration and the 2-cells in its codomain are
invertible, then \cref{prop:fiber_vs_htpyfiber} shows that we can
replace $\laco{F,\nh{x}}$ with $F^\inv(x)$ to obtain the following
result.
\begin{cor}\label{cor:quillenB-opfib-2}
  Suppose that all 2-cells in $D$ are invertible, and $F\cn C \to D$
  is an opfibration such that, for every 1-cell $\phi\cn x \to y$ in
  $D$ the the base change 2-functor $\phi_*\colon \laco{F,\nh{x}} \to
  \laco{F,\nh{y}}$ is a homotopy equivalence on classifying spaces.
  Then for each $x \in D$ and each $\dot{x}\in F^\inv(x)$, the
  sequence
  \[
  F^\inv(x) \to C \to D
  \]
  induces a homotopy fiber sequence on classifying spaces with respect
  to the basepoints $\dot{x} \in BF^\inv(x)$, $\dot{x} \in BC$, and $x
  \in BD$.
\end{cor}

\section{Homology spectral sequence for an opfibration}\label{sec:spectral-sequence}

In this section we let $F\cn C \to D$ be a 2-functor and construct an
associated homology spectral sequence.  This begins with a detailed
description of the normal oplax nerve $ND$ (\cref{defn:nopl-nerve}) in
terms of Street's oriented simplices.  Then in
\cref{sec:bisimp_complex} we define a bisimplicial set associated to
$F$.  In \cref{sec:ss_from_E0} we analyze the spectral sequence
associated to this bisimplicial set and give an identification of its
$E^2$ page in the case that $F$ is an opfibration
(\cref{cor:two_isos}).

\subsection{Oriented simplices}

We give the 2-truncated version of Street's orientals from \cite{Str87Algebra}. In \textit{loc. cit.}, Street
defines the $n$th oriental to be the ``free strict $n$-category generated by an $n$-simplex,'' and thus our orientals
are obtained from his by forcing all $k$-cells to be identities for $k > 2$.

\begin{defn}[\cite{Str87Algebra}]\label{defn:orientals}
  For $p \ge 0$, the \emph{oriental} $\sO(p)$ is the 2-category
  defined as follows.
  \begin{itemize}
  \item The objects of $\sO(p)$ are the natural numbers $0, \ldots,
    p$.
  \item The nonidentity 1-cells of $\sO(p)$ are generated by pairs
    $(i,j)\cn i \to j$ for each $i < j$.
  \item The nonidentity 2-cells of $\sO(p)$ are generated by triples
    $(i,j,k)$ for each $i < j < k$, with source and target as
    indicated below.
    
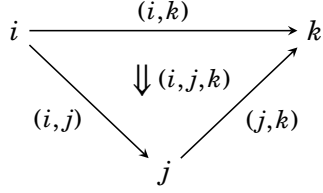
\[
    \begin{tikzpicture}[x=20mm,y=19mm]
      \draw[tikzob,mm] 
      (-1,1) node (i) {i}
      (1,1) node (k) {k}
      (0,0) node (j) {j};
      \path[tikzar,mm,style={font=\small}] 
      (i) edge[swap] node {(i,j)} (j)
      (j) edge[swap] node {(j,k)} (k)
      (i) edge node {(i,k)} (k);    
      \draw[tikzob,mm]
      (-.15,.65) node[2label={above,}] {\Downarrow}
      node[right,font=\small] {\ (i,j,k)};
    \end{tikzpicture}
    \]
  \end{itemize}
  The 2-cells are subject to the condition that each quadruple $i < j
  < k < l$ yields a commuting tetrahedron; that is, an equality of
  pastings as below.
  \[
  \begin{tikzpicture}[x=15mm,y=15mm,baseline=(-60:1)]
    \draw[tikzob,mm] 
    (0,0) node (i) {i}
      (-90:.55) node[rotate=-60,font=\small] {\Downarrow}
    (-120:1) node (j) {j}
    (-60:1) node (k) {k}
      ++(90:.55) node[font=\small] {\Downarrow}
    (0:1) node (l) {l};
    \path[tikzar,mm] 
    (i) edge (j)
    (j) edge (k)
    (k) edge (l)
    (i) edge (l)
    (i) edge (k);    
  \end{tikzpicture}
  =
  \begin{tikzpicture}[x=15mm,y=15mm,baseline=(-60:1)]
    \draw[tikzob,mm] 
    (0,0) node (i) {i}
    (-60:1) node (j) {j}
      +(90:.55) node[font=\small] {\Downarrow}
    ++(0:1) node (k) {k}
    (0:1) node (l) {l}
      ++(-90:.55) node[rotate=60,font=\small] {\Downarrow};
    \path[tikzar,mm] 
    (i) edge (j)
    (j) edge (k)
    (k) edge (l)
    (i) edge (l)
    (j) edge (l);
  \end{tikzpicture}
  \]
\end{defn}

\begin{rmk}\label{rmk:reading-tetrahedra}
  Two observations may help the reader parse these tetrahedra: First,
  the boundary of the two halves consists of the ``short spine'' $i
  \to l$ along the top as the source, and the ``long spine'' $i \to j
  \to k \to l$ along the bottom 3 edges as the target.  Second, each
  2-cell has a single arrow as its source, and a composite of two
  arrows as its target.  We use this oplax direction in virtually all
  of the discussion below.
\end{rmk}

The following two propositions can be read off from \cite{Gur2009Nerves} following 
\cite{Str82Nerves, Str87Algebra}.
\begin{prop}\label{prop:orientals-give-nopl-simplices}
  For $p \ge 0$, normal oplax functors $[p] \to D$ are in bijective
  correspondence with strict 2-functors $\sO(p) \to D$.
\end{prop}

\begin{prop}\label{prop:orientals-cosimplicial}
  The assignment $[p] \mapsto \sO(p)$ is the function on objects of a
  functor $\sO(-) \cn \De \to \iicat$. This cosimplicial 2-category
  has an associated nerve functor $N \cn \iicat \to \sset$ which is
  the normal oplax nerve (\cref{defn:nopl-nerve}).
\end{prop}

\begin{rmk}\label{rmk:oplax_init_term_orientals}
  Note that the oriental $\sO(p)$ has oplax initial object 0 and oplax
  terminal object $p$ (\cref{oplax-init}).
\end{rmk}

Unpacking the definition of $\oplax(\sO(p), D)$, we have the following
explicit description of its 1- and 2-cells.  For $\si, \si'\cn \sO(p)
\to D$, a 1-cell $f\cn \si \to \si'$ is an oplax transformation. It
consists of components $f_i\cn \si_i \to \si'_i$ for each $i$ together
with 2-cells $f_{ij}$ for each $i < j$ as shown below.
\[
\begin{tikzpicture}[x=20mm,y=15mm]
  \draw[tikzob,mm] 
  (0,0) node (a) {\si_i}
  (1,0) node (b) {\si_j}
  (0,-1) node (c) {\si'_i}
  (1,-1) node (d) {\si'_j}
  ;
  \path[tikzar,mm] 
  (a) edge node {\si_{(i,j)}} (b)
  (c) edge[swap] node {\si'_{(i,j)}} (d)
  (a) edge[swap] node {f_i} (c)
  (b) edge node {f_j} (d)
  ;
  \draw[tikzob,mm]
  (.5,-.4) node[rotate=215,2label={above,}] {\Rightarrow} 
  node[below right] {f_{ij}}
  ;
\end{tikzpicture}
\]
These data are subject to the axiom that a triple $i<j<k$ yields 
the following equality of pasting diagrams, where the unlabeled
2-cells are given by $\si(i,j,k)$ and $\si'(i,j,k)$, respectively.
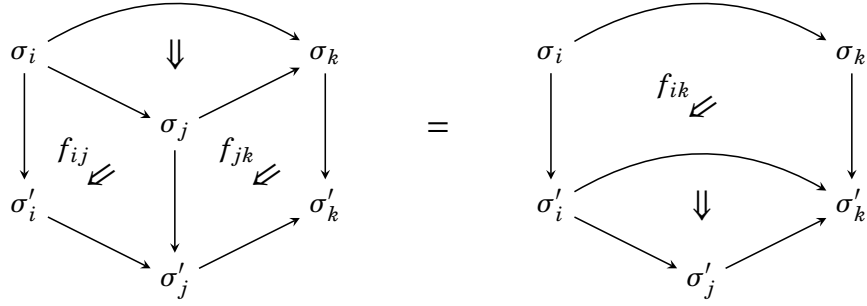
\[
\begin{tikzpicture}[x=20mm,y=20mm]
  \newcommand{\boundary}{
    \draw[tikzob,mm] 
    (0,0) node (si) {\si_i}
    (2,0) node (sk) {\si_k}
    (0,-1) node (spi) {\si'_i}
    (1,-1.5) node (spj) {\si'_j}
    (2,-1) node (spk) {\si'_k}
    ;
    \path[tikzar,mm] 
    (si) edge node {} (spi)
    (sk) edge node {} (spk)
    (si) edge[bend left] node {} (sk)
    (spi) edge node {} (spj)
    (spj) edge node {} (spk)
    ;
  }
  \boundary
  \draw[tikzob,mm] 
  (1,-.5) node (sj) {\si_j}
  (1,0) node[rotate=270,2label={above,}] {\Rightarrow} 
  (.5,-.8) node[rotate=215,2label={above,}] {\Rightarrow}
  node[above left] {f_{ij}}
  (1.6,-.8) node[rotate=215,2label={above,}] {\Rightarrow} 
  node[above left] {f_{jk}}
  (2.75,-.5) node[2label={above,}] {=}
  ;
  \path[tikzar,mm] 
  (si) edge node {} (sj)
  (sj) edge node {} (sk)
  (sj) edge node {} (spj)
  ;
  \begin{scope}[shift={(3.5,0)}]
    \boundary
    \draw[tikzob,mm] 
    (1,-.35) node[rotate=215,2label={above,}] {\Rightarrow} 
    node[above left] {f_{ik}}
    (1,-1) node[rotate=270,2label={above,}] {\Rightarrow} 
    ;
    \path[tikzar,mm] 
    (spi) edge[bend left] node {} (spk)
    ;
  \end{scope}
\end{tikzpicture}
\]

A 2-cell $m\cn f \Rightarrow f'$ between 1-cells $f, f'\cn \si \to
\si'$ is a modification.  This data consists of 2-cells $m_i \cn f_i
\Rightarrow f_i'$ for each $i$, subject to the axiom that the
following ``soup-can'' pastings are equal.
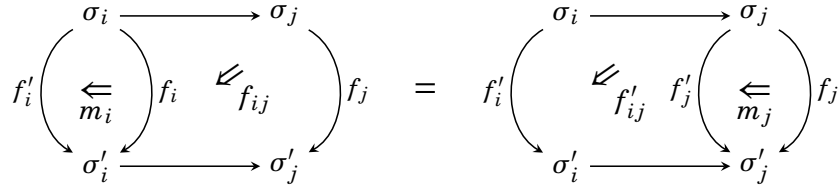
\[
\begin{tikzpicture}[x=25mm,y=20mm]
  \newcommand{\boundary}{
    \draw[tikzob,mm] 
    (0,0) node (a) {\si_i}
    (1,0) node (b) {\si_j}
    (0,-1) node (c) {\si'_i}
    (1,-1) node (d) {\si'_j}
    ;
    \path[tikzar,mm] 
    (a) edge node {} (b)
    (c) edge[swap] node {} (d)
    (b) edge[bend left=60] node {f_j} (d) 
    (a) edge[swap, bend right=60] node {f'_i} (c) 
    ;
  }
  \boundary
  \path[tikzar,mm] 
  (a) edge[bend left=60] node {f_i} (c) 
  ;
  \draw[tikzob,mm]
  (.7,-.4) node[rotate=215,2label={above,}] {\Rightarrow} 
  node[below right] {f_{ij}}
  (0,-.5) node[2label={above,}] {\Leftarrow}
  node[below=2pt] {m_i}
  (1.75,-.5) node[2label={above,}] {=}
  ;
  \begin{scope}[shift={(2.5,0)}]
    \boundary
    \path[tikzar,mm] 
    (b) edge[swap, bend right=60] node {f'_j} (d) 
    ;
    \draw[tikzob,mm]
    (.2,-.4) node[rotate=215,2label={above,}] {\Rightarrow} 
    node[below right] {f'_{ij}}
    (1,-.5) node[2label={above,}] {\Leftarrow}
    node[below=2pt] {m_j}
    ;
  \end{scope}
\end{tikzpicture}
\]

\subsection{Constructing a bisimplicial set}\label{sec:bisimp_complex}
Given a 2-functor $F\cn C \to D$, we construct a bisimplicial set and
give two alternate descriptions of it in terms of nerves of certain
(op)lax comma objects. This bisimplicial set is also considered (with the vertical and horizontal directions reversed) by Bullejos and Cegarra in their proof of Quillen's theorem A for 2-categories (see \cite[\S 4.1]{BC03Geometry}).

\begin{defn}\label{defn:E0}
  Let $F\cn C \to D$ be a 2-functor. Define a bisimplicial set
  $\bsimp$ as follows.  The set $\bsimp_{p,q}$ is the set of triples
  $(\om, \de, \si)$ of strict 2-functors which make the diagram below
  commute.
  \begin{equation}\label{eq:om-si-de-orientals}
    \begin{tikzpicture}[x=20mm,y=12mm,baseline={(qp).base}]
      \draw[tikzob,mm]
      (0,1) node (q) {\sO(q)}
      (0,-1) node (p) {\sO(p)}
      (1,0) node (qp) {\sO(q+1+p)}
      (2,1) node (C) {C}
      (2,-1) node (D) {D};
      \path[tikzar,mm]
      (q) edge node {\om} (C)
      (p) edge[swap] node {\si} (D)
      (qp) edge node {\de} (D)
      (q) edge (qp)
      (p) edge (qp)
      (C) edge node {F} (D);
    \end{tikzpicture}
  \end{equation}
  The map $\sO(q) \to \sO(q+1+p)$ is induced by the inclusion $[q] \to
  [q+1+p]$ in $\De$ that sends each $i\in [q]$ to itself, while the
  map $\sO(p) \to \sO(q+1+p)$ is induced by the inclusion $[p] \to
  [q+1+p]$ that sends each $j\in [p]$ to $q+1+j$.  The faces and
  degeneracies are given below, where $H$ denotes the horizontal
  direction (fix $q$, vary $p$) and $V$ denotes the vertical direction
  (fix $p$, vary $q$).
  \begin{align*}
    d_i^H(\om, \de, \si) &= (\om, d_{q+1+i}\de, d_i \si)\\
    s_i^H(\om, \de, \si) &= (\om, s_{q+1+i}\de, s_i \si)\\
    d_i^V(\om, \de, \si) &= (d_i\om, d_{i}\de, \si) \\
    s_i^V(\om, \de, \si) &= (s_i\om, s_{i}\de, \si)
  \end{align*}
\end{defn}

Unpacking \cref{defn:E0}, the data of a strict 2-functor $\de$ as
above consists of the following.
\begin{itemize}
\item The values of $\de$ on the sub-oriental with vertices $0 <
  \cdots < q$ are given by $F \circ \om$.
\item The values of $\de$ on the sub-oriental with vertices $q+1 <
  \cdots < q + 1 + p$ are given by $\si$.
\item For $i\in [q]$ and $j\in [p]$, $\de_{(i,q+1+j)}$ is a 1-cell
  $F(\om_i) \to \si_j$.
\item For $i\in [q]$ and $j<j'\in [p]$, $\de_{(i,q+1+j,q+1+j')}$ is a
  2-cell as below.
  \[
  \begin{tikzpicture}[x=14mm,y=16mm]
    \draw[tikzob,mm] 
    (0,1) node (omi) {F(\om_i)} 
    (-1,0) node (sj) {\si_j} 
    (1,0) node (sj') {\si_{j'}}; 
    \path[tikzar,mm] 
    (sj) edge[swap] node {\si_{(j,j')}} (sj') 
    (omi) edge[swap] node {\de_{(i,q+1+j)}} (sj) 
    (omi) edge node {\de_{(i,q+1+j')}} (sj'); 
    \draw[tikzob,mm]
    (0,.4) node[rotate=215,2label={above,}] {\Rightarrow} 
    ;
  \end{tikzpicture}
  \]
\item For $i<i'\in [q]$ and $j\in [p]$, $\de_{(i,i',q+1+j)}$ is a
  2-cell as below.
  \[
  \begin{tikzpicture}[x=14mm,y=16mm]
    \draw[tikzob,mm] 
    (0,1) node (omi) {F(\om_i)} 
    (-1,0) node (sj) {F(\om_{i'})} 
    (1,0) node (sj') {\si_{j}}; 
    \path[tikzar,mm] 
    (sj) edge[swap] node {\de_{(i',q+1+j)}} (sj') 
    (omi) edge[swap] node {F(\om_{(i,i')})} (sj) 
    (omi) edge node {\de_{(i,q+1+j)}} (sj'); 
    \draw[tikzob,mm]
    (0,.4) node[rotate=215,2label={above,}] {\Rightarrow} 
    ;
  \end{tikzpicture}
  \]
\item These 2-cells fit into commuting tetrahedra, as in the equalities of pasting diagrams below.
  \begin{itemize}
  \item For $i\in [q]$ and $j<j'<j''\in [p]$,
    \[
    \begin{tikzpicture}[x=15mm,y=15mm,baseline=(-60:1)]
      \draw[tikzob,mm] 
      (0,0) node (a) {F(\om_i)}
      (-90:.55) node[rotate=-60,font=\small] {\Downarrow}
      (-120:1) node (b) {\si_j}
      (-60:1) node (c) {\si_{j'}}
      ++(90:.55) node[font=\small] {\Downarrow}
      (0:1) node (d) {\si_{j''}};
      \path[tikzar,mm] 
      (a) edge (b)
      (b) edge (c)
      (c) edge (d)
      (a) edge (d)
      (a) edge (c);    
    \end{tikzpicture}
    =
    \begin{tikzpicture}[x=15mm,y=15mm,baseline=(-60:1)]
      \draw[tikzob,mm] 
      (0,0) node (a) {F(\om_i)}
      (-60:1) node (b) {\si_j}
      +(90:.55) node[font=\small] {\Downarrow}
      ++(0:1) node (c) {\si_{j'}}
      (0:1) node (d) {\si_{j''}}
      ++(-90:.55) node[rotate=60,font=\small] {\Downarrow};
      \path[tikzar,mm] 
      (a) edge (b)
      (b) edge (c)
      (c) edge (d)
      (a) edge (d)
      (b) edge (d);
    \end{tikzpicture}
    \]
  \item For $i<i'\in [q]$ and $j<j'\in [p]$,
    \[
    \begin{tikzpicture}[x=15mm,y=15mm,baseline=(-60:1)]
      \draw[tikzob,mm] 
      (0,0) node (a) {F(\om_i)}
      (-90:.55) node[rotate=-60,font=\small] {\Downarrow}
      (-120:1) node (b) {F(\om_{i'})}
      (-60:1) node (c) {\si_j}
      ++(90:.55) node[font=\small] {\Downarrow}
      (0:1) node (d) {\si_{j'}};
      \path[tikzar,mm] 
      (a) edge (b)
      (b) edge (c)
      (c) edge (d)
      (a) edge (d)
      (a) edge (c);    
    \end{tikzpicture}
    =
    \begin{tikzpicture}[x=15mm,y=15mm,baseline=(-60:1)]
      \draw[tikzob,mm] 
      (0,0) node (a) {F(\om_i)}
      (-60:1) node (b) {F(\om_{i'})}
      +(90:.55) node[font=\small] {\Downarrow}
      ++(0:1) node (c) {\si_j}
      (0:1) node (d) {\si_{j'}}
      ++(-90:.55) node[rotate=60,font=\small] {\Downarrow};
      \path[tikzar,mm] 
      (a) edge (b)
      (b) edge (c)
      (c) edge (d)
      (a) edge (d)
      (b) edge (d);
    \end{tikzpicture}
    \]
  \item For $i<i'<i''\in [q]$ and $j\in [p]$,
    \[
    \begin{tikzpicture}[x=15mm,y=15mm,baseline=(-60:1)]
      \draw[tikzob,mm] 
      (0,0) node (a) {F(\om_i)}
      (-90:.55) node[rotate=-60,font=\small] {\Downarrow}
      (-120:1) node (b) {F(\om_{i'})}
      (-60:1) node (c) {F(\om_{i''})}
      ++(90:.55) node[font=\small] {\Downarrow}
      (0:1) node (d) {\si_{j}};
      \path[tikzar,mm] 
      (a) edge (b)
      (b) edge (c)
      (c) edge (d)
      (a) edge (d)
      (a) edge (c);    
    \end{tikzpicture}
    =
    \begin{tikzpicture}[x=15mm,y=15mm,baseline=(-60:1)]
      \draw[tikzob,mm] 
      (0,0) node (a) {F(\om_i)}
      (-60:1) node (b) {F(\om_{i'})}
      +(90:.55) node[font=\small] {\Downarrow}
      ++(0:1) node (c) {F(\om_{i''})}
      (0:1) node (d) {\si_{j}}
      ++(-90:.55) node[rotate=60,font=\small] {\Downarrow};
      \path[tikzar,mm] 
      (a) edge (b)
      (b) edge (c)
      (c) edge (d)
      (a) edge (d)
      (b) edge (d);
    \end{tikzpicture}
    \]
  \end{itemize}
\end{itemize}

The data of $\delta$ can be interpreted in two ways: as a lax
transformation whose components are oplax transformations, or as an
oplax transformation whose components are lax transformations, as the
following two results explain.  Recall from \cref{hat} we let
$\nh{k}\cn * \to K$ denote the constant functor at an object $k$ of a
2-category $K$.
\begin{lem}\label{lem:p-filtration}
  Fix $\si \in N_pD$ and consider the lax comma object $\ldar{F , \si}$ in the square below.
  \[
  \begin{tikzpicture}[x=26mm,y=16mm]
    \draw[tikzob,mm] 
    (0,1) node (Oq) {\ldar{F , \si}}
    (1,1) node (C) {C}
    (0,0) node (*) {*}
    (1,0) node (DOp) {\oplax(\sO(p), D)};
    \path[tikzar,mm] 
    (Oq) edge node {} (C)
    (*) edge[swap] node {\nh{\si}} (DOp)
    (Oq) edge node {} (*)
    (C) edge node {\De F} (DOp);
    \draw[tikzob,mm]
    (.5,.5) node[rotate=225,2label={above,}] {\Rightarrow}
    ++(3mm,-2mm) node {};
  \end{tikzpicture}
  \]  
  Then the set of pairs $(\om,\de)$ such that $(\om, \de,
  \si) \in \bsimp_{p,q}$ is in bijective correspondence with
  $N_q(\ldar{F , \si})$, the set of $q$-simplices in the nerve of
  $\ldar{F , \si}$. Under this bijection, $d_i^V(-, -, \si)$
  corresponds to $d_i$, and similarly for degeneracies.
\end{lem}
\begin{proof}
  A $q$-simplex in $N(\ldar{F , \si})$ is given by a 2-functor
  \[
  \sO(q) \to \ldar{F , \si}.
  \]
  By \cref{prop:laco-univ}, such 2-functors are in bijection with pairs $(\omega,\la)$, where
  $\omega$ is a 2-functor and $\la$ is a lax transformation $\la$ as
  shown below.
  \[
  \begin{tikzpicture}[x=26mm,y=16mm]
    \draw[tikzob,mm] 
    (0,1) node (Oq) {\sO(q)}
    (1,1) node (C) {C}
    (0,0) node (*) {*}
    (1,0) node (DOp) {\oplax(\sO(p), D)};
    \path[tikzar,mm] 
    (Oq) edge node {\om} (C)
    (*) edge[swap] node {\nh{\si}} (DOp)
    (Oq) edge node {} (*)
    (C) edge node {\De F} (DOp);
    \draw[tikzob,mm]
    (.5,.5) node[rotate=225,2label={above,}] {\Rightarrow}
    ++(3mm,-2mm) node {\la};
  \end{tikzpicture}
  \]  
  The lax transformation $\la$ consists of the following data:
  \begin{itemize}
  \item For each $i \in \sO(q)$, an oplax transformation $\la_i\cn\De
    F(\om_i) \to \si$.  That is, for each $j \in \sO(p)$, a 1-cell
    \[
    \la_{i,j}\cn (\De F(\om_i))(j) = F(\om_i) \to \si(j) = \si_j
    \]
    and for $j < j'$, a 2-cell $\la_{i,(j,j')}$ as below.
    \[
    \begin{tikzpicture}[x=14mm,y=16mm]
      \draw[tikzob,mm] 
      (0,1) node (omi) {F(\om_i)} 
      (-1,0) node (sj) {\si_j} 
      (1,0) node (sj') {\si_{j'}}; 
      \path[tikzar,mm] 
      (sj) edge[swap] node {\si_{(j,j')}} (sj') 
      (omi) edge[swap] node {\la_{i,j}} (sj) 
      (omi) edge node {\la_{i,j'}} (sj'); 
      \draw[tikzob,mm]
      (-.2,.45) node[rotate=215,2label={above,}] {\Rightarrow} 
        node[below right] {\la_{i,(j,j')}};
    \end{tikzpicture}
    \]
    The condition on $\la_i$ being an oplax transformation means for
    all $j<j'<j''\in [p]$ the following equality of pasting diagrams
    holds.
    \[
    \begin{tikzpicture}[x=15mm,y=15mm,baseline=(-60:1)]
    \draw[tikzob,mm] 
    (0,0) node (a) {F(\om_i)}
    (-90:.55) node[rotate=-60,font=\small] {\Downarrow}
    (-120:1) node (b) {\si_j}
    (-60:1) node (c) {\si_{j'}}
    ++(90:.55) node[font=\small] {\Downarrow}
    (0:1) node (d) {\si_{j''}};
    \path[tikzar,mm] 
    (a) edge (b)
    (b) edge (c)
    (c) edge (d)
    (a) edge (d)
    (a) edge (c);    
    \end{tikzpicture}
    =
    \begin{tikzpicture}[x=15mm,y=15mm,baseline=(-60:1)]
    \draw[tikzob,mm] 
    (0,0) node (a) {F(\om_i)}
    (-60:1) node (b) {\si_j}
    +(90:.55) node[font=\small] {\Downarrow}
    ++(0:1) node (c) {\si_{j'}}
    (0:1) node (d) {\si_{j''}}
    ++(-90:.55) node[rotate=60,font=\small] {\Downarrow};
    \path[tikzar,mm] 
    (a) edge (b)
    (b) edge (c)
    (c) edge (d)
    (a) edge (d)
    (b) edge (d);
    \end{tikzpicture}
    \]
  \item For each $i < i'$ in $\sO(q)$, a modification $\la_{(i,i')} \cn \la_i
    \Rightarrow \la_{i'} \circ \De F(\om_{(i,i')})$.  That is, for each $j \in
    \sO(p)$, a 2-cell $\la_{(i,i'),j}$ as below.    
    \[
    \begin{tikzpicture}[x=14mm,y=16mm]
      \draw[tikzob,mm] 
      (0,1) node (omi) {F(\om_i)} 
      (-1,0) node (sj) {F(\om_{i'})} 
      (1,0) node (sj') {\si_{j}}; 
      \path[tikzar,mm] 
      (sj) edge[swap] node {\la_{(i',j)}} (sj') 
      (omi) edge[swap] node {F(\om_{(i,i')})} (sj) 
      (omi) edge node {\la_{i,j}} (sj'); 
      \draw[tikzob,mm]
      (-.2,.45) node[rotate=215,2label={above,}] {\Rightarrow} 
        node[below right] {\la_{(i,i'),j}};
    \end{tikzpicture}
    \]
  To say that $\la_{(i,i')}$ is a modification means that for
  all $j,j'\in [p]$, we have the equality of pastings as below.
  \[
  \begin{tikzpicture}[x=15mm,y=15mm,baseline=(-60:1)]
    \draw[tikzob,mm] 
    (0,0) node (a) {F(\om_i)}
    (-90:.55) node[rotate=-60,font=\small] {\Downarrow}
    (-120:1) node (b) {F(\om_{i'})}
    (-60:1) node (c) {\si_j}
    ++(90:.55) node[font=\small] {\Downarrow}
    (0:1) node (d) {\si_{j'}};
    \path[tikzar,mm] 
    (a) edge (b)
    (b) edge (c)
    (c) edge (d)
    (a) edge (d)
    (a) edge (c);    
  \end{tikzpicture}
  =
  \begin{tikzpicture}[x=15mm,y=15mm,baseline=(-60:1)]
    \draw[tikzob,mm] 
    (0,0) node (a) {F(\om_i)}
    (-60:1) node (b) {F(\om_{i'})}
    +(90:.55) node[font=\small] {\Downarrow}
    ++(0:1) node (c) {\si_j}
    (0:1) node (d) {\si_{j'}}
    ++(-90:.55) node[rotate=60,font=\small] {\Downarrow};
    \path[tikzar,mm] 
    (a) edge (b)
    (b) edge (c)
    (c) edge (d)
    (a) edge (d)
    (b) edge (d);
  \end{tikzpicture}
  \]
  \end{itemize}
  The condition on $\la$ being a lax transformation means there is a
  compatibility of the modifications $\la_{(i,i')}$ with composition
  in $\sO(q)$: for all $i<i'<i''\in[q]$ and $j\in [p]$, we
  have the following equality of pasting diagrams.
  \[
  \begin{tikzpicture}[x=15mm,y=15mm,baseline=(-60:1)]
    \draw[tikzob,mm] 
    (0,0) node (a) {F(\om_i)}
    (-90:.55) node[rotate=-60,font=\small] {\Downarrow}
    (-120:1) node (b) {F(\om_{i'})}
    (-60:1) node (c) {F(\om_{i''})}
    ++(90:.55) node[font=\small] {\Downarrow}
    (0:1) node (d) {\si_{j}};
    \path[tikzar,mm] 
    (a) edge (b)
    (b) edge (c)
    (c) edge (d)
    (a) edge (d)
    (a) edge (c);    
  \end{tikzpicture}
  =
  \begin{tikzpicture}[x=15mm,y=15mm,baseline=(-60:1)]
    \draw[tikzob,mm] 
    (0,0) node (a) {F(\om_i)}
    (-60:1) node (b) {F(\om_{i'})}
    +(90:.55) node[font=\small] {\Downarrow}
    ++(0:1) node (c) {F(\om_{i''})}
    (0:1) node (d) {\si_{j}}
    ++(-90:.55) node[rotate=60,font=\small] {\Downarrow};
    \path[tikzar,mm] 
    (a) edge (b)
    (b) edge (c)
    (c) edge (d)
    (a) edge (d)
    (b) edge (d);
  \end{tikzpicture}
  \]
  
  With this description, one can verify at once that defining 
  \begin{itemize}
  \item $\de_{(i,q+1+j)}=\la_{i,j}$, 
  \item $\de_{(i,q+1+j,q+1+j')}=\la_{i,(j,j')}$, and
  \item $\de_{(i,i',q+1+j)}=\la_{(i,i'),j}$
  \end{itemize}
  gives precisely the data of a 2-functor $\de\cn\sO(q + p + 1) \to D$
  in \cref{eq:om-si-de-orientals} as described above.  Verifying the
  formulas for faces and degeneracies is straightforward.
\end{proof}

\begin{lem}\label{lem:q-filtration}
  Fix $\om \in N_qC$ and consider the oplax comma object 
   in the square below.
  \[
  \begin{tikzpicture}[x=26mm,y=16mm]
    \draw[tikzob,mm] 
    (0,1) node (Op) {\oprdar{(F\circ\om) , D}}
    (1,1) node (*) {*}
    (1,0) node (DOq) {\lax(\sO(q), D)}
    (0,0) node (D) {D};
    \path[tikzar,mm] 
    (Op) edge node {} (*)
    (D) edge[swap] node {\De \Id_{D}} (DOq)
    (*) edge node {\nh{F\circ\om}} (DOq)
    (Op) edge[swap] node {} (D);
    \draw[tikzob,mm]
    (.5,.5) node[rotate=225,2label={above,}] {\Rightarrow}
    ++(3mm,-2mm) node {};
  \end{tikzpicture}
  \] 
  Then the set of pairs $(\de,\si)$ such that $(\om, \de,
  \si) \in \bsimp_{p,q}$ is in bijective correspondence with
  $N_p(\oprdar{(F\circ\om) , D})$, the set of $p$-simplices in the
  nerve of $\oprdar{(F\circ\om) , D}$. Under this bijection, $d_i^H(\om, -, -)$ corresponds to $d_i$, and similarly for
  degeneracies.
\end{lem}

\subsection{Analysis of the homology spectral sequence}\label{sec:ss_from_E0}

We now describe the spectral sequence in homology associated
to a 2-functor $F\cn C \to D$, arising from the simplicial set
$\bsimp$ of \cref{defn:E0}.  In \cref{thm:ss} we show that if $F$ is
an opfibration then the $E^2$ page is given by the homology of $D$
with local coefficients in the homology of the fibers of $F$.  We
apply this machinery in \cref{sec:application} to prove
\cref{thm:main}. We begin with a discussion of local coefficients.

\begin{defn}\label{defn:cat-of-elts}
  Let $N$ be a simplicial set.  Its \emph{category of elements} $\sint
  N$ has objects given by pairs $([p], x)$ where $x \in N_p$.  A
  morphism $([p], x) \to ([q], y)$ consists of $\phi\cn [q] \to [p]$
  in $\De$ such that $\phi^*(x) = y$.
\end{defn}
\begin{defn}\label{defn:loc-coeff-on-sset}
  A \emph{local coefficient system} on a simplicial set $N$ is a
  functor
  \[
  F\cn \sint N \to \Ab.
  \]
  We say that $F$ is \emph{morphism-inverting} if it sends all morphisms
  to isomorphisms.  A \emph{local coefficient system} on a topological
  space $X$ is a functor
  \[
  \Pi_1 X \to \Ab
  \]
  where $\Pi_1 X$ is the fundamental groupoid of $X$.
\end{defn}



\begin{rmk}
  A local coefficient system on a simplicial set $N$ does not
  necessarily induce a local coefficient system on its geometric
  realization $|N|$.  However, since the fundamental groupoid of $|N|$
  is equivalent to the free groupoid on $\sint N$, local coefficient
  systems on $|N|$ correspond to morphism-inverting local coefficient
  systems on $N$.  See \cite[Section VI.4]{GJ10Simplicial}.  This
  implies the following result.
\end{rmk}

\begin{prop}\label{lem:morphism-inverting-coeff}
  If $F\cn \sint N \to \Ab$ is morphism-inverting, then $H_*(N \coef
  F)$ is isomorphic to $H_*(|N| \coef F)$, the homology of the space
  $|N|$ with local coefficients given by the functor $\Pi_1(|N|) \to \Ab$
  induced by $F$.
\end{prop}

\begin{defn}\label{defn:homology-for-2-cat}
  Let $D$ be a 2-category, and $F\cn \sint (ND) \to \Ab$ a local
  coefficient system on $ND$.  Then $H_*(D \coef F)$ denotes the homology of the
  simplicial set $ND$ with local coefficients in $F$.  If $F$ is omitted, we implicitly take the constant coefficient system at the integers $\bZ$.
\end{defn}

\begin{rmk}
  The standard definition of a local coefficient system on a
  1-category $C$ is a functor $F \cn C \to \Ab$.  The theory of
  covering spaces and covering groupoids implies that there is a 1-1
  correspondence between morphism-inverting local coefficient systems
  on $C$ in this sense and local coefficient systems on the
  classifying space $BC$ in the sense of
  \cref{defn:loc-coeff-on-sset}.  This is discussed in
  \cite[Definition IV.3.5.1]{Wei13KBook} and \cite[Section
    1]{Qui1973Higher}.  Proofs that the homology of the corresponding
  chain complexes are naturally isomorphic appear in \cite[Section
    1]{Qui1973Higher} and \cite[Theorem VI.4.8]{Whi78Elements}.
  Therefore, in the case that $D$ is a 1-category, our
  \cref{defn:homology-for-2-cat} agrees with the standard definition
  when the coefficients are morphism-inverting.
\end{rmk}

\begin{prop}\label{prop:loc-coeff}
  Let $F\cn C \to D$ be a 2-functor and $q\geq 0$. The assignment 
  \[
  \si \mapsto H_q(\ldar{F , \si}\,)
  \]
  defines a local coefficient system $\sH_q(\ldar{F , \;-})$ on the
  simplicial set $ND$.
\end{prop}
\begin{proof}
  Since $\ldar{F , \si}$ is a lax comma object, the assignment $\si
  \mapsto H_q (\ldar{F , \si})$ is a functor on the category of
  elements of $ND$ (\cref{defn:cat-of-elts}).  Indeed, a morphism
  $([p],\si) \to ([p'],\si')$ in $\sint ND$, given by $\phi\cn [p']
  \to [p]$ in $\De$ with $\si'=\phi^* \si$, yields a lax
  transformation
  \[
  \begin{tikzpicture}[x=26mm,y=20mm]
    \draw[tikzob,mm] 
    (0,1) node (Fsi) {\ldar{F ,  \si}}
    (1,1) node (C) {C}
    (0,0) node (*) {*}
    (1,0) node (DOp) {\oplax(\sO(p), D)}
    (2.5,1) node (C') {C}
    (2.5,0) node (DOp') {\oplax(\sO(p'), D)};
    \path[tikzar,mm] 
    (Fsi) edge node {} (C)
    (*) edge[swap] node {\nh{\si}} (DOp)
    (Fsi) edge node {} (*)
    (C) edge node {} (DOp)
    (C) edge[/tikz/commutative diagrams/equal] node {} (C')
    (DOp) edge[swap] node {\phi^*} (DOp')
    (C') edge node {} (DOp');
    \draw[tikzob,mm]
    (.5,.5) node[rotate=225,2label={above,}] {\Rightarrow};
  \end{tikzpicture}
  \]
  and hence, by the universal property of $\ldar{F,(\phi^*\si)}$, we
  get a 2-functor 
  \[
  \phi^*\cn\ldar{F , \si} \to \ldar{F , (\phi^*
    \si)}.
    \]
  This assignment is clearly functorial.
\end{proof}

Given a bisimplicial set we have a double complex constructed by
taking free abelian groups on simplices, and alternating sums of face
maps.  We then have two spectral sequences (vertical and horizontal)
associated to that double complex. Let $E^0$ denote the double complex
associated to the bisimplicial set $\bsimp$ of \cref{defn:E0}.

\begin{thm}\label{thm:E0-ss}
  The vertical spectral sequence associated to the double complex $E^0$ has
  \begin{equation}\label{eq:ss}
  E^2_{p,q} = H_p\big( D \coef \sH_q(\ldar{F ,\,-}) \big)
  \end{equation}
  and abuts to $H_{p+q} (C)$.
\end{thm}
\begin{proof}
  First consider homology of $E^0$ in the horizontal direction (fix
  $q$, vary $p$). By \cref{lem:q-filtration}, taking homology and
  summing over $\om$ yields
  \[
  \bigoplus_{\om \in N_q(C)} H_p\big(\oprdar{(F \circ \om),D}\big).
  \]
  By \cref{cor:lp-terminal}, for fixed $\om$, $H_p\big(\oprdar{(F
    \circ \om),D}\big) = 0$ for $p>0$ and is isomorphic to $\bZ$ when
  $p=0$.  Therefore, the horizontal spectral sequence of the double
  complex collapses, and the homology of the total complex
  $\mathrm{Tot}\; E^0$ is given by
  \[
  H_q\big( \mathrm{Tot}\; E^0 \big) \iso H_q(C).
  \]
  
  Next consider homology of $E^0$ in the vertical direction (fix $p$,
  vary $q$).  By \cref{lem:p-filtration}, taking homology and summing
  over $\si$ yields
  \begin{align}
    E^1_{p,q} & = \bigoplus_{\si \in N_p(D)} H_q( \ldar{F , \si}\, ).
  \end{align}
  Taking homology of $E^1_{p,q}$ now yields
  \begin{align}
    E^2_{p,q} & = H_p\big( D \coef \sH_q(\ldar{F ,  \;-}) \big).
  \end{align}
  Hence we have a spectral sequence
  \begin{equation}
    E^2_{p,q} = H_p\big( D \coef \sH_q(\ldar{F ,  \;-}) \big)
    \Rightarrow H_{p+q}\big( \mathrm{Tot}\; E^0 \big) \iso H_{p+q} (C).
  \end{equation}
\end{proof}
If $F$ is an opfibration, the following result gives a further
simplification of the $E^2$ page.
\begin{lem}\label{cor:two_isos}
  If $F \cn C \to D$ is an opfibration and all 2-cells of $D$ are
  invertible, then
  \[
  H_q(\ldar{F, \si}) \iso H_q\Big(F^{-1}\big(\si(0)\big)\Big)
  \]
  for all $\si \in N_pD$.
\end{lem}
\begin{proof}
  Since $\sO(q)$ has oplax initial object 0, \cref{lem:lp-initial} gives
  the first isomorphism below.
  Applying \cref{prop:fiber_vs_htpyfiber} gives the second.
  \[
  H_q(\ldar{F , \si}) \iso H_q\big(\laco{F , \si(0)}\big) \iso
  H_q\Big(F^{-1}\big(\si(0)\big)\Big).\qedhere
  \]
\end{proof}

\begin{notn}\label{notn:Finv-loc-coeff}
  If $F\cn C \to D$ is an opfibration and all 2-cells of $D$ are
  invertible, we denote by $\sH_q F^\inv$ the local coefficient system
  given on objects by
  \[\si \mapsto H_q\Big(F^{-1}\big(\si(0)\big)\Big)\]
  and given on morphisms $\phi\cn ([p],\si) \to ([p'],\phi^*\si)$ by 
  composing the morphisms
  \[
  H_q\Big(\ldar{F , \si}\Big) \to H_q\Big(\ldar{F ,  \phi^*\si}\Big)
  \]
  of \cref{prop:loc-coeff} with the isomorphisms of
  \cref{cor:two_isos}.  Explicitly, this means that $\sH_q F^\inv
  (\phi)$ is given by taking nerves and applying $H_q$ to the
  composite below, using the morphisms $i$ and $H$ of
  \cref{prop:fiber_vs_htpyfiber}, and $d$ and $e$ from
  \cref{lem:lp-initial}
  \[
  \begin{array}{ccccccc}
  F^\inv\big(\si(0)\big) &
  \fto{i} &\laco{F,\si(0)}&
  \fto{d} &\ldar{F,\si}&
  \fto{\phi^*} &\ldar{F,\phi^*\si}\\
 &&& \fto{e} &\laco{F,\phi^*\si(0)}&
  \fto{H} &F^\inv\big(\phi^*\si(0)\big).
  \end{array}
  \]
\end{notn}

\begin{rmk}\label{rmk:Finv-base-change}
  Tracing through the constructions, one can check that the composite
  \[
  \laco{F,\si(0)}
  \fto{d} \ldar{F,\si}
  \fto{\phi^*} \ldar{F,\phi^*\si}
  \fto{e} \laco{F,\phi^*\si(0)} 
  \]
  is precisely the base change induced by the 1-cell
  $\si_{(0,\phi(0))}\cn \si(0) \to \phi^*\si(0)$.
\end{rmk}

Combining \cref{thm:E0-ss} with \cref{cor:two_isos} we have the following.
\begin{thm}\label{thm:ss}
  If $F \cn C \to D$ is an opfibration and all 2-cells of $D$ are
  invertible, then there is a spectral sequence
  \begin{equation}
    E^2_{p,q} = H_p\big( D \coef \sH_q F^\inv \big)
  \end{equation}
  that abuts to $H_{p+q} (C)$.
\end{thm}

\section{Construction of \texorpdfstring{$S^\inv S$}{S-1S}}\label{sec:SinvS}

In this section we construct a 2-categorical model for group
completion that generalizes Quillen's $S^\inv S$ \cite{Gra1976Higher}.  Here $S$ is assumed
to be both a 2-groupoid and a permutative Gray monoid with faithful
translations; these terms are defined in \cref{sec:pgms}.  In
\cref{sec:SinvS-construction} we define $S^\inv X$ when $S$ acts on a
2-category $X$.  In \cref{sec:application} we prove, under the
additional hypothesis that the 2-cells of $X$ are invertible, that the
map
\[
\rho\cn S^\inv X \to S^\inv *
\]
 is an opfibration
(\cref{prop:pi_opfib}).   We then use the homology spectral sequence
for $\rho$ (\cref{thm:ss}) to show that it induces localization by $\pi_0S$ on homology (\cref{thm:main_technical}).  In the case $X =
S$, this proves \cref{thm:main}~\eqref{main:ii}: $S^\inv S$ models the
group-completion of $S$.

\subsection{Background on permutative Gray monoids}\label{sec:pgms}

In this section we recall from
\cite{GJO2017KTheory,GJO2019dimensional} the notion of permutative
Gray monoid, a semi-strict type of symmetric monoidal bicategory. We
refer the reader to \cite{Gra74Formal,Gurski13Coherence} for further background on
the Gray tensor product and to \cite{GJO2017KTheory,GJOS2017Postnikov}
for further background on permutative Gray monoids.  Just as
permutative categories provide a strict model for symmetric monoidal
categories, permutative Gray monoids provide a strict model for
symmetric monoidal bicategories.  We sketch the relevant definitions
and then state the strictification result.

\begin{defn}\label{defn:graytensor}
  Let $X$ and $Y$ be 2-categories.  The \emph{Gray tensor product} of $X$
  and $Y$, written $X \otimes Y$, is the 2-category given by
  \begin{itemize}
  \item 0-cells consisting of pairs $x \otimes y$ with $x$ an object
    of $X$ and $y$ an object of $Y$;
  \item 1-cells generated under composition by two kinds of basic 1-cells denoted $f \otimes 1\cn x \otimes y \to x' \otimes y$ for $f\cn x \to x'$
    in $X$ and $1 \otimes g\cn x \otimes y \to x \otimes y'$ for
    $g\cn y \to y'$ in $Y$; and
  \item 2-cells generated by basic 2-cells of the form
    $\al \otimes 1$ for 2-cells $\al$ in $X$; $1 \otimes \de$ for
    2-cells $\de$ in $Y$; and new interchange isomorphism 2-cells
    $\Si_{f,g}\cn (f \otimes 1)(1 \otimes g) \cong (1 \otimes g)(f
    \otimes 1)$.
  \end{itemize}
  These cells satisfy axioms related to composition, naturality and
  bilinearity; for a complete list, see
  \cite[Section 3.1]{Gurski13Coherence} or
  \cite[Definition 3.16]{GJO2017KTheory}.
\end{defn}
\begin{rmk}
  The definition given above is sometimes called the \emph{pseudo}
  Gray tensor product.  The definition given by Gray
  \cite{Gra74Formal} does not require that the 2-cells $\Si_{f,g}$ be
  isomorphisms.  Our definition follows that of \cite{GPS95Coherence},
  where the pseudo version, defined here, is shown to be a monoidal
  product for $\IICat$ and is essential to the coherence theory of
  tricategories.  See \cite{GPS95Coherence,Gurski13Coherence} for
  further details.
\end{rmk}

\begin{thm}[{\cite[Section 4.8]{GPS95Coherence}, \cite[Theorem 3.16]{Gurski13Coherence}}]
  The assignment
  \[(X,Y)\mapsto X \otimes Y
  \]
  extends to a functor
  of categories
  \[
  \IICat \times \IICat \rtarr \IICat
  \]
  which defines a symmetric monoidal structure on $\IICat$. The unit for
  this monoidal structure is the terminal 2-category.
\end{thm}

\begin{defn}
  \label{defn:gray-monoid}
  A \emph{Gray monoid} is a monoid object in $(\IICat,\otimes)$.  This
  consists of a 2-category $S$, a 2-functor
  \[
  \oplus \cn S \otimes S \to S,
  \]
  and an object $e$ of $S$ satisfying associativity and unit
  axioms.
\end{defn}

\begin{defn}\label{defn:pgm}
  A \textit{permutative Gray monoid} $S$ consists of a Gray monoid
  $(S, \oplus, e)$ together with a 2-natural isomorphism,
  \[
    \begin{tikzpicture}[x=1mm,y=1mm]
    \draw[tikzob,mm] 
    (0,0) node (00) {S \otimes S}
    (25,0) node (10) {S \otimes S}
    (12.5,-10) node (01) {S}
    ;
    \path[tikzar,mm] 
    (00) edge node {\tau} (10)
    (10) edge node {\oplus} (01)
    (00) edge[swap] node {\oplus} (01)
    ;
    \draw[tikzob,mm]
    (12.5,-4) node {\Anglearrow{40} \beta}
    ;
  \end{tikzpicture}
  \]
  where $\tau \cn S \otimes S \to S \otimes S$ is the symmetry
  isomorphism in $\IICat$ for the Gray tensor product, such that the
  following axioms hold.
  \begin{itemize}
  \item The following pasting diagram is equal to the identity
    2-natural transformation for the 2-functor $\oplus$.
    \[
    \begin{tikzpicture}[x=1mm,y=1mm]
    \draw[tikzob,mm] 
    (0,0) node (00) {S \otimes S}
    (25,0) node (10) {S \otimes S}
    (50,0) node (20) {S \otimes S}
    (25,-15) node (11) {S}
    ;
    \path[tikzar,mm] 
    (00) edge node{\tau} (10)
    (10) edge node{\tau} (20)
    (00) edge[swap] node{\oplus} (11)
    (10) edge[swap] node{\oplus} (11)
    (20) edge node{\oplus} (11)
    (00) edge[bend left] node{\id} (20)
    ;
    \draw[tikzob,mm] 
    (14.5,-4) node {\scriptstyle \Anglearrow{40} \beta}
    (35.5,-4) node {\scriptstyle \Rightarrow \beta}    
    ;
    \end{tikzpicture}
    \]

  \item The following equality of pasting diagrams holds where we have
    abbreviated the tensor product to concatenation when labeling 1-
    or 2-cells.
    \[
    \begin{tikzpicture}[x=.95mm,y=1mm]
    \draw[tikzob,mm] 
    (3,-10) node (00) {S^{\otimes 3}}
    (18,0) node (10) {S^{\otimes 3}}
    (36,0) node (20) {S^{\otimes 3}}
    (51,-10) node (30) {S^{\otimes 2}}
    (27,-15) node (11) {S^{\otimes 2}}
    (18,-30) node (12) {S^{\otimes 2}}
    (36,-30) node (33) {S}
    (69,-10) node (40) {S^{\otimes 3}}
    (84,0) node (50) {S^{\otimes 3}}
    (102,0) node (60) {S^{\otimes 3}}
    (117,-10) node (70) {S^{\otimes 2}}
    (84,-30) node (52) {S^{\otimes 2}}
    (102,-30) node (73) {S}
    (102,-17) node (63) {S^{\otimes 2}}
    ;
    \path[tikzar,mm] 
    (00) edge node{\scriptstyle \tau \id} (10)
    (40) edge node{\scriptstyle \tau \id} (50)
    (10) edge node{\scriptstyle \id \tau } (20)
    (50) edge node{\scriptstyle \id \tau } (60)
    (20) edge node{\scriptstyle \oplus \id} (30)
    (60) edge node{\scriptstyle \oplus \id} (70)
    (30) edge node{\scriptstyle \oplus} (33)
    (70) edge node{\scriptstyle \oplus} (73)
    (00) edge[swap] node{\scriptstyle \oplus \id} (12)
    (40) edge[swap] node{\scriptstyle \oplus \id} (52)
    (12) edge[swap] node{\scriptstyle \oplus} (33)
    (52) edge[swap] node{\scriptstyle \oplus} (73)
    (00) edge[swap] node{\scriptstyle \id \oplus} (11)
    (11) edge node{\scriptstyle \tau} (30)
    (11) edge[swap] node{\scriptstyle \oplus} (33)
    (50) edge node{\scriptstyle \oplus \id} (52)
    (50) edge[swap] node{\scriptstyle \id \oplus} (63)
    (60) edge node{\scriptstyle \id \oplus} (63)
    (63) edge node{\scriptstyle \oplus} (73)
    ;
    \draw[tikzob,mm] 
    (27,-7.5) node {=}
    (19,-21) node {=}
    (108,-12) node {=}
    (93,-20) node {=}
    (59,-15) node {=}
    (37,-19) node {\scriptstyle \Anglearrow{40} \beta}
    (78.4,-12.5) node {\scriptstyle \Anglearrow{40} \beta \id}
    (96,-5) node {\scriptstyle \Anglearrow{40} \id \beta}
    ;
    \end{tikzpicture}
    \]
  \end{itemize}
\end{defn}

\begin{rmk}\label{rmk:sigmas-with-betas}
  If $S$ is a permutative Gray monoid, we abuse notation and let $\Sigma_{f,g}$ denote the image under $\oplus$ of the Gray structure 2-cell $\Sigma_{f,g}$.  The hexagon axiom and the 2-naturality of $\beta$ together imply that
  $\Sigma_{f,g}$ is an identity 2-cell in $S$ whenever $f$ or $g$ is a component of $\beta$
  \cite[Proposition 3.42]{GJO2017KTheory}.
\end{rmk}

\begin{thm}[{\cite[Theorem 2.97]{SP2011Classification}, \cite[Theorem 3.14]{GJO2017KTheory}}]
  \label{thm:sm2cat-equiv-to-pgm}
  Every symmetric monoidal bicategory is equivalent, via a symmetric
  monoidal pseudofunctor, to a permutative Gray monoid.
\end{thm}

\begin{defn}
  \label{defn:strict-functor-gray-mon}
  Let $(S,\oplus,e,\beta)$ and $(S',\oplus',e',\beta')$ be permutative
  Gray monoids.  A \textit{strict functor} is a 2-functor $F:S \to S'$
  of the underlying 2-categories satisfying the following conditions.
  \begin{itemize}
  \item $F(e) = e'$, so that $F$ strictly preserves the unit
    object.
  \item The diagram
    \[
    \begin{tikzpicture}[x=1mm,y=1mm]
    \draw[tikzob,mm] 
    (0,0) node (00) {S \otimes S}
    (30,0) node (10) {S' \otimes S'}
    (0,-15) node (01) {S}
    (30,-15) node (11) {S'}
    ;
    \path[tikzar,mm] 
    (00) edge node{F \otimes F} (10)
    (10) edge node{\oplus'} (11)
    (00) edge[swap] node{\oplus} (01)
    (01) edge[swap] node{F} (11)    
    ;
    \end{tikzpicture}
    \]
    commutes, so that $F$ strictly preserves the sum.
  \item The equation
    \[
    \beta' * (F \otimes F) = F * \beta
    \]
    holds, so that $F$ strictly preserves the symmetry.  This equation
    is equivalent to requiring that
    \[
    \beta'_{Fx,Fy} = F(\beta_{x,y})
    \]
    as 1-cells from $Fx \oplus' Fy = F(x \oplus y)$ to
    $Fy \oplus' Fx = F(y \oplus x)$.
  \end{itemize}
\end{defn}

\begin{notn}\label{notn:PGM}
  The category of permutative Gray monoids, $\PGM$, has objects permutative
  Gray monoids and morphisms the strict functors between them.
\end{notn}

\begin{defn}\label{defn:pgm_action}
Let $(S, \oplus, e, \beta)$ be a permutative Gray monoid, and let $X$
be a 2-category. An \emph{action} of $S$ on $X$ consists of a
2-functor $\mu \cn S \otimes X \to X$ such that
\begin{enumerate}
\item $\mu(e, -)$ is the identity 2-functor on $X$,
\item $\mu \circ (\oplus \otimes 1_X) = \mu \circ (1_S \otimes \mu)$
  as 2-functors $S \otimes S \otimes X \to X$, and
\item\label{item-3} for any 1-cell $f$ in $X$, and $\beta\cn s \oplus t \to t \oplus s$ in $S$, the image under $\mu$ of the 2-cell $\Si_{\be, f} \in S \otimes X$ is an identity 2-cell in $X$.
\end{enumerate}
\end{defn}

\begin{rmk}
  Let $\mu\colon S \otimes X \to X$ be a 2-functor. Given objects in $s,t$ in $S$ and $x$ in $X$, there is a 1-cell $\mu(\be_{s,t}\otimes 1) \cn \mu((s\oplus t)\otimes x) \to \mu((t\oplus s)\otimes x)$.  As $x$ varies over the objects of $X$, these are the components of a pseudonatural isomorphism $\mu((s\oplus t)\otimes -) \Rightarrow \mu((t\oplus s)\otimes -)$.  
  Condition (\eqref{item-3}) of \cref{defn:pgm_action} is equivalent to requiring that these components form a 2-natural transformation.
\end{rmk}

\begin{notn}\label{notn:no_dots}
  For an action $\mu \cn S \otimes X \to X$, we often write the image
  $\mu(s\otimes x)$ as merely $sx$, and similarly for higher cells.  We also
  use juxtaposition for the action of $S$ on itself via
  $\oplus$.  Generalizing \cref{rmk:sigmas-with-betas}, we denote by $\Sigma$ the image under $\mu$ of any 2-cell $\Sigma$ in $S \otimes X$.
\end{notn}

We now turn to the notion of invertible cells in a permutative Gray monoid.

\begin{defn}\label{defn:invertible2}
  Let $(S, \oplus, e)$ be a Gray monoid.
  \begin{enumerate}
  \item A 2-cell of $S$ is invertible if it has an inverse in the
    usual sense.
  \item A 1-cell $f \cn x \to y$ is invertible if there exists a
    1-cell $g \cn y \to x$ together with invertible 2-cells $g\circ f
    \cong \id_{x}$, $f\circ g \cong \id_{y}$. In other words, $f$ is
    invertible if it is an internal equivalence in $S$.
  \item An object $x$ of $S$ is invertible if there exists another
    object $y$ together with invertible 1-cells $x \oplus y \to e$, $y
    \oplus x \to e$.
  \end{enumerate}
  A 2-category satisfying the first and second condition for all 1-
  and 2-cells is a \emph{2-groupoid}, and a Gray monoid (or more
  generally, a monoidal bicategory) satisfying the third condition for
  all objects is called \emph{grouplike}.
\end{defn}

\begin{defn}\label{defn:picard}
  A \emph{Picard 2-category} is a grouplike symmetric monoidal
  2-groupoid.  We say that a Picard 2-category is \emph{strict} if it
  is a permutative Gray monoid.
\end{defn}

We also have a notion of group-completion for monoid-like structures
on spaces.

\begin{defn}\label{defn:gp_comp}
  Let $X$ be a homotopy commutative, homotopy associative $H$-space. A
  \emph{group completion} of $X$ is an $H$-space $Y$, together with an
  $H$-space map $f\cn X\to Y$, such that 
  \begin{itemize}
  \item $\pi_0(f)$ exhibits $\pi_0(Y)$ as the group completion
  of the abelian monoid $\pi_0(X)$; and 
  \item for all commutative rings $k$, the induced map on homology
    \[ H_*f \cn H_*(X;k) \to H_*(Y;k)\] exhibits $H_*(Y;k)$ as the localization
    $\pi_0(X)^\inv H_*(X;k)$.
   \end{itemize}
\end{defn}
\begin{defn}\label{defn:htpy_gp_comp}
  We say that a functor of symmetric monoidal categories or
  2-categories is a \emph{homotopy group completion} if it is a group
  completion, as in \cref{defn:gp_comp}, on classifying spaces.
\end{defn}

\subsection{Definition of \texorpdfstring{$S^\inv X$}{S-1X}}\label{sec:SinvS-construction}

Let $S$ be a permutative Gray monoid, and suppose that $S$ acts on a
2-category $X$, with action denoted by juxtaposition.  There is an
induced diagonal action of $S$ on $S \times X$, and we denote
this with a lower dot as in $s.(a,x) = (sa,sx)$ for $s \in S$ and
$(a,x) \in S \times X$.

\begin{defn}\label{defn:SinvX}
We describe the 0-, 1-, 2-cells of $S^\inv X$ as follows.
\begin{itemize}
\item An object of $S^\inv X$ consists of a pair $(a,x) \in S
  \times X$.
\item A 1-cell $(a,x) \to (b,y)$ is given by a triple $(s, (\al,
  \phi))$ where $s \in S$ and $(\al, \phi)$ is a morphism in $S
  \times X$ from $s.(a,x)$ to $(b,y)$.
\item A 2-cell from $(s,(\al, \phi))$ to $(s', (\al', \phi'))$ is given
  by an equivalence class $\<p, (A,F)\>$ where $p\cn s \to s'$ is a
  1-cell in $S$ and $(A,F)$ is a 2-cell in $S \times X$ as
  below.
  \[
  \begin{tikzpicture}[x=40mm,y=25mm]
    \draw[tikzob,mm] 
    (0,1) node (sax) {s.(a,x)}
    (1,1) node (tax) {s'.(a,x)}
    (.5,0) node (by) {(b,y)};
    \path[tikzar,mm] 
    (sax) edge node {p.(1,1)} (tax)
    (sax) edge[swap] node {(\al,\phi)} (by)
    (tax) edge node {(\al',\phi')} (by);
    \draw[tikzob,mm]
    (.45,.7) node[rotate=35,2label={below,(A,F)}] {\Rightarrow}
    ;
  \end{tikzpicture}
  \]
  Two equivalence classes $\<p, (A,F)\>$ and $\<q, (B,G)\>$ are equal
  if there is a 2-cell isomorphism $\Theta\cn p \iso q$ in $S$ such
  that we have the following equality of pastings in $S \times X$;
  the unmarked 2-cells are given by $(A,F)$ and $(B,G)$, respectively.
  \begin{equation}\label{eq:icc}
  \begin{tikzpicture}[x=38mm,y=25mm,baseline={(0,16mm)}]
    \draw[tikzob,mm] 
    (0,1) node (sax) {s.(a,x)}
    (1,1) node (tax) {s'.(a,x)}
    (.5,0) node (by) {(b,y)};
    \path[tikzar,mm] 
    (sax) edge[bend left=25] node {q.(1,1)} (tax) 
    (sax) edge[bend right=25,swap] node {p.(1,1)} (tax) 
    (sax) edge[swap] node {(\al,\phi)} (by)
    (tax) edge node {(\al',\phi')} (by);
    \draw[tikzob,mm]
    (.5,.45) node[rotate=35,2label={above,}] {\Rightarrow}
    (.4,1) node[rotate=90,2label={below,\Theta.(1,1)}] {\Rightarrow}
    ;
  \end{tikzpicture}
  \ =\ 
  \begin{tikzpicture}[x=35mm,y=25mm,baseline={(0,16mm)}]
    \draw[tikzob,mm] 
    (0,1) node (sax) {s.(a,x)}
    (1,1) node (tax) {s'.(a,x)}
    (.5,0) node (by) {(b,y)};
    \path[tikzar,mm] 
    (sax) edge[bend left=25] node {q.1} (tax) 
    (sax) edge[swap] node {(\al,\phi)} (by)
    (tax) edge node {(\al',\phi')} (by);
    \draw[tikzob,mm]
    (.5,.65) node[rotate=35,2label={above,}] {\Rightarrow};
  \end{tikzpicture}
  \end{equation}
\end{itemize}
\end{defn}

\begin{prop}\label{prop:sinvx_is_2cat}
  The data of $S^\inv X$ given in \cref{defn:SinvX} forms a
  2-category.  This construction is functorial with respect to strict functors of permutative Gray monoids on the first variable and maps that preserve the action strictly on the second variable. There is a 2-functor $i \cn X \to S^\inv X$ given by
  $i(x) = (e,x)$ on objects; it is natural on both variables. 
\end{prop}
\begin{proof}
  We define composition of 1-cells
  \[
  (s, (\al, \phi))\cn (a,x) \to (b,y)
  \mathrm{\quad and \quad}
  (t, (\ga, \psi))\cn (b,y) \to (c,z)
  \]
  by the formula
  \[
  (t, (\ga, \psi)) \circ (s, (\al, \phi)) = (t s, (\ga \circ t\al, \psi \circ t\phi)).
  \]
  Vertical composition of 2-cells is given by pasting their defining triangles
  in $S \times X$. Horizontal composition of 2-cells 
  \[
  \begin{tikzpicture}[x=45mm,y=25mm]
   \draw[tikzob,mm] 
   (0,0) node (ax) {(a,x)}
   (1,0) node (by) {(b,y)}
   (2,0) node (cz) {(c,z)};
   \path[tikzar,mm] 
   (ax) edge[bend left=25] node {(s,(\al,\phi))} (by) 
   (ax) edge[bend right=25,swap] node {(s',(\al',\phi'))} (by) 
   (by) edge[bend left=25] node {(t,(\ga,\psi))} (cz) 
   (by) edge[bend right=25,swap] node {(t',(\ga',\psi'))} (cz) ;
   \draw[tikzob,mm]
   (.4,0) node[rotate=270,2label={above,\:\langle p, (A,F)\rangle}] {\Rightarrow}
   (1.4,0) node[rotate=270,2label={above,\:\langle q, (B,G)\rangle}] {\Rightarrow}
   ;
  \end{tikzpicture}
  \]
  is given by $\<r,(C,H)\> = \<r',(C',H')\>$ where
  \[
  \begin{array}{l}
    r = 
    qs'\circ tp
    \\ 
    C = 
      (\ga'* \Si_{q,\al'}* tpa) \circ (B * tA)
    \\
    H = 
     (\psi'* \Si_{q,\phi'}* tpx) \circ (G * tF)
    \\ 
    r' = 
    t'p\circ qs \\
     C' = 
    (\ga' * t'A*qsa)\circ (\ga'*\Si_{q,\al})\circ (B*t\al) \\
    H' = 
   (\psi' * t'F*qsx) \circ (\psi'*\Si_{q,\phi})\circ (G*t\phi).
  \end{array}
  \]
   Verification of the axioms is routine. The statements about functoriality and naturality of the constructions follow directly from the definitions.
\end{proof}

The special case where $X = *$, the terminal 2-category with the
unique action, will be useful in later sections.
\begin{lem}\label{lem:s-inv-pt-contract}
  If the 2-cells of $S$ are invertible, then the topological space
  $|NS^\inv *|$ is contractible.
\end{lem}
\begin{proof}\label{ex:sinvpt}
  We will produce a lax transformation from the constant 2-functor
  $S^\inv* \to S^\inv*$ at $e$ to the identity functor on $S^\inv*$.
  The result then follows by \cref{thm:trans-htpy}.

  The objects of $S^\inv *$ can be identified with the objects of $S$
  and we omit the coordinate for the cells appearing in $*$.  The
  1-cells of $S^\inv *$ are given by pairs $(s,\al)\cn a \to b$ where
  $s \in S$ and $\al \cn sa \to b$.  The 2-cells of $S^\inv *$ are
  given by equivalence classes $\<p,A\>$ where $A \cn \al \Rightarrow
  \al' \circ (p \oplus 1_a)$.  Two equivalence classes $\<p,A\>$ and
  $\<q,B\>$ are equal if there is a 2-cell $\Th\cn p \iso q$ in $S$
  such that $(1_{\al'} * (\Th \oplus 1)) \circ A = B$.

  For each object $a \in S$, there is a canonical 1-cell $(a,1_a)\cn e
  \to a$ in $S^\inv *$ and for each other 1-cell $(s,\al)\cn e \to a$
  in $S^\inv *$ there is a canonical 2-cell $\<\al,1_{\al}\>\cn(s,\al)
  \Rightarrow (a,1)$.  If the 2-cells in $S$ are invertible, then this
  2-cell is unique since, for any other such 2-cell $\<q,B\>$, we have
  $B \cn \al \Rightarrow q$ invertible by hypothesis, and therefore
  taking $\Theta = B$ gives $\<\al,1_{\al}\> = \<q,B\>$.
  Thus $(a,1)$ is terminal in each hom-category $(S^\inv*)(e,a)$.  The
  1-cells $(a,1)$ therefore assemble to a lax transformation from the
  constant functor $S^\inv* \to S^\inv*$ at $e$ to the identity
  functor on $S^\inv*$.
\end{proof}

Next we discuss the action of $S$ on $S^\inv X$.
\begin{lem}\label{lem:act-on-X}
  The action $S \otimes X \to X$ induces an action of $S$ on $S^\inv
  X$ which we write as
  \[
  \xi \cn S
  \otimes S^\inv X \to S^\inv X.
  \]
\end{lem}
\begin{proof}
  We define $\xi\cn S \otimes S^\inv X \to S^\inv X$ on cells as
  follows; we remind the reader that the symmetry for $S$ is written
  $\beta_{p,q}\cn p q \cong q p$.
  \[
  \begin{array}{rrcl}
    \textrm{$0$-cells:} & s \otimes (a,x) & \mapsto & (a, sx)\\
    \textrm{$1$-cells:} &s \otimes (t,(\al,\phi)) & \mapsto & (t, (\al, s\phi\circ \beta_{t,s} x))\\
     &f \otimes (a,x) & \mapsto & (e, (1_a, fx)) \\
    \textrm{$2$-cells:} &s \otimes \<p, (A, F) \> & \mapsto & \<p, (A, sF*\beta_{t,s}x)\>\\
    & \ga \otimes (a,x) & \mapsto & \< 1_e, (1, \gamma x)\>\\
    & \Si_{f, (t(\al, \phi))} & \mapsto & \<1_e, (1, \beta_{t,s}x * \Si_{f, \phi})\>
  \end{array}
  \]
  All of the axioms (2-functoriality of $\xi$ plus those in
  \cref{defn:pgm_action}) are all straightforward consequences of the
  Gray tensor product, permutative Gray monoid, and action axioms for
  $X$.  Note in particular that condition (\ref{item-3}) of \cref{defn:pgm_action} is essential for the 2-functoriality of $\xi$.
\end{proof}

\begin{defn}\label{defn:S-acts-invertibly}
  We say that $S$ \emph{acts homotopy invertibly} on $X$ or that
  \emph{the action is homotopy invertible} if each 2-functor $\mu(s,
  -) \cn X \to X$ induces a homotopy equivalence on the classifying
  space, $|NX|$ (see \cref{defn:nopl-nerve}).
\end{defn}

\begin{rmk}\label{rmk:htpy_inv}
  Since classifying spaces are CW-complexes, an action is homotopy
  invertible if and only each $\mu(s, -)$ induces a weak homotopy
  equivalence.
\end{rmk}

\begin{prop}\label{prop:action-on-X-invertible}
  The action in \cref{lem:act-on-X} is homotopy invertible.
\end{prop}
\begin{proof}
  We construct a 2-functor $s^\inv \cn S^\inv X \to S^\inv X$ which
  serves as a homotopy inverse to $\xi(s, -)$. The 2-functor $s^\inv$
  is defined on cells below.
  \[
  \begin{array}{rrcl}
    \textrm{$0$-cells:} & (a,x) & \mapsto & (sa, x)\\
    \textrm{$1$-cells:} & (t,(\al,\phi)) & \mapsto & (t, (s\al \circ \beta_{t,s}a, \phi))\\
    \textrm{$2$-cells:} & \<p, (A, F) \> & \mapsto & \<p, (sA*\beta_{t,s}a, F)\>
  \end{array}
  \]
  As in \cref{lem:act-on-X}, the 2-functor axioms are simple to check.
  
  Now we show that $s^\inv$ and $\xi(s, -)$ are homotopy inverses to
  each other. First note that it is a simple matter of applying the
  definitions on cells in \cref{lem:act-on-X} and above to show that
  $\xi(s, -) \circ s^\inv = s^\inv \circ \xi(s, -)$ as 2-functors. In
  particular, we only need to define a single transformation between
  this composite and the identity to produce a homotopy exhibiting
  these 2-functors as homotopy inverse to each other (see
  \cref{thm:trans-htpy}). A pseudonatural transformation $T \cn 1
  \Rightarrow \xi(s,-) \circ s^\inv$ is given by the data below.
  \begin{itemize}
  \item[ ]Component for $0$-cells:
    \[
    T_{(a,x)}  =  (s, (1,1)) \cn (a, x) \to (sa, sx)
    \]
  \item[ ]Pseudonaturality for $1$-cells:
    \[
    T_{(t,(\al,\phi))} = \<\beta_{t,s}, (1,1)\> \cn 
    (t s, (s\al \circ \beta_{t,s}a, s\phi \circ \beta_{t,s}x))
    \Rightarrow
    (s  t, (s\al, s\phi))  \qedhere
    \]
  \end{itemize}
\end{proof}

\begin{rmk}\label{rmk:act-via-diag}
  The composite $\xi(s,-) \circ s^\inv$ discussed in the proof of
  \cref{prop:action-on-X-invertible} is the action of $S$ on $S^\inv
  X$ induced by the diagonal action of $S$ on $S \times X$.
\end{rmk}

There are two conditions we can study which are weaker than all
objects of a permutative Gray monoid being invertible.

\begin{defn}\label{defn:htpy_inv_obj}
  Let $S$ be a permutative Gray monoid. 
  \begin{enumerate}
  \item We say an object $s$ is \emph{homotopy invertible} if the
    2-functor $s \oplus - \cn S \to S$ induces a homotopy equivalence
    on the classifying space.
  \item If every object is homotopy invertible, we say that $S$ is
    \emph{homotopy grouplike}. This is equivalent to the action of $S$
    on itself being homotopy invertible.
  \end{enumerate}
\end{defn}

\begin{defn}\label{defn:faithful_trans}
  Let $S$ be a permutative Gray monoid. We say that $S$ has
  \emph{faithful translations} if, for each $s,x,y\in S$, the functor
  \[
  s \oplus - \cn S(x,y) \to S(s x, s y)
  \]
  is faithful. Explicitly, this means that if $\al, \be \cn f
  \Rightarrow g$ are parallel 2-cells such that $s \al = s \be$, then
  $\al = \be$.
\end{defn}

\begin{thm}\label{prop:S-inv-S-pgm}
  Let $S$ be a permutative Gray monoid.
  \begin{enumerate} 
  \item $S^\inv S$ is a permutative Gray monoid.
  \item The 2-functor $i \cn S \to S^\inv S$ of
    \cref{prop:sinvx_is_2cat} is a strict functor of permutative Gray
    monoids (\cref{defn:strict-functor-gray-mon}).
  \item\label{hty-grouplike} $S^\inv S$ is homotopy grouplike.
  \end{enumerate}
\end{thm}
\begin{proof}
  We begin by defining a 2-functor $S^\inv S \otimes S^\inv S \to
  S^\inv S$ which we also write as $\oplus$ using infix notation. An
  object of $S^\inv S \otimes S^\inv S$ is a pair $\big( (a, x), \,
  (a', x') \big)$, and we define
  \[
  (a,x) \oplus (a', x') = (a  a', x  x').
  \]
  Given $(s, (\al, \phi)) \cn (a,x) \to (b,y)$, we define
  \[
  (s, (\al, \phi)) \oplus (a',x') \cn (a,x) \oplus (a', x') \to (b,y) \oplus (a', x')
  \]
  to be the 1-cell $(s, (\al a', \phi x'))$. The 1-cell $(a',x')
  \oplus (s, (\al, \phi))$ is defined to be
  \[
  (s, (a' \al \circ \be_{s,a'} a, x' \phi \circ \be_{s,x'} x)).
  \]
  The 2-cells in $S^\inv S \otimes S^\inv S$ come in three varieties:
  $\Ga \otimes 1$ and $1\otimes\Ga$ for a 2-cell $\Ga$ in $S^\inv S$,
  and the invertible 2-cells $\Si$ indexed by pairs of 1-cells in
  $S^\inv S$. For a 2-cell $\<p, (A, F) \>$, we define
  \[
  \<p, (A, F) \> \oplus (a', x') = \<p, (A a', F  x')\>.
  \]
  We define
  \[
  (a', x') \oplus \<p, (A, F) \> = \<p, (a'A*\be_{s,a'}a, x'F*\be_{s,x'}x) \>.
  \]
  Let $(s, (\al, \phi)) \cn (a,x) \to (b,y)$ and $(s', (\al',
  \phi'))\cn (a',x') \to (b',y')$ be a pair of 1-cells in $S^\inv
  S$. The image of $\Si$ indexed by this pair is
  \[\<\be_{s,s'}, (\Si^\inv_{\al, \al'}*s\be a', \Si^\inv_{\phi, \phi'} *s\be x')\>.\]
  With the definitions in place, we leave the routine verifications to
  the reader, noting only that all the axioms follow from the
  corresponding axioms for $S$ and the Gray tensor product axioms. The
  symmetry isomorphism \[\be \cn (a,x) \oplus (a',x') \cong (a',x')
  \oplus (a,x)\] is defined to be $(e, (\be_{a,a'}, \be_{x,x'}))$.

  The second claim follows by straightforward application of the
  definitions of the 2-functor $i$ and the symmetric monoidal
  structure on $S^\inv S$.

  For the third part, to show that $(a,x) \oplus - $ is a homotopy
  equivalence upon taking nerves, we need only note that $(x,a) \oplus
  -$ is an inverse up to homotopy. Indeed, there is a pseudonatural transformation
  \[
  \theta\cn \Id \Rightarrow (x,a) \oplus (a,x) \oplus -
  \]
  with component
  \[
  \theta_{(b,y)} = (xa, (1,\beta_{x,a} y)) \cn (b,y) \to (xab,axy)
  \]
  for an object $(b,y)$ and pseudofunctoriality constraint
  \[
  \theta_{(t,(\al,\phi))} = \<\beta_{t,xa}, (1,1) \>
  \]
  for a 1-cell $(t,(\al,\phi))\cn (b,y) \to (b',y')$.
\end{proof}

\begin{rmk}
In the argument for \cref{prop:S-inv-S-pgm}~(\ref{hty-grouplike}), the
component of $\theta$ at $(e,e)$ is a morphism
\[
\theta_{e,e} = (xa, (1, \beta_{x,a}))\cn (e,e) \to (xa,ax)
\]
that we might call $\eta_{a,x}$.  In the case that $S$ is a permutative
1-category, Thomason \cite{Tho80Phony} has noted that the $\eta_{a,x}$ are not
natural with respect to morphisms $(a,x) \to (a',x')$ unless the
symmetry of $S$ is trivial.  A similar statement is true here, with the added caveat that the assignment $(a,x) \mapsto (xa,ax)$ is only a pseudofunctor, and the maps $\eta$ cannot be made the components of any kind of transformation.
\end{rmk} 

We end this section by comparing our construction with Quillen's
classical construction for 1-categories \cite{Gra1976Higher}. For a
category $C$, we write $dC$ for the locally discrete 2-category
obtained by adding identity 2-cells. We note that for a permutative
category $C$, $dC$ is a permutative Gray monoid.

\begin{prop}
  Let $S$ be a permutative category and $X$ be a category equipped with a
  strict action of $S$.  Then there is a bijective on objects
  biequivalence
  \[
  (dS)^\inv (dX) \to d(S^\inv X)
  \]
  where the source is the 2-categorical construction of
  \cref{defn:SinvX} applied to $dS$ and $dX$ and the target is the locally
  discrete 2-category obtained from Quillen's 1-categorical $S^\inv
  X$. This map is compatible with the action of $dS$.
\end{prop}
\begin{proof}
  The 2-cells of $(dS)^\inv (dX)$ are not necessarily all identities,
  but for any pair of 1-cells in $(dS)^\inv (dX)$ there is at most one
  2-cell between them, and one exists if and only if the corresponding
  1-cells in Quillen's $S^\inv X$ are in the same equivalence class.
  Taking the quotient of the 2-category $(dS)^\inv (dX)$ by
  identifying isomorphic 1-cells yields a category which is readily
  seen to be Quillen's $S^\inv X$. By the uniqueness of 2-cells
  between a given pair of parallel 1-cells, this quotienting process
  is then the bijective on objects biequivalence we desire.
\end{proof}

\subsection{The canonical projection is an opfibration}\label{sec:application}

In this section we prove that the canonical projection
\[
\rho\cn S^\inv X \to S^\inv *.
\]
induced by the unique 2-functor $X \to *$ is an opfibration. As in the
proof of \cref{lem:s-inv-pt-contract} we omit the final coordinate
(that of the terminal category) when describing cells in $S^\inv *$.
The following result will be useful for our applications in
\cref{thm:main_technical}.

\begin{lem}\label{lem:iso_in_sinvs}
  A 2-cell $\< p, (A, F) \>$ is an isomorphism in $S^\inv X$ if and
  only if the 1-cell $p$ is an equivalence in $S$, the 2-cell $A$ is
  an isomorphism in $S$, and the 2-cell $F$ is an isomorphism in $X$.
\end{lem}
\begin{proof}
  The identity 2-cell in $S^\inv X$ is $\< 1, (1, 1) \>$, and $\< p,
  (A, F) \> = \< 1, (1, 1) \>$ if there exists a 2-cell isomorphism
  $\Theta \cn p \cong 1$ satisfying the equality in \cref{eq:icc}.
  This implies that $A, F$ are invertible 2-cells. The claim in the
  statement of the lemma follows by considering when a composite $\<
  p', (A', F') \> \circ \< p, (A, F) \>$ is equal to the identity.
\end{proof}

We provide the following lemma to help the reader identify the data of
a lift (in the sense of \cref{defn:opcart} \eqref{it:opcart-1}) for
the special case of the 2-functor $\rho\cn S^\inv X \to S^\inv
*$. This will aid in the proof of \cref{prop:pi_opfib}.

\begin{lem}\label{lem:lift-witness-theta}
  Let $(s,\al,\phy)\cn (a,x) \to (d,z)$ and $(u,\ga, \chi)\cn (a,x)
  \to (b,y)$ be 1-cells in $S^\inv X$ and let $\<p,A\>\cn (t,\beta)
  \circ (s,\al) \iso (u,\ga)$ be a 2-cell isomorphism in $S^\inv *$.
  A 1-cell $(v, \de, \la) \in S^\inv X$ together with 2-cells $\<p_1,
  A_1 \> \in S^\inv *$ and $\<p_2, A_2, F_2 \> \in S^\inv X$ give a
  lift of the triple $\big((u,\ga,\chi),(t,\be),\<p,A\>\big)$ if and
  only if there is $\Theta\cn p \iso p_2 \circ (p_1 s)$ in $S$ such
  that the following equality of pasting diagrams holds in $S$.
  \begin{equation}\label{eq:theta-prism}
    \begin{tikzpicture}[x=22.5mm,y=18mm,vcenter]
      \newcommand{\boundary}{
        \draw[tikzob,mm] 
        (0,0) node (tsa) {tsa}
        (1,1) node (ua) {ua}
        (2,0) node (b) {b}
        (1,-1) node (td) {td}
        (0,1) node (vsa) {vsa}
        ;
        \path[tikzar,mm] 
        (tsa) edge node {p_1sa} (vsa)
        (vsa) edge node {p_2a} (ua)
        (ua) edge node {\ga} (b)
        (tsa) edge[swap] node {t\al} (td)
        (td) edge[swap] node{\beta} (b)
        ;
      }
      \draw (2.4,.25) node[font=\large] {=};
      \begin{scope}
        \boundary
        \path[tikzar,mm]
        (tsa) edge[swap] node {pa} (ua)
        ;
        \draw[2cell]
        (vsa) ++(-45:.5) node[rotate=135, 2label={below,\Theta a}] {\Rightarrow}
        node[between=td and ua at .5, rotate=90, 2label={below,A}] {\Rightarrow}
        ;
      \end{scope}
      
      \begin{scope}[shift={(3,0)}]
        \boundary
        \draw[tikzob,mm]
        (1,0) node (vd) {vd}
        ;
        \path[tikzar,mm]
        (vsa) edge[swap] node {v\al} (vd)
        (td) edge node {p_1d} (vd)
        (vd) edge node {\de} (b)
        ;
        \draw[2cell]
        node[between=tsa and vd at .5, rotate=90, 2label={below,\Sigma}] {\Rightarrow}
        node[between=vd and ua at .5, rotate=90, 2label={below,A_2}] {\Rightarrow}
        (vd) ++(-45:.5) node[rotate=135, 2label={below,A_1}] {\Rightarrow}
        ;
      \end{scope}
    \end{tikzpicture}
  \end{equation}
\end{lem}
\begin{defn}\label{defn:witness}
  In the context of \cref{lem:lift-witness-theta}, we say that $\Theta$
  is the \emph{witness} of this lift.
\end{defn}

\begin{prop}\label{prop:pi_opfib}
  Let $S$ be a permutative Gray monoid acting on a 2-category $X$.
  Assume that
  \begin{itemize}
  \item $S$ has faithful translations,
  \item $S$ has invertible 1- and 2-cells, and
  \item $X$ has invertible 2-cells.
  \end{itemize}
  Then the 2-functor $\rho\cn S^\inv X \to S^\inv *$
  induced by the unique 2-functor $X \to *$ is an opfibration.
\end{prop}
\begin{proof}
  We will show that any 1-cell in $S^\inv X$ which is of the form \[(s, \al, 1)
  \cn (a,x) \to (d, sx)\] is opcartesian with respect to $\rho$; this
  immediately shows that any 1-cell $(s,\al)$ in $ S^\inv *$ has an
  opcartesian lift.  Given $(u, \ga, \phy) \cn (a,x) \to (b,y)$ in
  $S^\inv X$ and $\< p, A \> \cn (t, \be) \circ (s,\al) \cong (u,
  \ga)$ in $S^\inv *$, we can choose the lift in the sense of
  \cref{defn:opcart}~\eqref{it:opcart-1} as follows.
  \begin{itemize}
  \item We require a 1-cell $(v, \de, \la) \cn (d, sx) \to (b,y)$, and
    choose it to have $v = t$, $\de = \be$, $\la = \phy \circ px$.
  \item We require a 2-cell $\< p_1, A_1 \> \cn (t, \be)
    \cong (v, \de)$, and choose it to be the identity 2-cell.
  \item We require a 2-cell $\<p_2, A_2, F_2 \> \cn (v, \de,
    \la) \circ (s, \al, \id) \cong (u, \ga, \phy)$, and choose it to
    have $p_2 = p$, $A_2 = A$, $F_2 = 1$.
  \end{itemize}
  This verifies \cref{defn:opcart}~\eqref{it:opcart-1}.  To verify
  \cref{defn:opcart}~\eqref{it:opcart-2}, suppose we are given the following
  data:
  \begin{itemize}
  \item a 2-cell $\<q, B\>\cn (t,\be) \Rightarrow (t',\be')$ in $S^\inv *$; 
  \item a pair of 1-cells $(u, \ga, \phy)$, $(u', \ga', \phy') \cn (a,x)
    \to (b,y)$ in $S^\inv X$;
  \item a 2-cell $\< r, C, G \> \cn (u, \ga, \phy) \Rightarrow (u',
    \ga', \phy')$ in $S^\inv X$;
  \item a pair of 2-cell isomorphisms 
    \[
    \<p, A \> \cn (t, \be) \circ (s, \al) \cong (u, \ga), \quad \<p', A' \> \cn (t', \be') \circ (s, \al) \cong (u', \ga')
    \]
    in $S^\inv *$;
  \item a lift $(v, \de, \la)$, $\<p_1, A_1 \>$, $\<p_2, A_2, F_2 \>$ of
    $(u, \ga, \phy)$, $(t,\be)$, $\<p, A \>$ with witness (see \cref{defn:witness})
    $\Theta$; and
  \item a lift $(v', \de', \la')$, $\<p_1', A_1' \>$, $\<p_2', A_2',
    F_2' \>$ of $(u', \ga', \phy')$, $(t',\be')$, $\<p', A' \>$ with witness
    $\Theta'$,
  \end{itemize}
  satisfying equation
  \eqref{eq:opcart-diag0}.
  
  We require a unique isomorphism 2-cell $\< \tilde{q}, \tilde{B},
  \tilde{H} \> \cn (v, \de, \la) \Rightarrow (v', \de', \la')$ in
  $S^\inv X$ satisfying the two equations \eqref{eq:opcart-diag1} and
  \eqref{eq:opcart-diag2}.

  Equation \eqref{eq:opcart-diag0} yields the existence of a
  2-cell isomorphism $\Phi \cn r \circ p \cong p' \circ qs$ in $S$ such
  that the equality 
  \begin{equation}\label{eq:1_p3b}
    \begin{tikzpicture}[x=24mm,y=15mm,vcenter]
      \draw[tikzob,mm] 
      (0,0) node (tsa) {tsa}
      (.5,1) node (t'sa) {t'sa}
      (1.5,1) node (u'a) {u'a}
      (2,0) node (b) {b}
      (1,-1) node (td) {td}
      (1,0) node (t'd) {t'd}
      ;
      \path[tikzar,mm] 
      (tsa) edge node {qsa} (t'sa)
      (t'sa) edge node {p'a} (u'a)
      (u'a) edge node {\ga'} (b)
      (tsa) edge[swap] node {t\al} (td)
      (td) edge[swap] node {\be} (b)
      (t'sa) edge[swap] node {t'\al} (t'd)
      (td) edge node {qd} (t'd)
      (t'd) edge[swap] node {\be'} (b)
      ;
      \draw[2cell]
      (.5,0) node[rotate=90, 2label={below,\Si}] {\Rightarrow}
      (1.25, .5) node[rotate=45, 2label={below,A'}] {\Rightarrow}
      (t'd)++(-60:.4) node[rotate=135, 2label={below,B}] {\Rightarrow}
      (-.5,0) node {=}
      ;
      
      \begin{scope}[shift={(-3,0)}]
        \draw[tikzob,mm] 
        (0,0) node (tsa) {tsa}
        (.5,1) node (t'sa) {t'sa}
        (1.5,1) node (u'a) {u'a}
        (2,0) node (b) {b}
        (1,-1) node (td) {td}
        (1,0) node (ua) {ua}
        ;
        \path[tikzar,mm] 
        (tsa) edge node {qsa} (t'sa)
        (t'sa) edge node {p'a} (u'a)
        (u'a) edge node {\ga'} (b)
        (tsa) edge[swap] node {t\al} (td)
        (td) edge[swap] node {\be} (b)
        (tsa) edge[swap] node {pa} (ua)
        (ua) edge node {ra} (u'a)
        (ua) edge[swap] node {\ga} (b)
        ;
        \draw[2cell]
        (.6, .4) node[rotate=135, 2label={below,\Phi\, a}] {\Rightarrow}
        (1.5, .4) node[rotate=90, 2label={below,C}] {\Rightarrow}
        (1, -.4) node[rotate=90, 2label={below,A}] {\Rightarrow}
        ;
      \end{scope}
    \end{tikzpicture}
  \end{equation}
  holds. To check the two required equations of \cref{defn:opcart}
  \eqref{it:opcart-2}, we need the following:
  \begin{itemize}
  \item a 2-cell isomorphism $\Psi \cn p_1' \circ q \cong \tilde{q}
    \circ p_1$ such that the equality
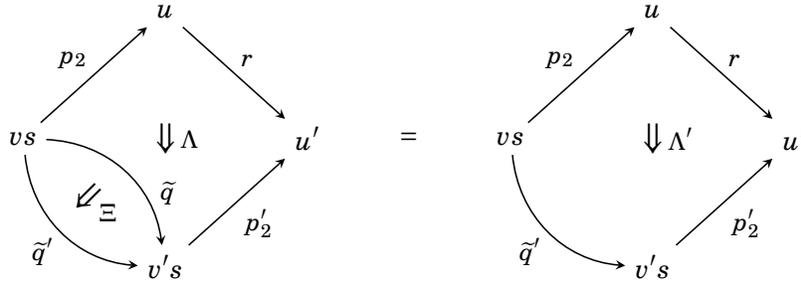
    \begin{equation}\label{eq:2_p3b}
      \begin{tikzpicture}[x=10mm,y=10mm,vcenter]
        \draw[tikzob,mm] 
        (0,0) node (td) {td}
        (1.25,3) node (vd) {vd}
        (3.75,3) node (v'd) {v'd}
        (5,0) node (b) {b}
        (2,1.5) node (t'd) {t'd}
        (1.25, 2) node[rotate=135, 2label={below,\Psi d}] {\Rightarrow}
        (2, .75) node[rotate=90, 2label={below,B}] {\Rightarrow}
        (3.5, 1.5) node[rotate=60, 2label={below,A_1'}] {\Rightarrow}
        (6,1.5) node {=}
        ;
        \path[tikzar,mm] 
        (td) edge node {p_1d} (vd)
        (vd) edge node {\wt{q}d} (v'd)
        (v'd) edge node {\de'} (b)
        (td) edge[swap] node {\be} (b)
        (td) edge[swap] node {qd} (t'd)
        (t'd) edge[swap] node {p_1'd} (v'd)
        (t'd) edge[swap] node {\be'} (b)
        ;
        \begin{scope}[shift={(7,0)}]
          \draw[tikzob,mm] 
          (0,0) node (td) {td}
          (1.25,3) node (vd) {vd}
          (3.75,3) node (v'd) {v'd}
          (5,0) node (b) {b}
          (1.25, 1) node[rotate=90, 2label={below,A_1}] {\Rightarrow}
          (3.5, 2) node[rotate=45, 2label={below,\wt{B}}] {\Rightarrow}
          ;
          \path[tikzar,mm] 
          (td) edge node {p_1d} (vd)
          (vd) edge node {\wt{q}d} (v'd)
          (v'd) edge node {\de'} (b)
          (td) edge[swap] node {\be} (b)
          (vd) edge[swap] node {\de} (b)
          ;
        \end{scope}
      \end{tikzpicture}
    \end{equation}
    holds, and
  \item a 2-cell isomorphism $\La \cn r \circ p_2 \cong p_2' \circ \tilde{q}s$ such that
    \begin{equation}\label{eq:3.1_p3b}
      \begin{tikzpicture}[x=24mm,y=15mm,vcenter]
        \draw[tikzob,mm] 
        (0,0) node (tsa) {vsa}
        (.5,1) node (t'sa) {v'sa}
        (1.5,1) node (u'a) {u'a}
        (2,0) node (b) {b}
        (1,-1) node (td) {vd}
        (1,0) node (t'd) {v'd}
        (.5,0) node[rotate=90, 2label={below,\Si}] {\Rightarrow}
        (1.25, .5) node[rotate=45, 2label={below,A_2'}] {\Rightarrow}
        (t'd)++(-65:.45) node[rotate=135, 2label={below,\wt{B}}] {\Rightarrow}
        (-.5,0) node {=}
        ;
        \path[tikzar,mm] 
        (tsa) edge node {\wt{q}sa} (t'sa)
        (t'sa) edge node {p_2'a} (u'a)
        (u'a) edge node {\ga'} (b)
        (tsa) edge[swap] node {v\al} (td)
        (td) edge[swap] node {\de} (b)
        (t'sa) edge[swap] node {v'\al} (t'd)
        (td) edge node {\wt{q}d} (t'd)
        (t'd) edge[swap] node {\de'} (b)
        ;
        
        \begin{scope}[shift={(-3,0)}]
          \draw[tikzob,mm] 
          (0,0) node (tsa) {vsa}
          (.5,1) node (t'sa) {v'sa}
          (1.5,1) node (u'a) {u'a}
          (2,0) node (b) {b}
          (1,-1) node (td) {vd}
          (1,0) node (ua) {ua}
          (.6, .4) node[rotate=135, 2label={below,\La a}] {\Rightarrow}
          (1.5, .4) node[rotate=90, 2label={below,C}] {\Rightarrow}
          (1, -.4) node[rotate=90, 2label={below,A_2}] {\Rightarrow}
          ;
          \path[tikzar,mm] 
          (tsa) edge node {\wt{q}sa} (t'sa)
          (t'sa) edge node {p_2'a} (u'a)
          (u'a) edge node {\ga'} (b)
          (tsa) edge[swap] node {v\al} (td)
          (td) edge[swap] node {\de} (b)
          (tsa) edge[swap] node {p_2a} (ua)
          (ua) edge node {ra} (u'a)
          (ua) edge[swap] node {\ga} (b)
          ;
        \end{scope}
      \end{tikzpicture}
    \end{equation}
    \begin{equation}\label{eq:3.2_p3b}
      \begin{tikzpicture}[x=10mm,y=10mm]
        \draw[tikzob,mm] 
        (0,0) node (td) {vsx}
        (1.25,3) node (vd) {v'sx}
        (3.75,3) node (v'd) {u'x}
        (5,0) node (b) {y}
        (2,1.5) node (t'd) {ux}
        (t'd) ++(125:.7) node[rotate=135, 2label={below,\La x}] {\Rightarrow}
        (2, .75) node[rotate=90, 2label={below,F_2}] {\Rightarrow}
        (3.5, 1.5) node[rotate=60, 2label={below,G}] {\Rightarrow}
        (6,1.5) node {=}
        ;
        \path[tikzar,mm] 
        (td) edge node {\wt{q}sx} (vd)
        (vd) edge node {p_2'x} (v'd)
        (v'd) edge node {\phy'} (b)
        (td) edge[swap] node {\la} (b)
        (td) edge[swap] node {p_2x} (t'd)
        (t'd) edge[swap] node {rx} (v'd)
        (t'd) edge[swap] node {\phy} (b)
        ;
        \begin{scope}[shift={(7,0)}]
          \draw[tikzob,mm] 
          (0,0) node (td) {vsx}
          (1.25,3) node (vd) {v'sx}
          (3.75,3) node (v'd) {u'x}
          (5,0) node (b) {y}
          (1.25, 1) node[rotate=90, 2label={below,\wt{H}}] {\Rightarrow}
          (3.5, 2) node[rotate=45, 2label={below,F_2'}] {\Rightarrow}
          ;
          \path[tikzar,mm] 
          (td) edge node {\wt{q}sx} (vd)
          (vd) edge node {p_2'x} (v'd)
          (v'd) edge node {\phy'} (b)
          (td) edge[swap] node {\la} (b)
          (vd) edge[swap] node {\la'} (b)
          ;
          
        \end{scope}
      \end{tikzpicture}
    \end{equation}
    both hold.
  \end{itemize}

  Since the 1-cell $p_1 \cn t \to v$ is an equivalence,
  we fix a pseudo-inverse $m \cn v \to t$ to $p_1$
  together with $\eta \cn 1 \cong mp_1$, $\epz \cn p_1 m \cong 1$
  satisfying the triangle identities.
  
  Let $\tilde{q} = p_1' \circ q \circ m$, and let $\tilde{B}$ be given
  by the pasting diagram below.
  \begin{equation}\label{eq:p4}
    \begin{tikzpicture}[x=30mm,y=25mm,vcenter]
    \draw[tikzob,mm] 
    (0,1) node (vd1) {vd}
    (1,1) node (td) {td}
    (2,1) node (t'd) {t'd}
    (3,1) node (v'd) {v'd}
    (2,0) node (b) {b}
    (1,0) node (vd2) {vd}
    (.70,.70) node[rotate=45, 2label={above,\epz^{-1}d}] {\Rightarrow}
    (1.25, .25) node[rotate=45, 2label={below,A_1^{-1}}] {\Rightarrow}
    (1.70, .70) node[rotate=45, 2label={below,B}] {\Rightarrow}
    (2.25,.70) node[2label={below,A_1'}] {\Rightarrow}
    ;
    \path[tikzar,mm] 
    (vd1) edge node {md} (td)
    (td) edge node {qd} (t'd)
    (t'd) edge node {p_1'd} (v'd)
    (v'd) edge node {\de'} (b)
    (vd1) edge[swap] node {1} (vd2)
    (vd2) edge[swap] node {\de} (b)
    (td) edge node {p_1 d} (vd2)
    (td) edge[swap] node {\be} (b)
    (t'd) edge[swap] node {\be'} (b)
    ;
    \end{tikzpicture}
  \end{equation}

  The 2-cell isomorphism $\Psi$ is the whiskering
  \[
  \begin{tikzpicture}[x=20mm,y=20mm,baseline={(0,0.75).base}]
  \draw[tikzob,mm] 
  (0,1) node (t0) {t}
  (1,1) node (v) {v}
  (3,1) node (v') {v'}
  (1.5,0) node (t1) {t}
  (2.5,0) node (t') {t'}
  (.75,.75) node[rotate=45, 2label={below,\eta}] {\Rightarrow}
  ;
  \path[tikzar,mm] 
  (t0) edge node {p_1} (v)
  (t0) edge[swap] node (id) {1} (t1)
  (v) edge node {\wt{q}} (v')
  (v) edge node {m} (t1)
  (t1) edge[swap] node {q} (t')
  (t') edge[swap] node {p'_1} (v')
  ;
  \end{tikzpicture}
  \]
  and equation \eqref{eq:2_p3b} is satisfied by construction. The
  2-cell isomorphism $\La$ is defined to be the following pasting.
  \begin{equation}\label{eq:X}
  \begin{tikzpicture}[x=26mm,y=20mm,baseline={(0,0.75).base}]
  \draw[tikzob,mm] 
  (0,0) node (vs0) {vs}
  (1.5,0) node (u) {u}
  (3,0) node (u') {u'}
  (0,1) node (vs1) {vs}
  (1,1) node (ts) {ts}
  (2,1) node (t's) {t's}
  (3,1) node (v's) {v's}
  ;
  \path[tikzar,mm] 
  (vs0) edge[swap] node[pos=.55] (p2) {p_2} (u)
  (u) edge[swap] node {r} (u')
  (vs1) edge node {ms} (ts)
  (ts) edge node {qs} (t's)
  (t's) edge node {p'_1s} (v's)
  (vs1) edge[swap] node (id) {1} (vs0)
  (ts) edge node[pos=.65] (p1s) {p_1s} (vs0)
  (ts) edge node {p} (u)
  (t's) edge[swap] node (p') {p'} (u')
  (v's) edge node {p'_2} (u')
  ;
  \draw[2cell]
  (vs1) ++(-45:.5) node[rotate=40, 2label={above,\epz^{-1} s}] {\Rightarrow}
  node[between=ts and p2 at .6, rotate=45, 2label={above=.2,\ \Theta^{-1}\hspace{-1em}}] {\Rightarrow}
  node[between=u and t's at .5, rotate=45, 2label={below,\Phi}] {\Rightarrow}
  node[between=p' and v's at .5, rotate=45, 2label={below,\Theta'}] {\Rightarrow}
  ;
  \end{tikzpicture}
  \end{equation}
  Now we verify equation \eqref{eq:3.1_p3b} as follows.  The two diagrams
  below are equal by the definition of $\Theta'$ (see \ref{defn:witness}). 
  {
  \newcommand{\boundaryXcond}{
    \draw[tikzob,mm] 
    (0,0) node (vsa) {vsa}
    (vsa)++(90:1) node (tsa) {tsa}
    (tsa)++(90:1.0) node (t'sa) {t'sa}
    (t'sa)++(90:2.25) node (v'sa) {v'sa}
    (v'sa)++(0:1.25) node (u'a) {u'a}
    (vsa)++(0:2.75) node (vd1) {vd}
    (vd1)++(90:3) node (b) {b}
    ;
    \path[tikzar,mm] 
    (vsa) edge node {msa} (tsa)
    (tsa) edge node {qsa} (t'sa)
    (t'sa) edge node {p'_1sa} (v'sa)
    (v'sa) edge node {p'_2a} (u'a)
    (u'a) edge node (ga') {\ga'} (b)
    (vd1) edge[swap] node {\de} (b)
    ;
  }
  \newcommand{\tdpartXcond}{ 
    \draw[tikzob,mm] 
    (tsa)++(0:1.25) node (td) {td}
    ;
    \path[tikzar,mm] 
    (tsa) edge node {t\al} (td)
    (td) edge[swap] node {\beta} (b)
    (td) edge node {p_1d} (vd1)
    ;
    \draw[2cell]
    (td) ++(.75,.25) node[rotate=135,2label={below,A_1^{-1}}] {\Rightarrow}
    ;
  }
  \newcommand{\thetaprimeXcond}{ 
    \path[tikzar,mm] 
    (t'sa) edge node[pos=.4] {p'a} (u'a)
    ;
    \draw[2cell]
    (v'sa) ++(-45:.75) node[rotate=135,2label={below,\Theta'a}] {\Rightarrow}
    ;
  }
  \newcommand{\tdprimepartXcond}{ 
    \draw[tikzob,mm] 
    (td)++(90:1.0) node (t'd) {t'd}
    (vsa)++(0:1.25) node (vd0) {vd}
    ;
    \path[tikzar,mm] 
    (t'sa) edge node {t'\al} (t'd)
    (t'd) edge node[pos=.6] {\beta'} (b)
    (td) edge node {qd} (t'd)
    (vd0) edge node {md} (td)
    (vsa) edge[swap] node {v\al} (vd0)
    (vd0) edge[swap] node {1} (vd1)
    ;
    \draw[2cell]
    node[between=td and t'sa at .55, rotate=135, 2label={below,\Sigma}] {\Rightarrow}
    node[between=td and vsa at .55, rotate=135, 2label={below,\Sigma}] {\Rightarrow}
    (t'd) ++(-35:.37) node[rotate=135, 2label={below,B}] {\Rightarrow}
    (vd0) ++(50:.6) node[rotate=135, 2label={above left,\raisebox{-12pt}{$\epz^{-1}d$}}] {\Rightarrow}
    ;
  }
  \newcommand{\uapartXcond}{ 
    \draw[tikzob,mm] 
    (t'sa)++(-0:1.25) node (ua) {ua}
    (vsa)++(0:1.25) node (vsa1) {vsa}
    ;
    \path[tikzar,mm] 
    (tsa) edge node[pos=.45] {pa} (ua)
    (ua) edge node {ra} (u'a)
    (ua) edge node[pos=.35] {\ga} (b)
    (tsa) edge node {p_1s\al} (vsa1)
    (vsa) edge[swap] node {1} (vsa1)
    (vsa1) edge[swap] node {v\al} (vd1)
    ;
    \draw[2cell]
    node[between=ua and ga' at .55, rotate=135, 2label={below,C}] {\Rightarrow}
    node[between=ua and t'sa at .5, shift={(0,.25)}, rotate=135, 2label={below,\Phi\, a}] {\Rightarrow}
    (vsa) ++(45:.6) node[rotate=140, 2label={above left,\hspace{-4.5mm}\raisebox{-11pt}{$\epz^{-1}sa$}}] {\Rightarrow}
    ;
  }
  \begin{equation}\label{eq:Xcond01}
    \begin{tikzpicture}[x=16.5mm,y=16.5mm,vcenter]
      \draw (3.4,2.25) node[2label={above,}] {=};
      \begin{scope}[shift={(0,0)}] 
        \boundaryXcond
        \tdpartXcond
        \tdprimepartXcond
        \draw[tikzob,mm] 
        (t'd)++(90:1.25) node (v'd) {v'd}
        ;
        \path[tikzar,mm] 
        (t'd) edge node {p'_1d} (v'd)
        (v'sa) edge node {v'\al} (v'd)
        (v'd) edge node {\de'} (b)
        ;
        \draw[2cell]
        (v'sa) ++(300:1.2) node[shift={(0,-.25)}, rotate=135, 2label={below,\Sigma}] {\Rightarrow}
        node[between=v'd and u'a at .5, rotate=90, 2label={below,A'_2}] {\Rightarrow}
        (v'd) ++(-60:.6) node[rotate=135, 2label={below,A'_1}] {\Rightarrow}
        ;
      \end{scope}
      \begin{scope}[shift={(4,0)}] 
        \boundaryXcond
        \thetaprimeXcond
        \tdpartXcond
        \tdprimepartXcond
        \draw[tikzob,mm] 
        ;
        \path[tikzar,mm] 
        ;
        \draw[2cell]
        node[between=t'd and u'a at .5, rotate=135, 2label={below,A'}] {\Rightarrow}
        ;
      \end{scope}
    \end{tikzpicture}
  \end{equation}
  Next, the right hand diagram above and left hand
  diagram below are equal by definition of $\Phi$ (\cref{eq:1_p3b})
  and the naturality of $\Si$ with respect to 2-cells.  Lastly, the two diagrams below
  are equal by the definition of $\Theta$.
  \begin{equation}\label{eq:Xcond23}
    \begin{tikzpicture}[x=16.5mm,y=16.5mm,vcenter]
      \draw (3.4,2.25) node[2label={above,}] {=};
      \begin{scope}[shift={(0,0)}] 
        \boundaryXcond
        \thetaprimeXcond
        \tdpartXcond
        \uapartXcond
        \draw[tikzob,mm] 
        ;
        \path[tikzar,mm] 
        ;
        \draw[2cell]
        (td) ++(110:.5) node[rotate=135, 2label={below,A}] {\Rightarrow}
        (vsa1) ++(45:.5) node[rotate=90, 2label={below,\Sigma}] {\Rightarrow}
        ;
      \end{scope}
      \begin{scope}[shift={(4,0)}] 
        \boundaryXcond
        \thetaprimeXcond
        \uapartXcond
        \draw[tikzob,mm] 
        ;
        \path[tikzar,mm] 
        (vsa1) edge[swap] node {p_2a} (ua)
        ;
        \draw[2cell]
        (tsa) ++(0:.65) node[rotate=135, 2label={below,\Theta^{-1}a}] {\Rightarrow}
        node[between=vsa1 and b at .5, rotate=135, 2label={below,A_2}] {\Rightarrow}
        ;
      \end{scope}
    \end{tikzpicture}
  \end{equation}
  }

  Using this definition of $\La$ and the fact that $F_2'$ is
  invertible, we get a unique choice of $\tilde{H}$, concluding the
  construction of the 2-cell $\< \tilde{q}, \tilde{B}, \tilde{H} \>$.
  To show uniqueness, suppose that $\< \wt{q},\wt{B},\wt{H}\>$ and $\<
  \wt{q}',\wt{B}',\wt{H}'\>$ are two lifts, with their corresponding
  2-isomorphisms $\Psi$, $\La$, $\Psi'$ and $\La'$. The 2-cell
  \[
  \begin{tikzpicture}[x=26mm,y=20mm,baseline={(0,0.75).base}]
    \draw[tikzob,mm] 
    (0,0) node (v) {v}
    (1,0) node (t) {t}
    (2,0) node (t') {t'}
    (3,0) node (v') {v'}
    (1.5,1) node (v1) {v}
    (1.5,-1) node (v2) {v}
    ;
    \path[tikzar,mm] 
    (v) edge[swap] node[pos=.45] {m} (t)
    (t) edge[swap] node {q} (t')
    (t') edge[swap] node {p'_1} (v')
    (v) edge node (id) {1} (v1)
    (v) edge[swap] node (id') {1} (v2)
    (t) edge[swap] node {p_1} (v1)
    (t) edge node {p_1} (v2)
    (v1) edge node {\wt{q}} (v')
    (v2) edge[swap] node {\wt{q}'} (v')
    ;
    \draw[2cell]
    node[between=id and t at .6, rotate=315, 2label={above,\raisebox{5pt}{$\epz^{-1}$}}] {\Rightarrow}
    node[between=t and id' at .4, rotate=225, 2label={above,\raisebox{-6pt}{$\epz$}}] {\Rightarrow}
    node[between=v1 and t' at .6, shift={(-.1,0)}, rotate=305, 2label={above,\,\Psi^{-1}}] {\Rightarrow}
    node[between=v2 and t' at .6, rotate=235, 2label={above,\,\Psi'}] {\Rightarrow}
    ;
  \end{tikzpicture}
  \]
  defines a 2-isomorphism $\Xi\cn\wt{q}\cong\wt{q}'$, which by
  \eqref{eq:2_p3b} for $\Psi$ and $\Psi'$ satisfies \eqref{eq:icc} for
  $\wt{B}$ and $\wt{B'}$. Using \cref{eq:2_p3b,eq:3.1_p3b}, one can
  show the following equality of 2-cells in $S$.
  \[
  \begin{tikzpicture}[x=24mm,y=15mm,vcenter]
    \draw[tikzob,mm] 
    (0,0) node (vsa) {vsa}
    (.7,-1) node (v'sa) {v'sa}
    (.7,1) node (ua) {ua}
    (1.4,0) node (u'a) {u'a}
    (2,0) node (b) {b}
    (-.5,0) node {=}
    ;
    \path[tikzar,mm] 
    (vsa) edge node {p_2a} (ua)
    (ua) edge node {ra} (u'a)
    (u'a) edge node {\ga '} (b)
    (v'sa) edge[swap] node {p'_2a} (u'a)
    (vsa) edge[bend right=40,swap] node {\wt{q}'a} (v'sa)
    ;
    \draw[2cell]
    node[between=ua and v'sa at .5, rotate=270, 2label={above,\;\La' a}] {\Rightarrow}
    ;
    
    \begin{scope}[shift={(-3,0)}]
    \draw[tikzob,mm] 
    (0,0) node (vsa) {vsa}
    (.7,-1) node (v'sa) {v'sa}
    (.7,1) node (ua) {ua}
    (1.4,0) node (u'a) {u'a}
    (2,0) node (b) {b}
    ;
    \path[tikzar,mm] 
    (vsa) edge node {p_2a} (ua)
    (ua) edge node {ra} (u'a)
    (u'a) edge node {\ga '} (b)
    (vsa) edge[bend left=40] node[pos=.8] {\wt{q}a} (v'sa)
    (v'sa) edge[swap] node {p'_2a} (u'a)
    (vsa) edge[bend right=40,swap] node {\wt{q}'a} (v'sa)
    ;
    \draw[2cell]
    node[between=ua and v'sa at .5, rotate=270, 2label={above,\;\La a}] {\Rightarrow}
    node[between=vsa and v'sa at .45, rotate=225, 2label={above,\;\Xi\thinspace a}] {\Rightarrow}
    ;
    \end{scope}
  \end{tikzpicture}
  \]
  From the assumption of faithful translations and invertibility of
  1-cells in $S$, we have the equality below, which together with
  \eqref{eq:3.2_p3b} implies that $\Xi$ satisfies \eqref{eq:icc} for
  $\wt{H}$ and $\wt{H'}$.
  \[
  \begin{tikzpicture}[x=27mm,y=17mm,vcenter]
    \draw[tikzob,mm] 
    (0,0) node (vsa) {vs}
    (.7,-1) node (v'sa) {v's}
    (.7,1) node (ua) {u}
    (1.4,0) node (u'a) {u'}
    (-.5,0) node {=}
    ;
     \path[tikzar,mm] 
    (vsa) edge node {p_2} (ua)
    (ua) edge node {r} (u'a)
    (v'sa) edge[swap] node {p'_2} (u'a)
    (vsa) edge[bend right=40,swap] node {\wt{q}'} (v'sa)
    ;
    \draw[2cell]
   node[between=ua and v'sa at .5, rotate=270, 2label={above,\;\La' }] {\Rightarrow}
   ;
    
   \begin{scope}[shift={(-2.4,0)}]
     \draw[tikzob,mm] 
     (0,0) node (vsa) {vs}
     (.7,-1) node (v'sa) {v's}
     (.7,1) node (ua) {u}
     (1.4,0) node (u'a) {u'}
     ;
     \path[tikzar,mm] 
     (vsa) edge node {p_2} (ua)
     (ua) edge node {r} (u'a)
     (vsa) edge[bend left=40] node[pos=.8] {\wt{q}} (v'sa)
     (v'sa) edge[swap] node {p'_2} (u'a)
     (vsa) edge[bend right=40,swap] node {\wt{q}'} (v'sa)
     ;
     \draw[2cell]
     node[between=ua and v'sa at .5, rotate=270, 2label={above,\;\La }] {\Rightarrow}
     node[between=vsa and v'sa at .45, rotate=225, 2label={above,\;\Xi }] {\Rightarrow}
     ;
   \end{scope}
  \end{tikzpicture}
  \]

  This shows that $\< \wt{q},\wt{B},\wt{H}\>=\<
  \wt{q}',\wt{B}',\wt{H}'\>$. Thus we have proved that 1-cells of the
  form $(s, \al, \id)$ are opcartesian with respect to $\rho\cn S^\inv
  X \to S^\inv *$ and further that every 1-cell of $S^\inv *$ has an
  opcartesian lift.

  Note that to prove the second and third condition for $\rho$ to be
  an opfibration, it is sufficient to prove that any 2-cell in $S^\inv
  X$ is cartesian. Indeed, the second condition of \cref{defn:opfib}
  states that $\rho$ is locally a fibration, or equivalently that
  every 2-cell with \emph{target} in the image of $\rho$ has a
  cartesian lift. If $(t, \be, \psi) \cn (a,x) \to (b,y)$ is a 1-cell
  in $S^\inv X$ and $\<p, A\>\cn (s, \al) \Rightarrow \rho(t, \be,
  \psi)$ is any 2-cell in $S^\inv *$, then $\<p, A, \id_{\phy}\>\cn
  (s, \al, \phy) \Rightarrow (t, \be, \psi)$ is a cartesian lift where
  $\phy = \psi \circ px$. The third condition of \cref{defn:opfib}
  states that the horizontal composite of cartesian 2-cells is again
  cartesian, and this follows if all 2-cells are cartesian.

  Fix the data of a pair of 
  2-cells in $S^\inv X$ 
  \[
  \<p, A, F\>:(s,\al,\phy) \Rightarrow (t, \be, \psi), \quad \<q,B,G\> \cn (u, \ga, \chi) \Rightarrow (t, \be,
  \psi),\]
  and let $\<r,C\>: (u,\ga)
  \Rightarrow (s,\al)$ be a 2-cell in $S^\inv *$ such that
  \[
  \rho\<q,B,G\> = \rho\<p,A,F\> \circ \<r,C\>.
  \]
  We must show there is a unique lift $\<r',C',H'\>$ of $\<r,C\>$,
  such that the equality $\<q,B,G\> = \<p,A,F\> \circ \<r',C',H'\>$ holds. To prove
  existence, we construct $H$ such that $\< r,C,H\>$ satisfies the
  condition. By assumption, we have a 2-cell isomorphism $\Theta \cn p
  \circ r \cong q$ such that the following pastings are equal.
  
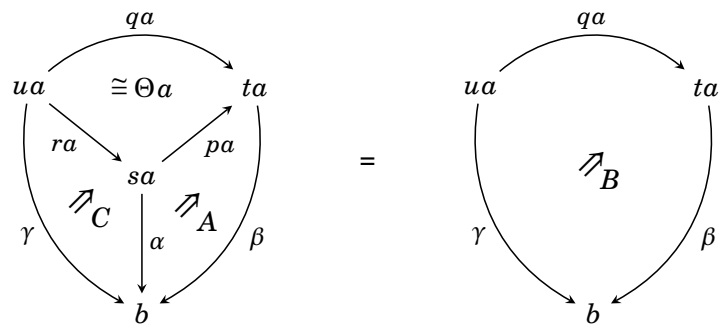
\begin{equation}\label{eq:r_cond1}
    \begin{tikzpicture}[x=30mm,y=20mm]
    \newcommand{\boundary}{
      \draw[tikzob,mm] 
      (0,0) node (ua) {ua}
      (1,0) node (ta) {ta}
      (0.5,-1.5) node (b) {b}
      ;
      \path[tikzar,mm] 
      (ua) edge[bend left=40] node {qa} (ta)
      (ua) edge[bend right=35,swap] node {\ga} (b)
      (ta) edge[bend left=35] node {\be} (b)
      ;
    }
    \boundary
    \draw[tikzob,mm] 
    (0.5,-0.6) node (sa) {sa}
    (sa) ++(0,.6) node {\cong \Theta\, a}
    (sa) ++(210:.3) node[rotate=45, 2label={below,C}] {\Rightarrow}
    (sa) ++(-40:.275) node[rotate=45, 2label={below,A}] {\Rightarrow}
    ;
    \path[tikzar,mm] 
    (ua) edge[swap] node {ra} (sa)
    (sa) edge[swap] node {pa} (ta)
    (sa) edge node {\al} (b)
    ;
    
    \draw (1.5,-.5) node {=};
    
    \begin{scope}[shift={(2,0)}]
      \boundary
      \draw[tikzob,mm] 
      (b) ++(0,1) node[rotate=45, 2label={below,B}] {\Rightarrow}
      ;
    \end{scope}
    \end{tikzpicture}
  \end{equation}
  It suffices to construct a 2-cell $H$ in $X$ such that the equality
  \begin{equation}\label{eq:r_cond2}
    \begin{tikzpicture}[x=30mm,y=20mm]
    \newcommand{\boundary}{
      \draw[tikzob,mm] 
      (0,0) node (ua) {ux}
      (1,0) node (ta) {tx}
      (0.5,-1.5) node (b) {y}
      ;
      \path[tikzar,mm] 
      (ua) edge[bend left=40] node {qx} (ta)
      (ua) edge[bend right=35,swap] node {\chi} (b)
      (ta) edge[bend left=35] node {\psi} (b)
      ;
    }
    \boundary
    \draw[tikzob,mm] 
    (0.5,-0.6) node (sa) {sx}
    (sa) ++(0,.6) node {\cong \Theta x}
    (sa) ++(210:.3) node[rotate=45, 2label={below,H}] {\Rightarrow}
    (sa) ++(-40:.275) node[rotate=45, 2label={below,F}] {\Rightarrow}
    ;
    \path[tikzar,mm] 
    (ua) edge[swap] node {rx} (sa)
    (sa) edge[swap] node {px} (ta)
    (sa) edge node {\phy} (b)
    ;
    
    \draw (1.5,-.5) node {=};
    
    \begin{scope}[shift={(2,0)}]
      \boundary
      \draw[tikzob,mm] 
      (b) ++(0,1) node[rotate=45, 2label={below,G}] {\Rightarrow}
      ;
    \end{scope}
    \end{tikzpicture}
  \end{equation}
  holds, but since 2-cells are invertible in $X$ there is
  always such an $H$. 

  To show uniqueness, we assume there is another lift of the form
  $\<r',C',H'\>$ and we prove that it is equal to the lift we just
  constructed. By assumption, there exists a 2-cell isomorphism
  $\Theta'\cn p\circ r'\cong q$ satisfying the appropriate versions of
  \cref{eq:r_cond1,eq:r_cond2}. Since
  \[\<r,C\>=\rho\<r',C',H'\>=\<r',C'\>,\]
  there exists a 2-isomorphism $\La\cn r \cong r'$ that satisfies
  \cref{eq:icc} with respect to $C$ and $C'$. It remains to show that
  $\La$ satisfies \cref{eq:icc} with respect to $H$ and $H'$, which
  follows from the equality below.
  \[
  \begin{tikzpicture}[x=35mm,y=25mm,baseline={(0,16mm)}]
    \draw[tikzob,mm] 
    (0,1) node (u) {u}
    (1,1) node (s) {s}
    (.5,0) node (t) {t}
    (-.5,.5) node {=}
    ;
    \path[tikzar,mm] 
    (s) edge node {p} (t)
    (u) edge[swap] node {q} (t)
    (u) edge[bend left=25] node (r) {r} (s)
    ;
    \draw[2cell]
    node[between=r and t at .5, rotate=225, 2label={below,\Theta }] {\Rightarrow}
    ;
    
    \begin{scope}[shift={(-2,0)}]
    \draw[tikzob,mm] 
    (0,1) node (u) {u}
    (1,1) node (s) {s}
    (.5,0) node (t) {t}
    ;
    \path[tikzar,mm] 
    (s) edge node (p) {p} (t)
    (u) edge[swap] node (q) {q} (t)
    (u) edge[bend left=25] node (r) {r} (s)
    (u) edge[bend right=25,swap] node {r'} (s)
    ;
    \draw[2cell]
    node[between=p and q at .5, shift={(.05,-.01)}, rotate=225, 2label={below,\Theta' }] {\Rightarrow}
    node[between=u and s at .5, rotate=270, 2label={above,\;\La }] {\Rightarrow}
    ;
    \end{scope}
  \end{tikzpicture}
  \]
  This equality in turn follows from \cref{eq:r_cond1,eq:r_cond2}
  using invertibility of 1- and 2-cells in $S$ and faithful
  translations. Therefore $\< p, A, F \>$ is cartesian.
\end{proof}

\begin{cor}\label{cor:pi_fiber_vs_htpyfiber}
  Let $S$ be a permutative Gray monoid acting on a 2-category $X$. Assume that
  \begin{itemize}
  \item $S$ has faithful translations,
  \item $S$ has invertible 1- and 2-cells, and
  \item $X$ has invertible 2-cells.
  \end{itemize}
  Then for any 2-functor $F \cn Y \to S^\inv *$, we have a homotopy
  equivalence
  \[
  \pb(\rho,F) \to \laco{\rho,F}
  \]
  on classifying spaces. In particular this map induces an isomorphism
  on homotopy and homology groups.
\end{cor}
\begin{proof}
  This follows by combining \cref{prop:fiber_vs_htpyfiber} and \cref{prop:pi_opfib}.
\end{proof}

Before stating the main theorem, we note an explicit description of
the fibers of the projection $\rho$.
\begin{lem}\label{lem:pi_inv_is_X}
  For any $a \in S^\inv *$, there is an isomorphism $\rho^{-1}(a)
  \cong X$ commuting with the action of $S$.
\end{lem} \begin{proof}
  Objects of $\rho^{-1}(a)$ are $(a,x)$ for $x \in X$.  1-cells $(a,x)
  \to (a,x')$ are $(e,(\id_a,\phi))$.  2-cells are $\langle \id_e,
  (\Id_{\id_a}, F)\rangle$.
\end{proof}
\cref{thm:main} is a special case the following, with $X = S$.
\begin{thm}\label{thm:main_technical}
  Let $S$ be a permutative Gray monoid acting on a 2-category $X$.
  Assume that
  \begin{itemize}
  \item $S$ has faithful translations,
  \item $S$ has invertible 1- and 2-cells, and
  \item $X$ has invertible 2-cells.
  \end{itemize}
  Then the inclusion $i\cn X \to S^\inv X$ induces an isomorphism
  \[
  [\pi_0 S]^{-1} H_q(X) \fto{\iso} H_q(S^\inv X).
  \]
  In particular, when $X=S$, we obtain that $i\cn S \to S^\inv S$ is a
  group completion on classifying spaces.
\end{thm}
\begin{proof}
  The projection
  \[
  \rho \cn S^\inv X \to S^\inv *
  \]
  is an opfibration by \cref{prop:pi_opfib}.  Since all the 1- and
  2-cells of $S$ are invertible, \cref{lem:iso_in_sinvs} implies that
  $S^\inv *$ has invertible 2-cells. Thus, we have the (simplicial)
  local coefficient system $\sH_q \rho^\inv$ on $S^\inv *$ of
  \cref{notn:Finv-loc-coeff}, and we consider the spectral sequence
  \[
  E^2_{p,q}=H_p\big( S^\inv * \coef \sH_q\rho^\inv \big) \;\Rightarrow\; H_{p+q}(S^\inv X)
  \]
  of \cref{thm:ss} for $F=\rho$.

  Let $E^0$ be the double complex associated to the bisimplicial set
  $\bsimp$ constructed in \cref{defn:E0}.  Since $S$ is a permutative
  Gray monoid, the set of objects $\ob S$ becomes a monoid by
  restriction. The actions of $S$ on $S^\inv X$ and $S^\inv *$ of
  \cref{lem:act-on-X}, the latter trivial, induce an action of $\ob S$
  on $\bsimp$; explicitly, $s \cdot (\om, \de, \si) = (s \om, s \de, s
  \si)$.  A 1-cell $s \to t$ gives rise to a triple of pseudonatural
  transformations $(s\om,s\de,s\si) \to (t\om,t\de,t\si)$ and hence to
  a chain homotopy between the corresponding chain maps $s \cdot -$
  and $t \cdot -$ on $\mathrm{Tot}$. Therefore $\pi_0 S$ acts on the
  homology of the total complex.

  By \cref{lem:pi_inv_is_X} we have an isomorphism of $\pi_0
  S$-modules
  \[
  \sH_q\rho^\inv(\si)=H_q\Big(\rho^\inv\big(\si(0)\big)\Big) \iso H_q(X)
  \]
  for all $\si$ and all $q$. We can thus define a local coefficient
  system $M_q$ on $S^\inv *$ by setting $M_q(\si)=H_q(X)$ and using
  the isomorphism above to define the action on morphisms.  It is
  important to remark that as a coefficient system, $M_q$ is generally
  not constant. Recall that $\si$ consists of objects $\si_i$ and
  1-cells $\si_{(i,j)}=(s_{(i,j)},\al_{(i,j)})$ in $S^\inv *$, for
  $i<j\in [p]$. Let $\phi \cn [p'] \to [p]$ in $\De$. Using
  \cref{rmk:Finv-base-change} and tracing through the isomorphism in
  \cref{lem:pi_inv_is_X}, one can check that the map
  \[
  M_q(\phi)\cn M_q(\si) \to M_q(\phi^*\si)
  \]
  is given by the action of $s_{0,\phi(0)}$ on $X$.

  Because $S$ acts trivially on $S^\inv *$, the action of $\ob S$ on
  $E^0$ descends to an action of $\pi_0 S$ on $E^2$ and can be
  identified with the action on $M_q = H_q(X)$ induced by the action
  of $S$ on $X$.
  Localizing with respect to the action of $\pi_0 S$ is exact, and thus we have
  \[
  [\pi_0S]^\inv E^2_{p,q} \cong H_p\Big(S^\inv * \coef [\pi_0S]^\inv  M_q  \Big)
  \;\Rightarrow\; [\pi_0S]^\inv H_{p+q}(S^\inv X).
  \]
  The edge homomorphism of this localized spectral sequence is the
  localization of the map induced on homology by $i\cn X \to S^\inv X$
  in \cref{prop:sinvx_is_2cat}.

  Consider the coefficient system $L_q$ given by
  \[
  L_q(\si) =  [\pi_0S]^\inv H_q(X)
  \]
  for all $\si$. Since it is morphism-inverting, by
  \cref{lem:morphism-inverting-coeff} it induces a topological local
  coefficient system on $|NS^\inv *|$, also denoted $L_q$.  Thus we
  have
  \[
  H_p(S^\inv * \coef L_q) \fto{\iso} H_p(|NS^\inv *| \coef L_q)
  \]
  for all $p$ and $q$.  Next we recall that $|NS^\inv *|$ is
  contractible by \cref{lem:s-inv-pt-contract}, and therefore
  $H_p(|NS^\inv *| \coef L_q)$ is 0 for all $p>0$.  Thus the localized
  spectral sequence collapses and the edge map induces an isomorphism
  \[
  [\pi_0 S]^{-1} H_q(X) \fto{\iso} H_q(S^\inv X).\qedhere
  \]
\end{proof}


\providecommand{\bysame}{\leavevmode\hbox to3em{\hrulefill}\thinspace}
\providecommand{\MR}{\relax\ifhmode\unskip\space\fi MR }
\providecommand{\MRhref}[2]{%
  \href{http://www.ams.org/mathscinet-getitem?mr=#1}{#2}
}
\providecommand{\doi}[1]{%
  doi:\href{https://dx.doi.org/#1}{\nolinkurl{#1}}}
\providecommand{\arxiv}[1]{%
  arXiv:\href{https://arxiv.org/abs/#1}{#1}}

\end{document}